\documentclass{jonasart}

\usepackage{mathabx}
\usepackage{stmaryrd}
\usepackage{tikz}
\usepackage[all]{xy}
\usetikzlibrary{arrows}

\DeclareMathOperator\id{Id}
\DeclareMathOperator\im{Im}
\DeclareMathOperator\iso{Iso}
\DeclareMathOperator\Ker{Ker}
\DeclareMathOperator\lin{End}
\DeclareMathOperator\rk{Rank}
\DeclareMathOperator\rt{rt}
\DeclareMathOperator\Up{Up}
\DeclareMathOperator\vect{Vect}
\newcommand{\bfG}{\mathbf{G}}
\newcommand{\bfT}{\mathbf{T}}
\newcommand{\boite}[1]{*+[F]{#1}}
\newcommand{\C}{\mathbb{C}}
\newcommand{\calB}{\mathcal{B}}
\newcommand{\calC}{\mathcal{C}}
\newcommand{\DEPDE}{D_E^{\mathrm{s}}}
\newcommand{\DVPDE}{D_V^{\mathrm{s}}}
\newcommand{\calT}{\mathcal{T}}
\newcommand{\g}{\mathfrak{g}}

\newcommand{\K}{\mathbb{K}}
\newcommand{\N}{\mathbb{N}}
\newcommand{\sym}{\mathfrak{S}}

\newcommand{\tun}{\begin{tikzpicture}[line cap=round,line join=round,>=triangle 45,x=0.5cm,y=0.5cm]
\clip(-0.2,-0.1) rectangle (0.2,0.2);
\begin{scriptsize}
\draw [fill=black] (0.,0.) circle (1pt);
\end{scriptsize}
\end{tikzpicture}}

\newcommand{\tdeux}{\begin{tikzpicture}[line cap=round,line join=round,>=triangle 45,x=0.5cm,y=0.5cm]
\clip(-.2,-.1) rectangle (0.2,0.7);
\draw [line width=.5pt] (0.,0.5)-- (0.,0.);
\begin{scriptsize}
\draw [fill=black] (0.,0.) circle (1pt);
\draw [fill=black] (0.,0.5) circle (1pt);
\end{scriptsize}
\end{tikzpicture}}

\newcommand{\ttroisun}{\begin{tikzpicture}[line cap=round,line join=round,>=triangle 45,x=0.5cm,y=0.5cm]
\clip(-0.5,-0.1) rectangle (0.5,0.7);
\draw [line width=0.5pt] (0.,0.)-- (-0.3,0.5);
\draw [line width=0.5pt] (0.,0.)-- (0.3,0.5);
\begin{scriptsize}
\draw [fill=black] (-0.3,0.5) circle (1pt);
\draw [fill=black] (0.,0.) circle (1pt);
\draw [fill=black] (0.3,0.5) circle (1pt);
\end{scriptsize}
\end{tikzpicture}}
\newcommand{\ttroisdeux}{\begin{tikzpicture}[line cap=round,line join=round,>=triangle 45,x=0.5cm,y=0.5cm]
\clip(-.2,-.1) rectangle (0.2,1.2);
\draw [line width=0.5pt] (0.,0.5)-- (0.,0.);
\draw [line width=0.5pt] (0.,0.5)-- (0.,1.);
\begin{scriptsize}
\draw [fill=black] (0.,0.) circle (1pt);
\draw [fill=black] (0.,0.5) circle (1pt);
\draw [fill=black] (0.,1.) circle (1pt);
\end{scriptsize}
\end{tikzpicture}}

\newcommand{\tquatreun}{\begin{tikzpicture}[line cap=round,line join=round,>=triangle 45,x=0.5cm,y=0.5cm]
\clip(-0.5,-0.1) rectangle (0.5,0.7);
\draw [line width=0.5pt] (0.,0.)-- (-0.3,0.5);
\draw [line width=0.5pt] (0.,0.)-- (0.3,0.5);
\draw [line width=0.5pt] (0.,0.)-- (0.,0.5);
\begin{scriptsize}
\draw [fill=black] (-0.3,0.5) circle (1.0pt);
\draw [fill=black] (0.,0.) circle (1.0pt);
\draw [fill=black] (0.3,0.5) circle (1.0pt);
\draw [fill=black] (0.,0.5) circle (1.0pt);
\end{scriptsize}
\end{tikzpicture}}
\newcommand{\tquatredeux}{\begin{tikzpicture}[line cap=round,line join=round,>=triangle 45,x=0.5cm,y=0.5cm]
\clip(-0.5,-0.1) rectangle (0.5,1.2);
\draw [line width=0.5pt] (0.,0.)-- (-0.3,0.5);
\draw [line width=0.5pt] (0.,0.)-- (0.3,0.5);
\draw [line width=0.5pt] (-0.3,0.5)-- (-0.3,1.);
\begin{scriptsize}
\draw [fill=black] (-0.3,0.5) circle (1.0pt);
\draw [fill=black] (0.,0.) circle (1.0pt);
\draw [fill=black] (0.3,0.5) circle (1.0pt);
\draw [fill=black] (-0.3,1.) circle (1.0pt);
\end{scriptsize}
\end{tikzpicture}}

\newcommand{\tquatrequatre}{\begin{tikzpicture}[line cap=round,line join=round,>=triangle 45,x=0.5cm,y=0.5cm]
\clip(-0.5,-0.1) rectangle (0.5,1.7);
\draw [line width=0.5pt] (0.,0.)-- (0.,0.5);
\draw [line width=0.5pt] (0.,0.5)-- (0.3,1.);
\draw [line width=0.5pt] (0.,0.5)-- (-0.3,1.);
\begin{scriptsize}
\draw [fill=black] (0.,0.) circle (1.0pt);
\draw [fill=black] (0.,0.5) circle (1.0pt);
\draw [fill=black] (-0.3,1.) circle (1.0pt);
\draw [fill=black] (0.3,1.) circle (1.0pt);
\end{scriptsize}
\end{tikzpicture}}
\newcommand{\tquatrecinq}{\begin{tikzpicture}[line cap=round,line join=round,>=triangle 45,x=0.5cm,y=0.5cm]
\clip(-.2,-.1) rectangle (0.2,1.7);
\draw [line width=0.5pt] (0.,0.)-- (0.,0.5);
\draw [line width=0.5pt] (0.,0.5)-- (0.,1.);
\draw [line width=0.5pt] (0.,1.)-- (0.,1.5);
\begin{scriptsize}
\draw [fill=black] (0.,0.) circle (1.0pt);
\draw [fill=black] (0.,0.5) circle (1.0pt);
\draw [fill=black] (0.,1.) circle (1.0pt);
\draw [fill=black] (0.,1.5) circle (1.0pt);
\end{scriptsize}
\end{tikzpicture}}

\newcommand{\tcinqun}{\begin{tikzpicture}[line cap=round,line join=round,>=triangle 45,x=0.5cm,y=0.5cm]
\clip(-0.7,-0.1) rectangle (0.8,0.7);
\draw [line width=0.5pt] (0.,0.)-- (-0.5,0.5);
\draw [line width=0.5pt] (0.,0.)-- (-0.2,0.5);
\draw [line width=0.5pt] (0.,0.)-- (0.2,0.5);
\draw [line width=0.5pt] (0.,0.)-- (0.5,0.5);
\begin{scriptsize}
\draw [fill=black] (-0.5,0.5) circle (1.0pt);
\draw [fill=black] (-0.2,0.5) circle (1.0pt);
\draw [fill=black] (0.2,0.5) circle (1.0pt);
\draw [fill=black] (0.5,0.5) circle (1.0pt);
\draw [fill=black] (0.,0.) circle (1.0pt);
\end{scriptsize}
\end{tikzpicture}}
\newcommand{\tcinqdeux}{\begin{tikzpicture}[line cap=round,line join=round,>=triangle 45,x=0.5cm,y=0.5cm]
\clip(-0.5,-0.1) rectangle (0.5,1.2);
\draw [line width=0.5pt] (0.,0.)-- (-0.3,0.5);
\draw [line width=0.5pt] (0.,0.)-- (0.3,0.5);
\draw [line width=0.5pt] (0.,0.)-- (0.,0.5);
\draw [line width=0.5pt] (-0.3,0.5)-- (-0.3,1.);
\begin{scriptsize}
\draw [fill=black] (-0.3,0.5) circle (1.0pt);
\draw [fill=black] (0.,0.) circle (1.0pt);
\draw [fill=black] (0.,0.5) circle (1.0pt);
\draw [fill=black] (0.3,0.5) circle (1.0pt);
\draw [fill=black] (-0.3,1.) circle (1.0pt);
\end{scriptsize}
\end{tikzpicture}}

\newcommand{\tcinqcinq}{\begin{tikzpicture}[line cap=round,line join=round,>=triangle 45,x=0.5cm,y=0.5cm]
\clip(-0.5,-0.1) rectangle (0.5,1.2);
\draw [line width=0.5pt] (0.,0.)-- (-0.3,0.5);
\draw [line width=0.5pt] (-0.3,0.5)-- (-0.3,1.);
\draw [line width=0.5pt] (0.,0.)-- (0.3,0.5);
\draw [line width=0.5pt] (0.3,0.5)-- (0.3,1.);
\begin{scriptsize}
\draw [fill=black] (-0.3,0.5) circle (1.0pt);
\draw [fill=black] (0.,0.) circle (1.0pt);
\draw [fill=black] (-0.3,1.) circle (1.0pt);
\draw [fill=black] (0.3,0.5) circle (1.0pt);
\draw [fill=black] (0.3,1.) circle (1.0pt);
\end{scriptsize}
\end{tikzpicture}}
\newcommand{\tcinqsix}{\begin{tikzpicture}[line cap=round,line join=round,>=triangle 45,x=0.5cm,y=0.5cm]
\clip(-0.7,-0.1) rectangle (0.5,1.2);
\draw [line width=0.5pt] (0.,0.)-- (-0.3,0.5);
\draw [line width=0.5pt] (0.,0.)-- (0.3,0.5);
\draw [line width=0.5pt] (-0.3,0.5)-- (-0.6,1.);
\draw [line width=0.5pt] (-0.3,0.5)-- (0.,1.);
\begin{scriptsize}
\draw [fill=black] (-0.3,0.5) circle (1.0pt);
\draw [fill=black] (0.,0.) circle (1.0pt);
\draw [fill=black] (0.3,0.5) circle (1.0pt);
\draw [fill=black] (-0.6,1.) circle (1.0pt);
\draw [fill=black] (0.,1.) circle (1.0pt);
\end{scriptsize}
\end{tikzpicture}}

\newcommand{\tcinqhuit}{\begin{tikzpicture}[line cap=round,line join=round,>=triangle 45,x=0.5cm,y=0.5cm]
\clip(-0.5,-0.1) rectangle (0.5,1.7);
\draw [line width=0.5pt] (0.,0.)-- (-0.3,0.5);
\draw [line width=0.5pt] (-0.3,0.5)-- (-0.3,1.);
\draw [line width=0.5pt] (0.,0.)-- (0.3,0.5);
\draw [line width=0.5pt] (-0.3,1.)-- (-0.3,1.5);
\begin{scriptsize}
\draw [fill=black] (-0.3,0.5) circle (1.0pt);
\draw [fill=black] (0.,0.) circle (1.0pt);
\draw [fill=black] (-0.3,1.) circle (1.0pt);
\draw [fill=black] (0.3,0.5) circle (1.0pt);
\draw [fill=black] (-0.3,1.5) circle (1.0pt);
\end{scriptsize}
\end{tikzpicture}}

\newcommand{\tcinqdix}{\begin{tikzpicture}[line cap=round,line join=round,>=triangle 45,x=0.5cm,y=0.5cm]
\clip(-0.5,-0.1) rectangle (0.5,1.7);
\draw [line width=0.5pt] (0.,0.)-- (0.,0.5);
\draw [line width=0.5pt] (0.,0.5)-- (0.3,1.);
\draw [line width=0.5pt] (0.,0.5)-- (-0.3,1.);
\draw [line width=0.5pt] (0.,0.5)-- (0.,1.);
\begin{scriptsize}
\draw [fill=black] (0.,0.) circle (1.0pt);
\draw [fill=black] (0.,0.5) circle (1.0pt);
\draw [fill=black] (-0.3,1.) circle (1.0pt);
\draw [fill=black] (0.3,1.) circle (1.0pt);
\draw [fill=black] (0.,1.) circle (1.0pt);
\end{scriptsize}
\end{tikzpicture}}
\newcommand{\tcinqonze}{\begin{tikzpicture}[line cap=round,line join=round,>=triangle 45,x=0.5cm,y=0.5cm]
\clip(-0.5,-0.1) rectangle (0.5,1.7);
\draw [line width=0.5pt] (0.,0.)-- (0.,0.5);
\draw [line width=0.5pt] (0.,0.5)-- (0.3,1.);
\draw [line width=0.5pt] (0.,0.5)-- (-0.3,1.);
\draw [line width=0.5pt] (-0.3,1.)-- (-0.3,1.5);
\begin{scriptsize}
\draw [fill=black] (0.,0.) circle (1.0pt);
\draw [fill=black] (0.,0.5) circle (1.0pt);
\draw [fill=black] (-0.3,1.) circle (1.0pt);
\draw [fill=black] (0.3,1.) circle (1.0pt);
\draw [fill=black] (-0.3,1.5) circle (1.0pt);
\end{scriptsize}
\end{tikzpicture}}

\newcommand{\tcinqtreize}{\begin{tikzpicture}[line cap=round,line join=round,>=triangle 45,x=0.5cm,y=0.5cm]
\clip(-0.5,-0.1) rectangle (0.5,1.7);
\draw [line width=0.5pt] (0.,0.)-- (0.,0.5);
\draw [line width=0.5pt] (0.,0.5)-- (0.,1.);
\draw [line width=0.5pt] (0.,1.)-- (-0.3,1.5);
\draw [line width=0.5pt] (0.,1.)-- (0.3,1.5);
\begin{scriptsize}
\draw [fill=black] (0.,0.) circle (1.0pt);
\draw [fill=black] (0.,0.5) circle (1.0pt);
\draw [fill=black] (0.,1.) circle (1.0pt);
\draw [fill=black] (-0.3,1.5) circle (1.0pt);
\draw [fill=black] (0.3,1.5) circle (1.0pt);
\end{scriptsize}
\end{tikzpicture}}
\newcommand{\tcinqquatorze}{\begin{tikzpicture}[line cap=round,line join=round,>=triangle 45,x=0.5cm,y=0.5cm]
\clip(-.2,-.1) rectangle (0.2,2.2);
\draw [line width=0.5pt] (0.,0.)-- (0.,0.5);
\draw [line width=0.5pt] (0.,0.5)-- (0.,1.);
\draw [line width=0.5pt] (0.,1.)-- (0.,1.5);
\draw [line width=0.5pt] (0.,1.5)-- (0.,2.);
\begin{scriptsize}
\draw [fill=black] (0.,0.) circle (1.0pt);
\draw [fill=black] (0.,0.5) circle (1.0pt);
\draw [fill=black] (0.,1.) circle (1.0pt);
\draw [fill=black] (0.,1.5) circle (1.0pt);
\draw [fill=black] (0.,2.) circle (1.0pt);
\end{scriptsize}
\end{tikzpicture}}

\title{Construction of pre- and post-Lie algebras for stochastic PDEs}

\authors{Loïc Foissy}

\authorinfo{University Littoral Côte d'Opale, France}{foissy@univ-littoral.fr}

\abstract{%
    We give and study a construction of pre-Lie algebra structures on rooted trees whose edges and vertices are decorated, with a grafting product acting, through a map $\phi$, both on the decoration of the created edge and on the vertex that holds the grafting. We show that this construction gives a pre-Lie algebra if, and only if, the map $\phi$ satisfies a commutation relation, called tree-compatibility. We show how to extend this pre-Lie algebra structure to a post-Lie one by a semi-direct extension with another post-Lie algebra. We also define several constructions to obtain tree-compatible maps, and give examples, including a description of all tree-compatible maps when the space of decorations of the vertices is $2$-dimensional and the space of decorations of the edges is finite-dimensional. A particular example of such a construction is used by Bruned, Hairer and Zambotti for the study of stochastic partial differential equations: when no noise is involved, we show that the underlying tree-compatible map is the exponential of a simpler one and deduce an explicit isomorphism with a classical pre-Lie algebra of rooted trees; when a noise is involved, we obtain the underlying tree-compatible map as a direct sum. We also obtain with our formalism the post-Lie algebras described by Bruned and Katsetsiadis.
    }



\acknowledgments{The author thanks the anonymous referee for her or his useful comments.}

\VOLUME{1}

\ISSUE{1}

\NUMBER{8}

\DOI{https://doi.org/10.46298/jonas.17868}

\editinfo{March 31, 2026}{July 23, 2026}{Anastasia Doikou}

\begin{document}

	\section*{Introduction}

	Pre-Lie algebras, also known under the name Gerstenhaber or left-symmetric algebras, since their introduction by Vinberg~\cite{Vinberg63} and Gerstenhaber~\cite{Gerstenhaber63} in the sixties, are now widely used in numerous areas of mathematics. A pre-Lie algebra is a~vector space with a~bilinear product $\triangleleft$ satisfying the axiom
	\[
		x\triangleleft(y\triangleleft z)-(x\triangleleft y)\triangleleft z=y\triangleleft(x\triangleleft z)-(y\triangleleft x)\triangleleft z.
	\]
	Consequently, the antisymmetrization of $\triangleleft$ is a~Lie bracket. A typical example is given by the flat and torsion free connection on a~locally Euclidean space, see~\cite{Floystad2018} for more details on these geometric aspects. Chapoton and Livernet~\cite{Chapoton2001} gave a~construction of the operad of pre-Lie algebras, combinatorially described by rooted trees with an operadic composition based on insertion of a~tree at a~vertex of another tree. Equivalently, the free pre-Lie algebra generated by a~set $D_V$ is the vector space of rooted trees whose vertices are decorated by elements of $D_V$, with the pre-Lie product defined by graftings. Here is an example of computation in a~free pre-Lie algebra:
	if $b_1,b_2,b_3,b_4,b_5\in D_V$,
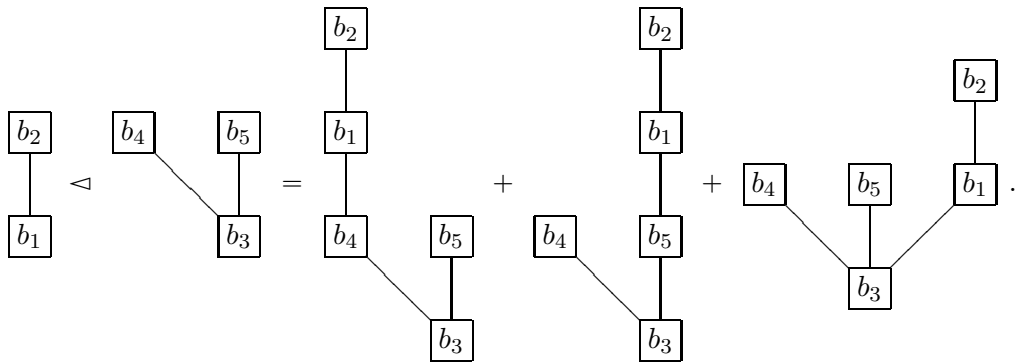
	\begin{align*}
		\begin{array}{c}\xymatrix{\boite{b_2}\ar@{-}[d]\\ \boite{b_1}}
		\end{array}
		\triangleleft
		\begin{array}{c}\xymatrix{\boite{b_4}\ar@{-}[rd]&\boite{b_5}\ar@{-}[d]\\&\boite{b_3}}
		\end{array}
		&=
		\begin{array}{c}\xymatrix{\boite{b_2}\ar@{-}[d]&\\ \boite{b_1}\ar@{-}[d]&\\\boite{b_4}\ar@{-}[rd]&\boite{b_5}\ar@{-}[d]\\&\boite{b_3}}
		\end{array}
		+
		\begin{array}{c}\xymatrix{&\boite{b_2}\ar@{-}[d]\\&\boite{b_1}\ar@{-}[d]\\\boite{b_4}\ar@{-}[rd]&\boite{b_5}\ar@{-}[d]\\&\boite{b_3}}
		\end{array}
		+
		\begin{array}{c}\xymatrix{&&\boite{b_2}\ar@{-}[d]\\\boite{b_4}\ar@{-}[rd]&\boite{b_5}\ar@{-}[d]&\boite{b_1}\ar@{-}[ld]\\&\boite{b_3}&}
		\end{array}
		.
	\end{align*}
	It is noticeable that this formalism of decorated trees, and consequently a~more or less hidden pre-Lie structure, appears in different works and domains:
	numerical analysis, Runge-Kutta methods and Butcher's series~\cite{Brouder2004,Grossman89,Grossman2005},
	Quantum Field Theory and Renormalization in the work of Connes and Kreimer~\cite{Connes1998},
	Ecalle's mould calculus and arborification's process~\cite{Ebrahimi-Fard2017-2,Ecalle2004,Fauvet2017}, etc.

	In the last decade, Bruned, Hairer and Zambotti~\cite{Bruned2019,Chandra2016,Hairer2014} introduced and used Regularity Structures, in order to study a~wide class of stochastic partial differential equations (shortly, sPDEs). They are based on Hopf algebraic structures built on rooted trees, with numerous decorations on vertices and edges, which interact in some sense. Dually, we obtain a~structure on these trees, with a~family of grafting products (one for each possible decoration of an edge), which modify the decoration of the vertex on which it is grafted. For example, in a~one-dimensional case, where vertices and edges are decorated by natural integers $a,a_1,a_2,a_3,b_1,b_2,b_3,b_4,b_5\in \N$,
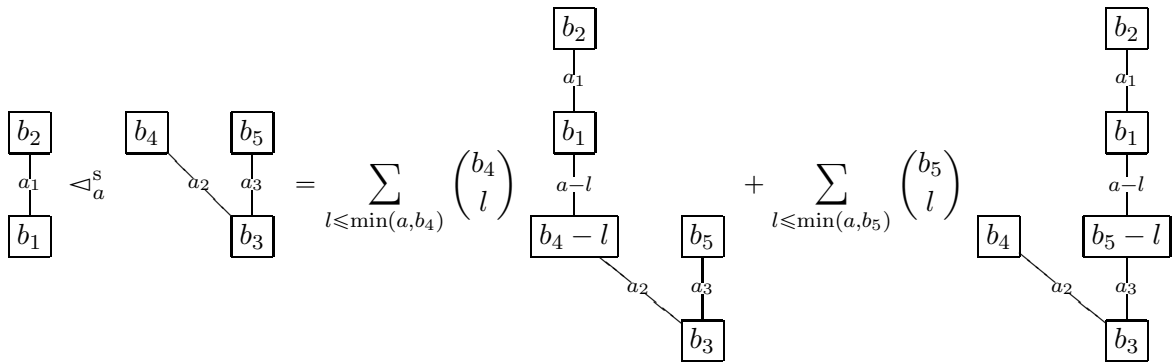
	\begin{align*}
		\begin{array}{c}\xymatrix{\boite{b_2}\ar@{-}[d]|{a_1}\\ \boite{b_1}}
		\end{array}
		\triangleleft_a^{\mathrm{s}}
		\begin{array}{c}\xymatrix{\boite{b_4}\ar@{-}[rd]|{a_2}&\boite{b_5}\ar@{-}[d]|{a_3}\\&\boite{b_3}}
		\end{array}
		&=\sum_{l\leq \min(a,b_4 )} \binom{b_4}{l}
		\begin{array}{c}\xymatrix{\boite{b_2}\ar@{-}[d]|{a_1}&\\ \boite{b_1}\ar@{-}[d]|{a-l}&\\\boite{b_4 -l}\ar@{-}[rd]|{a_2}&\boite{b_5}\ar@{-}[d]|{a_3}\\&\boite{b_3}}
		\end{array} \\
		&{\quad}+\sum_{l\leq \min(a,b_5 )} \binom{b_5}{l}
		\begin{array}{c}\xymatrix{&\boite{b_2}\ar@{-}[d]|{a_1}\\&\boite{b_1}\ar@{-}[d]|{a-l}\\\boite{b_4}\ar@{-}[rd]|{a_2}&\boite{b_5 -l}\ar@{-}[d]|{a_3}\\&\boite{b_3}}
		\end{array} \\
		&{\quad}+\sum_{l\leq \min(a,b_3 )} \binom{b_3}{l}
		\begin{array}{c}\xymatrix{&&\boite{b_2}\ar@{-}[d]|{a_1}\\\boite{b_4}\ar@{-}[rd]|{a_2}&\boite{b_5}\ar@{-}[d]|{a_3}&\boite{b_1}\ar@{-}[ld]|{a-l}\\&\boite{b_3 -l}&}
		\end{array}
		.
	\end{align*}
	In fact, this can be even more intricate when noise is considered, which needs other decorations $\Xi$ for the edges and $\star$ (represented by an absence of decoration in~\cite{Bruned2023-1}) for the vertices. For example,
	\begin{align*}
		\begin{array}{c}\xymatrix{\boite{b_2}\ar@{-}[d]|{a_1}\\ \boite{b_1}}
		\end{array}
		\triangleleft_a^{\mathrm{s}}
		\begin{array}{c}\xymatrix{\boite{b_4}\ar@{-}[rd]|{a_2}&\boite{\star}\ar@{-}[d]|{\Xi}\\&\boite{b_3}}
		\end{array}
		&=\sum_{l\leq \min(a,b_4 )} \binom{b_4}{l}
		\begin{array}{c}\xymatrix{\boite{b_2}\ar@{-}[d]|{a_1}&\\ \boite{b_1}\ar@{-}[d]|{a-l}&\\\boite{b_4 -l}\ar@{-}[rd]|{a_2}&\boite{\star}\ar@{-}[d]|{\Xi}\\&\boite{b_3}}
		\end{array} \\
		&{\qquad}+\sum_{l\leq \min(a,b_3 )} \binom{b_3}{l}
		\begin{array}{c}\xymatrix{&&\boite{b_2}\ar@{-}[d]|{a_1}\\\boite{b_4}\ar@{-}[rd]|{a_2}&\boite{\star}\ar@{-}[d]|{\Xi}&\boite{b_1}\ar@{-}[ld]|{a-l}\\&\boite{b_3 -l}&}
		\end{array}
		.
	\end{align*}
	In this formalism, edges may be decorated by pairs of elements in $\mathcal{T}\times \N^{d+1}$, where $\mathcal{T}$ is a~set, whose decorations are ``inert'' and do not interact. This structure is a~multiple pre-Lie algebra, also called matching pre-Lie algebra, see~\cite{Foissy47,Zhang2020}. These are objects with a~family $\triangleleft=(\triangleleft_a )_{a\in D_E}$ of products indexed by the set of decorations $D_E$, with the following compatibility: for any $a,a'\in D_E$,
	\[
		x\triangleleft_a (y\triangleleft_{a'}z)-(x\triangleleft_a y)\triangleleft_{a'}z=y\triangleleft_{a'}(x\triangleleft_a z)-(y\triangleleft_{a'} x)\triangleleft_a z.
	\]
	When $D_E$ is a~singleton, we recover (classical) pre-Lie algebras. Free multiple pre-Lie algebras have been described in~\cite{Bruned2023-1,Foissy47} in terms of rooted trees whose edges are decorated by elements of $D_E$ and vertices by elements of a~set $D_V$, corresponding to generators. Here is an example of computation in a~free multiple pre-Lie algebra: if $b_1,b_2,b_3,b_4,b_5\in D_V$ and $a,a_1,a_2,a_3\in D_E$,
	\begin{align*}
		\begin{array}{c}\xymatrix{\boite{b_2}\ar@{-}[d]|{a_1}\\ \boite{b_1}}
		\end{array}
		\triangleleft_a
		\begin{array}{c}\xymatrix{\boite{b_4}\ar@{-}[rd]|{a_2}&\boite{b_5}\ar@{-}[d]|{a_3}\\&\boite{b_3}}
		\end{array}
		&=
		\begin{array}{c}\xymatrix{\boite{b_2}\ar@{-}[d]|{a_1}&\\ \boite{b_1}\ar@{-}[d]|{a}&\\\boite{b_4}\ar@{-}[rd]|{a_2}&\boite{b_5}\ar@{-}[d]|{a_3}\\&\boite{b_3}}
		\end{array}
		+
		\begin{array}{c}\xymatrix{&\boite{b_2}\ar@{-}[d]|{a_1}\\&\boite{b_1}\ar@{-}[d]|{a}\\\boite{b_4}\ar@{-}[rd]|{a_2}&\boite{b_5}\ar@{-}[d]|{a_3}\\&\boite{b_3}}
		\end{array}
		+
		\begin{array}{c}\xymatrix{&&\boite{b_2}\ar@{-}[d]|{a_1}\\\boite{b_4}\ar@{-}[rd]|{a_2}&\boite{b_5}\ar@{-}[d]|{a_3}&\boite{b_1}\ar@{-}[ld]|{a}\\&\boite{b_3}&}
		\end{array}
		.
	\end{align*}
	In fact, the multiple pre-lie algebra structure $\triangleleft^{\mathrm{s}}$ used for sPDEs without noise, is shown in~\cite[Theorem 2.7]{Bruned2023-1} to be isomorphic to the free one $\triangleleft$, by a~triangularity argument.

	Another way to consider multiple pre-Lie algebras is to see them as $D_E$-graded homogeneous pre-Lie algebras: in a~clearer way, if $A$ is a~multiple pre-Lie algebra, then $\K D_E\otimes V$ is a~pre-Lie algebra, with the product defined by
	\begin{align*}
		&\forall a,a'\in D_E,\: \forall x,x'\in A,&a\otimes x\triangleleft a'\otimes x'&=a'\otimes (x\triangleleft_a x').
	\end{align*}
	When the set $D_E$ have more structures, variations are possible, such as family pre-Lie algebras~\cite{Manchon2020} or more general objects~\cite{Foissy48}. Here is an example of computation in a~free family pre-Lie algebra. Here, the set $D_E$ of decorations of the edges is a~commutative semi-group, whose product is denoted by $*$.
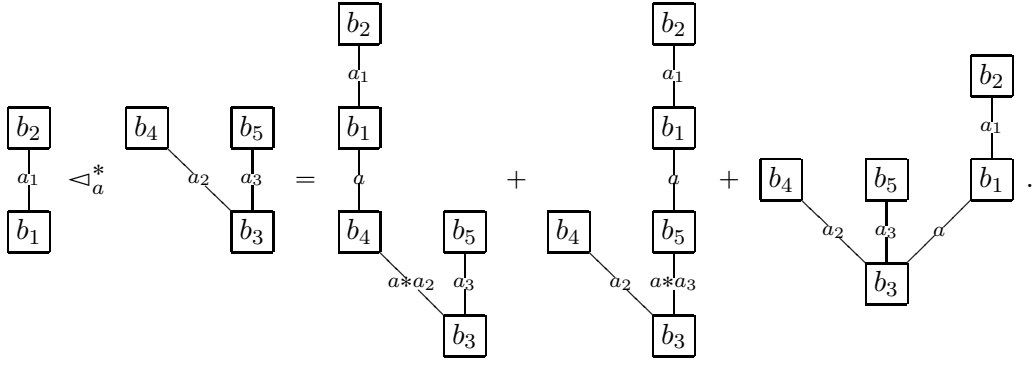
	\begin{align*}
		\begin{array}{c}\xymatrix{\boite{b_2}\ar@{-}[d]|{a_1}\\ \boite{b_1}}
		\end{array}
		\triangleleft_a^*
		\begin{array}{c}\xymatrix{\boite{b_4}\ar@{-}[rd]|{a_2}&\boite{b_5}\ar@{-}[d]|{a_3}\\&\boite{b_3}}
		\end{array}
		&=
		\begin{array}{c}\xymatrix{\boite{b_2}\ar@{-}[d]|{a_1}&\\ \boite{b_1}\ar@{-}[d]|{a}&\\\boite{b_4}\ar@{-}[rd]|{a* a_2}&\boite{b_5}\ar@{-}[d]|{a_3}\\&\boite{b_3}}
		\end{array}
		+
		\begin{array}{c}\xymatrix{&\boite{b_2}\ar@{-}[d]|{a_1}\\&\boite{b_1}\ar@{-}[d]|{a}\\\boite{b_4}\ar@{-}[rd]|{a_2}&\boite{b_5}\ar@{-}[d]|{a* a_3}\\&\boite{b_3}}
		\end{array}
		+
		\begin{array}{c}\xymatrix{&&\boite{b_2}\ar@{-}[d]|{a_1}\\\boite{b_4}\ar@{-}[rd]|{a_2}&\boite{b_5}\ar@{-}[d]|{a_3}&\boite{b_1}\ar@{-}[ld]|{a}\\&\boite{b_3}&}
		\end{array}
		.
	\end{align*}
	In all these examples, the decorations of the edges can be modified, but the decorations of the vertices are left untouched, a~major difference with the operations used for sPDEs.

	Our aim in this paper is to understand and formalize the grafting operations used for sPDEs. We now fix vector spaces $D_E$ of decorations of the edges, and $D_V$ of decorations of the vertices $D_V$, and consider the space $\calT(D_E,D_V )$ of rooted trees where vertices are decorated by elements of $D_V$ and edges by elements of $D_E$; we take in account the vector space structure of $D_E$ and $D_V$
	by considering that the decorations are understood as linear in each vertex and each edge, which will lead us to define a~tensor vector space of decorations for each tree, see Definition~\ref{defi1.1}. We fix a~linear map $\phi:D_E\otimes D_V\longrightarrow D_E\otimes D_V$, which will represent the modification of the decorations, of the vertex on which one grafts, and of the edge created when one grafts. We define for any $a\in D_E$ a~grafting product $\triangleleft^\phi_a$ on $\calT(D_E,D_V )$. Here is an example of grafting in $\calT(D_E,D_V )$: if $b_1,b_2,b_3,b_4,b_5\in D_V$ and $a,a_1,a_2,a_3\in D_E$,
	\begin{align*}
		\begin{array}{c}\xymatrix{\boite{b_2}\ar@{-}[d]|{a_1}\\ \boite{b_1}}
		\end{array}
        { }&{ }
		\triangleleft^\phi_a
		\begin{array}{c}\xymatrix{\boite{b_4}\ar@{-}[rd]|{a_2}&\boite{b_5}\ar@{-}[d]|{a_3}\\&\boite{b_3}}
		\end{array} \\
		&=\sum_i\left(
		\begin{array}{c}\xymatrix{\boite{b_2}\ar@{-}[d]|{a_1}&\\ \boite{b_1}\ar@{-}[d]|{\phi_E^i (a)}&\\\boite{\phi_V^i (b_4 )}\ar@{-}[rd]|{a_2}&\boite{b_5}\ar@{-}[d]|{a_3}\\&\boite{b_3}}
		\end{array}
		+
		\begin{array}{c}\xymatrix{&\boite{b_2}\ar@{-}[d]|{a_1}\\&\boite{b_1}\ar@{-}[d]|{\phi_E^i (a)}\\\boite{b_4}\ar@{-}[rd]|{a_2}&\boite{\phi_V^i (b_5 )}\ar@{-}[d]|{a_3}\\&\boite{b_3}}
		\end{array}
		+
		\begin{array}{c}\xymatrix{&&\boite{b_2}\ar@{-}[d]|{a_1}\\\boite{b_4}\ar@{-}[rd]|{a_2}&\boite{b_5}\ar@{-}[d]|{a_3}&\boite{b_1}\ar@{-}[ld]|{\phi_E^i (a)}\\&\boite{\phi_V^i (b_3 )}&}
		\end{array}
		\right),
	\end{align*}
	where we used Sweedler's like notations for $\phi$: if $a\in D_E$ and $b\in D_V$, we put
	\[
		\phi(a\otimes b)=\sum_i \phi_E^i (a)\otimes \phi_V^i (b).
	\]
	For a~generic $\phi$, these products have no specific relations. However, we prove that they form a~multiple pre-Lie algebra if, and only if, for any $a,a'\in D_E$ and any $b\in D_V$,
	\[
		\sum_i\sum_j \phi_E^i (a)\otimes \phi_E^j (a')\otimes\phi_V^i \circ \phi_V^j (b)=\sum_i\sum_j \phi_E^i (a)\otimes \phi_E^j (a')\otimes \phi_V^j \circ \phi_V^i (b).
	\]
	This condition will be called the tree-compatibility of $\phi$ and will be shortly denoted by the commutation condition
	\[
		\phi_{13}\circ \phi_{23}=\phi_{23}\circ \phi_{13}:D_E\otimes D_E\otimes D_V\longrightarrow D_E\otimes D_E\otimes D_V,
	\]
	see Definition~\ref{defi1.2}. If $\phi$ is tree-compatible, then it is possible to define a~linear endomorphism $\Theta_\phi$ of $\calT(D_E,D_V )$, acting on each decorated tree by the action of $\phi$ on any pair $(e,s(e))$, where $e$ is an edge and $s(e)$ is the source of $e$; the tree-compatibility insures that this does not depend on the order chosen on the edges to perform the action of $\phi$. For example,
	\begin{align*}
		\Theta_\phi\left(
		\begin{array}{c}\xymatrix{\boite{b_1}\ar@{-}[dr]|{a_1}&&\boite{b_2}\ar@{-}[dl]|{a_2}\\ &\boite{b_3}&}
		\end{array}
		\right)
		&=\sum_i\sum_j
		\begin{array}{c}
			\xymatrix{\boite{b_1}\ar@{-}[dr]|{\phi_E^i (a_1 )}&&\boite{b_2}\ar@{-}[dl]|{\phi_E^j (a_2 )}\\ &\boite{\phi_V^i\circ \phi_V^j (b_3 )}&}
		\end{array}
		.
	\end{align*}
	We prove that $\Theta_\phi$ is a~multiple pre-Lie algebra morphism from the free multiple pre-Lie algebra $(\calT(D_E,D_V ),\triangleleft)$ to its $\phi$-deformed version $(\calT(D_E,D_V ),\triangleleft^\phi)$. As a~consequence, if $\phi$ is tree-compatible and bijective, then $\phi^{-1}$ is also tree-compatible and $\Theta_\phi$ is an isomorphism, of inverse $\Theta_{\phi^{-1}}$, see Corollary~\ref{cor3.5}. We also prove that $(\calT(D_E,D_V ),\triangleleft^\phi)$ is generated by the space of trees with only one vertex if, and only if, $\phi$ is surjective (Corollary~\ref{cor3.6}), and freely generated by these elements if, and only if, $\phi$ is bijective (Corollary~\ref{cor3.7}). The pre-Lie algebra on $D_E\otimes \calT(D_E,D_V )$ can be seen as a~pre-Lie algebra on planted decorated trees, which are obtained from decorated tree by adding a~non-decorated root and a~decorated edge. The Guin-Oudom construction~\cite{Oudom2005,Oudom2008} can be used on this pre-Lie algebra of planted rooted trees, which we do in Proposition~\ref{prop3.11}. We obtain a~Grossman-Larson-like Hopf algebra~\cite{Grossman89,Grossman90}. For example, if $a_1,a_2,a_3,a_4\in D_E$ and $b_1,b_2,b_3,b_4\in D_V$,
	\begin{align*}
		\begin{array}{c}
			\xymatrix{\boite{b_1}\ar@{-}[d]|{a_1}\\ \boite{\blackdiamond}}
		\end{array}
		\begin{array}{c}
			\xymatrix{\boite{b_2}\ar@{-}[d]|{a_a}\\ \boite{\blackdiamond}}
		\end{array}
		\star^\phi
		\begin{array}{c}
			\xymatrix{\boite{b_4}\ar@{-}[d]|{a_4}\\\boite{b_3}\ar@{-}[d]|{a_3}\\ \boite{\blackdiamond}}
		\end{array}
		&=\sum_{i,j}\left(
		\begin{array}{c}
			\begin{array}{c}
			\xymatrix{\boite{b_1}\ar@{-}[rd]|{\phi_E^i (a_1 )}&\boite{b_2}\ar@{-}[d]|{\phi_E^i (a_2 )}&\boite{b_4}\ar@{-}[ld]|{a_4}\\ &\boite{\phi_V^i\circ \phi_V^j (b_3 )}\ar@{-}[d]|{a_3}&\\ &\boite{\blackdiamond}&}
		\end{array}
		+
		\begin{array}{c}
			\xymatrix{&\boite{b_2}\ar@{-}[d]|{\phi_E^j (a_2 )}\\\boite{b_1}\ar@{-}[rd]|{\phi_E^i (a_1 )}&\boite{\phi_V^j (b_4 )}\ar@{-}[d]|{a_4}\\&\boite{\phi_V^i (b_3 )}\ar@{-}[d]|{a_3}\\&\boite{\blackdiamond}}
		\end{array}
		\\
		+
		\begin{array}{c}
			\xymatrix{&\boite{b_1}\ar@{-}[d]|{\phi_E^i (a_1 )}\\\boite{b_2}\ar@{-}[rd]|{\phi_E^j (a_2 )}&\boite{\phi_V^i (b_4 )}\ar@{-}[d]|{a_4}\\&\boite{\phi_V^j (b_3 )}\ar@{-}[d]|{a_3}\\&\boite{\blackdiamond}}
		\end{array}
		+
		\begin{array}{c}
			\xymatrix{\boite{b_1}\ar@{-}[rd]|{\phi_E^i (a_1 )}&\boite{b_2}\ar@{-}[d]|{\phi_E^j (a_2 )}\\&\boite{\phi_V^i\circ \phi_V^j (b_4 )}\ar@{-}[d]|{a_4}\\&\boite{b_3}\ar@{-}[d]|{a_3}\\&\boite{\blackdiamond}}
		\end{array}
		\end{array}
		\right)\\
		&{\quad}+\sum_i \left(
		\begin{array}{c}
			\xymatrix{\boite{b_1}\ar@{-}[d]|{a_1}\\ \boite{\blackdiamond}}
		\end{array}
		\begin{array}{c}
			\xymatrix{\boite{b_2}\ar@{-}[rd]|{\phi_E^i (a_2 )}&\boite{b_4}\ar@{-}[d]|{a_4}\\&\boite{\phi_V^i (b_3 )}\ar@{-}[d]|{a_3}\\&\boite{\blackdiamond}}
		\end{array}
		+
		\begin{array}{c}
			\xymatrix{\boite{b_1}\ar@{-}[d]|{a_1}\\ \boite{\blackdiamond}}
		\end{array}
		\begin{array}{c}
			\xymatrix{\boite{b_2}\ar@{-}[d]|{\phi_E^i (a_2 )}\\\boite{\phi_V^i (b_4 )}\ar@{-}[d]|{a_4}\\\boite{b_3}\ar@{-}[d]|{a_3}\\\boite{\blackdiamond}}
		\end{array}
		\right)
	\end{align*}

	\begin{align*}
		&+\sum_i \left(
		\begin{array}{c}
			\xymatrix{\boite{b_2}\ar@{-}[d]|{a_2}\\ \boite{\blackdiamond}}
		\end{array}
		\begin{array}{c}
			\xymatrix{\boite{b_1}\ar@{-}[rd]|{\phi_E^i (a_1 )}&\boite{b_4}\ar@{-}[d]|{a_4}\\&\boite{\phi_V^i (b_3 )}\ar@{-}[d]|{a_3}\\&\boite{\blackdiamond}}
		\end{array}
		+
		\begin{array}{c}
			\xymatrix{\boite{b_2}\ar@{-}[d]|{a_2}\\ \boite{\blackdiamond}}
		\end{array}
		\begin{array}{c}
			\xymatrix{\boite{b_1}\ar@{-}[d]|{\phi_E^i (a_1 )}\\\boite{\phi_V^i (b_4 )}\ar@{-}[d]|{a_4}\\\boite{b_3}\ar@{-}[d]|{a_3}\\\boite{\blackdiamond}}
		\end{array}
		\right)+
		\begin{array}{c}
			\xymatrix{\boite{b_1}\ar@{-}[d]|{a_1}\\ \boite{\blackdiamond}}
		\end{array}
		\begin{array}{c}
			\xymatrix{\boite{b_2}\ar@{-}[d]|{a_a}\\ \boite{\blackdiamond}}
		\end{array}
		\begin{array}{c}
			\xymatrix{\boite{b_4}\ar@{-}[d]|{a_4}\\\boite{b_3}\ar@{-}[d]|{a_3}\\ \boite{\blackdiamond}}
		\end{array}
		.
	\end{align*}
	The dual construction gives a~Butcher-Connes-Kreimer Hopf algebra of planted rooted trees~\cite{Connes1998,Panaite2000,Hoffman2003}. For example, if $a_1,a_2,a_3\in D_E$ and $b_1,b_2,b_3\in D_V$,
	\begin{align*}
		\Delta^\phi\left(
		\begin{array}{c}
			\xymatrix{\boite{b_1}\ar@{-}[rd]|{a_1}&\boite{b_2}\ar@{-}[d]|{a_2}\\&\boite{b_3}\ar@{-}[d]|{a_3}\\&\boite{\blackdiamond}}
		\end{array}
		\right)&=x\otimes 1+1\otimes x+\sum_i
		\begin{array}{c}
			\xymatrix{\boite{b_1}\ar@{-}[d]|{\phi_E^i (a_1 )}\\ \boite{\blackdiamond}}
		\end{array}
		\otimes
		\begin{array}{c}
			\xymatrix{\boite{b_2}\ar@{-}[d]|{a_2}\\\boite{\phi_V^i (b_3 )}\ar@{-}[d]|{a_3}\\ \boite{\blackdiamond}}
		\end{array}
		+\sum_i
		\begin{array}{c}
			\xymatrix{\boite{b_2}\ar@{-}[d]|{\phi_E^i (a_2 )}\\ \boite{\blackdiamond}}
		\end{array}
		\otimes
		\begin{array}{c}
			\xymatrix{\boite{b_1}\ar@{-}[d]|{a_1}\\\boite{\phi_V^i (b_3 )}\ar@{-}[d]|{a_3}\\ \boite{\blackdiamond}}
		\end{array}
		\\
		&+\sum_{i,j}
		\begin{array}{c}
			\xymatrix{\boite{b_1}\ar@{-}[d]|{\phi_E^i (a_1 )}\\ \boite{\blackdiamond}}
		\end{array}
		\begin{array}{c}
			\xymatrix{\boite{b_2}\ar@{-}[d]|{\phi_E^j (a_2 )}\\ \boite{\blackdiamond}}
		\end{array}
		\otimes
		\begin{array}{c}
			\xymatrix{\boite{\phi_V^i\circ \phi_V^j (b_3 )}\ar@{-}[d]|{a_1}\\ \boite{\blackdiamond}}
		\end{array}
		.
	\end{align*}
	Let us make precise what we mean by ``dual construction''. When $D_E$ and $D_V$ are finite-dimensional, then we can identify the dual $(D_E\otimes D_V )^*$ with $D_E^*\otimes D_V^*$. Then, the transpose of an endomorphism $\phi$ is tree-compatible if, and only if, $\phi$ is tree-compatible. More generally, we consider a~pairing $\langle-,-\rangle$ between $D_E '\otimes D_V '$ and $D_E\otimes D_V$, maybe degenerate, and two tree-compatible maps $\phi$ on $D_E\otimes D_V$ and $D'_E\otimes D'_V$ such that
	\begin{align*}
		&\forall a\in D_E,\: b\in D_V,\:a'\in D'_E,\:b'\in D'_V,&\langle \phi'(a'\otimes b'),a\otimes b\rangle&=\langle a'\otimes b,\phi(a\otimes b)\rangle.
	\end{align*}
	Then this pairing is extended into a~Hopf pairing between the bialgebra associated by the Guin-Oudom construction to the pre-Lie algebra $(D_E\otimes \calT(D'_E,D'_V ),\triangleleft^{\phi'})$ and the Butcher-Connes-Kreimer bialgebra on $S(D_E\otimes \calT(D_E,D_V ))$ associated to $\phi$
	(Theorem~\ref{theodualite}).

	We introduce several operations on tree-compatible maps: for example, we prove that if $\phi$ is tree-compatible, then any polynomial in $\phi$ is also tree-compatible and, under a~condition of local nilpotency of $\phi$, any formal series in $\phi$ is also tree-compatible, see Proposition~\ref{prop2.4}. This observation allows a~better understanding of the structure used for sPDEs, firstly in the case without noise: in this case, $\DEPDE=\DVPDE=\vect\left(\N^{d+1}\right)$. For any $i\in \llbracket 0;d\rrbracket$, a~rather simple tree-compatible map is given by
	\[
		\partial_i (a\otimes b)=
		\begin{cases}
			b_i\left(a-\epsilon^{(i)}\right)\otimes \left(b-\epsilon^{(i)}\right)\mbox{ if }a_i,b_i\geq 1,\\
			0\mbox{ otherwise},
		\end{cases}
	\]
	where $(\epsilon^{(0)},\ldots,\epsilon^{(d)})$ is the canonical basis of $\N^{d+1}$. We prove that $\partial^\lambda=\sum_{i=0}^d \lambda_i \partial_i$ is tree-compatible and locally nilpotent for all $\lambda \in \K^{d+1}$, so any formal series in $\partial^\lambda$ is also tree-compatible: in particular, the map used for sPDEs is $\phi^{\mathrm{s}}=\exp(\partial_0 +\cdots+\partial_d )$, which is consequently tree-compatible:
	we recover the fact that the products associated to this map are multiple pre-Lie~\cite{Bruned2023,Bruned2023-1}, and that $(\calT(\DEPDE,\DVPDE),\triangleleft)$ is isomorphic to $(\calT(\DEPDE,\DVPDE),\triangleleft^{\phi^{\mathrm{s}}})$, through the explicit isomorphism $\Theta_{\phi^{\mathrm{s}}}$, whose inverse is the map $\Theta_{\phi^{-\mathrm{s}}}$ associated to the tree-compatible map $\phi^{-\mathrm{s}}=\exp(-\partial_0 -\cdots-\partial_d )$, see Theorem~\ref{theo4.6}. For sPDEs with noise, one uses two other symbols $\Xi$ on the edges and $\star$ on the vertices. In our context, this is obtained by an operation of direct sum of tree-compatible maps (defined in Proposition~\ref{prop2.1}), applied on $\phi^{\mathrm{s}}$ and to the zero map of $\vect(\Xi)\otimes \vect(\star)$. We prove in Proposition~\ref{prop4.9} that the planted trees needed for sPDES are precisely the one generated by certain planted trees with only one edge (Proposition~\ref{prop4.10}).

	Another important aspect of~\cite{Bruned2023} is the definition of post-Lie algebraic structures. Roughly speaking, post-Lie algebras are Lie algebras with an extra product $\triangleleft$, such that the sum of the Lie bracket and the antisymmetrization of $\triangleleft$ is also a~Lie bracket, see Definition~\ref{defi3.8} for details. In particular, pre-Lie algebras are post-Lie algebras whose underlying Lie bracket is abelian. Firstly defined in an operadic context~\cite{Vallette2007}, they turned out to describe the structure of a~flat, constant torsion connection given by the Maurer--Cartan form on a~Lie group, see for example~\cite{Al-Kaabi2022-2,Al-Kaabi2022,Bruned2023,Curry2020,Ebrahimi-Fard2015,Ebrahimi-Fard2017,Floystad2018,Lundervold2013,Munthe-Kaas2008} for examples of applications in this setting. The post-Lie structure of~\cite{Bruned2023} is an extension of the pre-Lie structure on planted trees by the addition of generators $X_i$ for $0\leq i\leq d$, acting on the decoration of the edge attached to the root for the bracket, and on the decoration of the vertices for $\triangleleft$. We formalize this here with the help of a~fixed post-Lie algebra $P$ and two linear maps $\psi_E$ and $\psi_V$, the first giving the action of $P$ on the decorations of the edges and the second one the action of $P$ on the vertices, and we extend both the post-Lie structure on $P$ and the pre-Lie structure on $D_E\otimes \calT(D_E,D_V )$ to $(D_E\otimes \calT(D_E,D_V ))\oplus P$ using these actions. We give the conditions that the triple $(\phi,\psi_V,\psi_E )$ in Proposition~\ref{prop3.9} to obtain a~post-Lie algebra. When $P$ is a~trivial post-Lie algebra, these conditions simplify, as shown in Corollary~\ref{cor3.10}. We show that these conditions are satisfied in the case coming from sPDEs, with or without noise, with $P=\vect(x_i\mid 0\leq i\leq d)$ as a~trivial post-Lie algebra, and $\psi_V$ and $\psi_E$ defined by
	\begin{align*}
		&\forall a,b\in \N^{d+1},&\psi_E (X_i )(a)&=
		\begin{cases}
			a-\epsilon^{(i)}\mbox{ if }a_i\geq 1,\\
			0\mbox{ otherwise},
		\end{cases}
		& \psi_V (X_i )(b)&=b+\epsilon^{(i)},\\
		&&\overline{\psi}_E (X_i )(\Xi)&=0,&\overline{\psi}_V (X_i )(\star)&=0
	\end{align*}
	see Propositions~\ref{prop4.8} and~\ref{prop4.9}. Here are examples of computations in the case of sPDEs: if $a,b_1,b_2,b_4,b_4,a_1,a_2,a_3\in \N^{d+1}$,
	\begin{align*}
		X_i\triangleleft a\otimes
		\begin{array}{c}\xymatrix{ \boite{b_3}\ar@{-}[d]|{a_2}&\\\boite{b_2}\ar@{-}[rd]|{a_1}&\boite{b_4}\ar@{-}[d]|{a_3}\\&\boite{b_1}}
		\end{array}
		&= a\otimes
		\begin{array}{c}\xymatrix{ \boite{b_3}\ar@{-}[d]|{a_2}&\\\boite{b_2}\ar@{-}[rd]|{a_1}&\boite{b_4}\ar@{-}[d]|{a_3}\\&\boite{b_1 +\epsilon^{(i)}}}
		\end{array}
		+ a\otimes
		\begin{array}{c}\xymatrix{ \boite{b_3}\ar@{-}[d]|{a_2}&\\\boite{b_2 +\epsilon^{(i)}}\ar@{-}[rd]|{a_1}&\boite{b_4}\ar@{-}[d]|{a_3}\\&\boite{b_1}}
		\end{array}
		\\
		&+ a\otimes
		\begin{array}{c}\xymatrix{ \boite{b_3 +\epsilon^{(i)}}\ar@{-}[d]|{a_2}&\\\boite{b_2}\ar@{-}[rd]|{a_1}&\boite{b_4}\ar@{-}[d]|{a_3}\\&\boite{b_1}}
		\end{array}
		+ a\otimes
		\begin{array}{c}\xymatrix{ \boite{b_3}\ar@{-}[d]|{a_2}&\\\boite{b_2}\ar@{-}[rd]|{a_1}&\boite{b_4 +\epsilon^{(i)}}\ar@{-}[d]|{a_3}\\&\boite{b_1}}
		\end{array}
		,
	\end{align*}
	and
	\begin{align*}
		\left\{a\otimes
		\begin{array}{c}\xymatrix{ \boite{b_3}\ar@{-}[d]|{a_2}&\\\boite{b_2}\ar@{-}[rd]|{a_1}&\boite{b_4}\ar@{-}[d]|{a_3}\\&\boite{b_1}}
		\end{array}
		,X_i\right\}
		&=\left(a-\epsilon^{(i)}\right)\otimes
		\begin{array}{c}\xymatrix{ \boite{b_3}\ar@{-}[d]|{a_2}&\\\boite{b_2}\ar@{-}[rd]|{a_1}&\boite{b_4}\ar@{-}[d]|{a_3}\\&\boite{b_1}}
		\end{array}
		.
	\end{align*}
	This text is organized as follows. In the first section, we make the space of decorations associated to any rooted tree precise, and define tree-compatibility. We study operations of tree-compatible maps (direct sums, polynomials, formal series, tensor products) in the second section. We give the main theorem (Theorem~\ref{theo3.3}) on the pre-Lie multiple algebra associated to a~tree-compatible map in the next section, as well the construction of the morphism $\Theta_\phi$ (Corollary~\ref{cor3.5}) and its applications. This third section also contains results on the extension by a~post-Lie algebra (Proposition~\ref{prop3.9} and Corollary~\ref{cor3.10}). The last section contains examples, sPDEs with or without noises, giving back the structure of~\cite{Bruned2023}, and a~description of all tree-compatible maps in dimension $m\times 2$ when one works over $\C$ in Theorem~\ref{theo4.14}. In the appendix, we prove that the pre-Lie algebra $D_E\otimes \calT(D_E\otimes D_V )$, with the product $\triangleleft^\phi$ associated to a~tree-compatible map $\phi$, is free whenever $\phi$ is invertible.

	\begin{notation}
		\begin{enumerate}
			\item We denote by $\K$ a~commutative field of characteristic zero. All the vector spaces of this text will be taken over this field.
			\item In the whole text, $D_E$ and $D_V$ are two nonzero vector spaces, which correspond to the set of decorations of the edges and of the vertices of the rooted trees we shall consider.
		\end{enumerate}
	\end{notation}
	\section{Construction}

	\subsection{Vector space associated to a~rooted tree}
	\begin{notation}
		\begin{enumerate}
			\item The set of rooted trees is denoted by $\calT$:
			\[
				\calT=\left\{
				\begin{array}
					{c}
					\tun,\tdeux,\ttroisun,\ttroisdeux,\tquatreun,\tquatredeux,\tquatrequatre,\tquatrecinq,\tcinqun,\tcinqdeux,\tcinqcinq,\tcinqsix,\tcinqhuit,\tcinqdix,\tcinqonze,\tcinqtreize,\tcinqquatorze,\ldots
				\end{array}
				\right\}
			\]
			\item Let $T\in \calT$. The sets of vertices and edges of $T$ are respectively denoted by $V(T)$ and $E(T)$. The edges of $T$ are oriented upwards, from the root to the leaves. This defines two maps, the source $s$ and the target $t$ maps, both from $E(T)$ to $V(T)$.
		\end{enumerate}
	\end{notation}

	\begin{remark}
		For any rooted tree $T$, the target map is injective, and its image is the set of vertices of $T$ different from the root. The source map is generally not injective, except if $T$ is a~ladder:
		\[
			\tun,\tdeux,\ttroisdeux,\tquatrecinq,\tcinqquatorze\ldots
		\]
	\end{remark}

	\begin{definition}\label{defi1.1}
		\begin{enumerate}
			\item Let $T$ be a~rooted tree. The vector space associated to $T$ is
			\[
				D(T)=\bigotimes_{e\in E(T)} D_E \otimes \bigotimes_{v\in V(T)} D_V.
			\]
			The copy of $D_E$ indexed by $e\in E(T)$ will be denoted by $D_e$, and the copy of $D_V$ indexed by $v\in V(T)$ by $D_v$. Elements of $D(T)$ are linear spans of tensors of the form
			\[
				\bigotimes_{e\in E(T)} a_e\otimes \bigotimes_{v\in V(T)} b_v,
			\]
			where, for any $e\in E(T)$, $a_e\in D_e$ is considered as the decoration of the edge $e$ and for any $v\in V(T)$, $b_v\in D_v$ is considered as the decoration of the vertex $v$. In the sequel, we shall allow the order of the factors in the tensor product of decorations of a~given rooted tree to permute, in order to simplify the writing if necessary.
			\item We denote by $\calT(D_E,D_V )$ the vector space
			\[
				\calT(D_E,D_V )=\bigoplus_{T\in \calT} D(T).
			\]
			Elements of this space are linear span of elements of the form
			\begin{align}\label{EQ1}
				x=T\otimes\underbrace{\bigotimes_{e\in E(T)} a_e\otimes \bigotimes_{v\in V(T)} b_v}_{W_T},
			\end{align}
			where $T$ is a~rooted tree, written here to distinguish the different terms in the direct sum defining the space $\calT(D_E,D_V )$. In the sequel, each we shall use an element $x$, $x'$ or $x''$ of $\calT(D_E,D_V )$, it will have the form of \eqref{EQ1} with coherent $'$ or $''$ everywhere it will be needed.
		\end{enumerate}
	\end{definition}
	Here is a~graphical representation of an element of $\calT(D_E,D_V )$ with four vertices and three edges, where $a_1,a_2,a_3\in D_E$ and $b_1,b_2,b_3,b_4\in D_V$:
	\[
		\tquatredeux\otimes a_1\otimes a_2\otimes a_3\otimes b_1\otimes b_2\otimes b_3\otimes b_4 =
		\begin{array}{c}
			\xymatrix{\boite{b_4}\ar@{-}[d]|{a_3}&&\\ \boite{b_2}\ar@{-}[dr]|{a_1}&&\boite{b_3}\ar@{-}[dl]|{a_2}\\&\boite{b_1}&}
		\end{array}
		.
	\]
	We ordered the edges and vertices of $\tquatredeux$ according to the depth-first search order to write the decorations of this tree as a~tensor.

	\subsection{Actions on decorations}
	\begin{definition}\label{defi1.2}
		Let $\phi:D_E\otimes D_V\longrightarrow D_E\otimes D_V$ be a~linear map.
		\begin{itemize}
			\item We define two endomorphisms $\phi_{13}$ and $\phi_{23}$ of $D_E\otimes D_E\otimes D_V$ by
			\begin{align*}
				\phi_{23}&=\id_{D_E}\otimes \phi,&
				\phi_{13}&=(\tau\otimes \id_{D_V})\circ (\id_{D_E}\otimes \phi)\circ (\tau\otimes \id_{D_V}),
			\end{align*}
			where $\tau:D_E\otimes D_E\longrightarrow D_E\otimes D_E$ is the usual flip.
			\item We shall say that $\phi$ is tree-compatible if $\phi_{13}\circ \phi_{23}=\phi_{23}\circ \phi_{13}$.
		\end{itemize}
	\end{definition}

	\begin{notation}
		We adopt a~Sweedler-like notation for $\phi$: if $a\in D_E$, $b\in D_V$,
		\[
			\phi(a\otimes b)=\sum_i \phi_E^i (a)\otimes \phi_V^i (b).
		\]
		If $a,a'\in D_E$ and $b\in D_V$,
		\begin{align*}
			\phi_{13}(a\otimes a'\otimes b)&=\sum_i \phi_E^i (a)\otimes a'\otimes \phi_V^i (b),&\phi_{23}(a\otimes a'\otimes b)&=\sum_i a\otimes \phi_E^i (a')\otimes \phi_V^i (b).
		\end{align*}
		The tree-compatibility can be written as
		\[
			\sum_i\sum_j \phi_E^i (a)\otimes \phi_E^j (a')\otimes\phi_V^i \circ \phi_V^j (b)
            =\sum_i\sum_j \phi_E^i (a)\otimes \phi_E^j (a')\otimes \phi_V^j \circ \phi_V^i (b).
		\]
        for all $a,a'\in D_E$ and $b\in D_V$.
	\end{notation}

	\begin{notation}
		Let $T$ be a~rooted tree. For any $e_0\in E(T)$ and $v_0\in V(T)$, we consider the endomorphism $\phi_{e_0,v_0}$ of $D(T)$ defined by
		\[
			\phi_{e_0,v_0}=\bigotimes_{e\in E(T)\setminus\{e_0\}} \id_{D_e}\otimes \bigotimes_{v\in V(T)\setminus\{v_0\}} \id_{D_v}\otimes \phi.
		\]
		In other words, $\phi_{e_0,v_0}$ is given by the action of $\phi$ on $D_{e_0}\otimes D_{v_0}$, the other decorations being unchanged. With Sweedler's notation,
		\[
			\phi_{e_0,v_0}\left( \bigotimes_{e\in E(T)}a_e \otimes \bigotimes_{v\in V(T)}b_v\right)=\sum_i \phi_E^i (a_{e_0})\otimes \bigotimes_{e\in E(T)\setminus\{e_0\}}d_e\otimes \phi_V^i (b_{v_0})\otimes \bigotimes_{v\in V(T)\setminus\{v_0\}}b_v.
		\]
	\end{notation}

	\begin{lemma}\label{lem1.3}
		Assume that $\phi$ is tree-compatible. Let $T$ be a~rooted tree. For any $e,e'\in E(T)$, $\phi_{e,s(e)}\circ \phi_{e',s(e')}=\phi_{e',s(e')}\circ \phi_{e,s(e)}$.
	\end{lemma}

	\begin{proof}
		If $s(e)\neq s(e')$, this comes from the relation $\phi_{13}\circ \phi_{24}=\phi_{24}\circ \phi_{13}$ as endomorphisms of $D_E\otimes D_E\otimes D_V\otimes D_V$, with immediate notations. If $e\neq e'$ and $s(e)=s(e')$, this comes from the relation $\phi_{13}\circ \phi_{23}=\phi_{23}\circ \phi_{13}$.
	\end{proof}

	\begin{definition}
		Let us assume that $\phi$ is tree-compatible.
		\begin{enumerate}
			\item We define, for any rooted tree $T$, putting $E(T)=\{e_1,\ldots,e_k\}$,
			\[
				\phi_T =\phi_{e_1,s(e_1 )}\circ\ldots \circ \phi_{e_k,s(e_k )}:D(T)\longrightarrow D(T).
			\]
			By Lemma~\ref{lem1.3}, this does not depend on the choice of the order on the edges of $T$.
			\item We shall consider $\Theta_\phi=\bigoplus_{T\in \calT} \phi_T:\calT(D_E,D_V )\longrightarrow \calT(D_E,D_V )$.
		\end{enumerate}
	\end{definition}

	\begin{example}
		Let $a_1,a_2\in D_E$ and $b_1,b_2,b_3\in D_V$.
		\begin{align*}
			\Theta_\phi\left(
			\begin{array}{c}\xymatrix{\boite{b_2}\ar@{-}[d]|{a_1}\\ \boite{b_1}}
			\end{array}
			\right)
			&=\sum_i
			\begin{array}{c}
				\xymatrix{\boite{b_2}\ar@{-}[d]|{\phi_E^i (a_1 )}\\ \boite{\phi_V^i (b_1 )}}
			\end{array}
			,\\[4mm]
			\Theta_\phi\left(
			\begin{array}{c}\xymatrix{\boite{b_2}\ar@{-}[dr]|{a_1}&&\boite{b_3}\ar@{-}[dl]|{a_2}\\ &\boite{b_1}&}
			\end{array}
			\right)
			&=\sum_i\sum_j
			\begin{array}{c}
				\xymatrix{\boite{b_2}\ar@{-}[dr]|{\phi_E^i (a_1 )}&&\boite{b_3}\ar@{-}[dl]|{\phi_E^j (a_2 )}\\ &\boite{\phi_V^i\circ \phi_V^j (b_1 )}&}
			\end{array}
			,\\[4mm]
			\Theta_\phi\left(
			\begin{array}{c}\xymatrix{\boite{b_3}\ar@{-}[d]|{a_1}\\ \boite{b_2} \ar@{-}[d]|{a_2}\\ \boite{b_1}}
			\end{array}
			\right)
			&=\sum_i\sum_j
			\begin{array}{c} \xymatrix{\boite{b_3}\ar@{-}[d]|{\phi_E^i (a_1 )}\\ \boite{\phi_V^i (b_2 )} \ar@{-}[d]|{\phi_E^j (a_2 )}\\ \boite{\phi_V^j (b_1 )}}
			\end{array}
			.
		\end{align*}
	\end{example}

	\section{Operations on tree-compatible maps}

	\subsection{Direct sums}

	\begin{proposition}\label{prop2.1}
		Let $\phi^1:D_E^1\otimes D_V^1\to D_E^1\otimes D_V^1$ and $\phi^2:D_E^1\otimes D_V^2\to D_E^2\otimes D_V^2$ be two maps, and let $\lambda,\mu\in \K$. We define an endomorphism $\Phi=\phi^1\oplus_{\lambda,\mu} \phi^2$ of $(D_E^1\oplus D_E^2 )\otimes (D_V^1\oplus D_V^2 )$ by the following: for any $a_i\in D_E^i$, $b_i\in D_V^i$, with $i=1$ or $2$,
		\begin{align*}
			\Phi(a_1\otimes b_1 )&=\phi^1 (a_1\otimes b_1 ),&\Phi(a_2\otimes b_2 )&=\phi^2 (a_2\otimes b_2 ),\\
			\Phi(a_1\otimes b_2 )&=\lambda a_1\otimes b_2,&\Phi(a_2\otimes b_1 )&=\mu a_2\otimes b_1.
		\end{align*}
		Then $\Phi$ is tree-compatible if, and only if, $\phi^1$ and $\phi^2$ are tree-compatible.
	\end{proposition}

	\begin{proof}
		$\Longrightarrow$. Note that $\Phi_{\mid D_E^i\otimes D_V^i}=\phi^i$. Therefore, by restriction, if $\Phi$ is tree-compatible, then $\phi^i$ is compatible for $i=1$ or $i=2$.

		$\Longleftarrow$. Let $a_i,a'_i\in D_E^i$ and $b_i\in D_V^i$ for $i\in \llbracket 1;2\rrbracket$. Let us show that
		\[
    		\Phi_{23}\circ \Phi_{13}(a_i\otimes a'_j\otimes b_k )=\Phi_{13}\circ \Phi_{23}(a_i\otimes a'_j\otimes b_k )
		\]
		for $(i,j,k)\in \{1,2\}^3$.
		\begin{itemize}
			\item If $(i,j,k)=(1,1,1)$ or $(2,2,2)$, this comes from the tree-compatibility of $\phi^1$ and $\phi^2$.
			\item If $(i,j,k)=(1,2,1)$, as $\Phi_{13}(a_1\otimes a'_2\otimes b_1 )\in D_E^1\otimes D_E^2\otimes D_E^1$, we obtain
			\begin{align*}
				\Phi_{13}\circ \Phi_{23}(a_1\otimes a'_2\otimes b_1 )&=\mu \Phi_{13}(a_1\otimes a'_2\otimes b_1 ),\\
				\Phi_{23}\circ \Phi_{13}(a_1\otimes a'_2\otimes b_1 )&=\mu \Phi_{13}(a_1\otimes a'_2\otimes b_1 ).
			\end{align*}
			The computation is similar for $(i,j,k)=(2,1,1)$, $(2,1,2)$, and $(1,2,2)$.
			\item If $(i,j,k)=(1,1,2)$, we obtain
			\begin{align*}
				\Phi_{13}\circ \Phi_{23}(a_1\otimes a'_1\otimes b_2 )&=\lambda \mu a_1\otimes a'_1\otimes b_2,\\
				\Phi_{23}\circ \Phi_{13}(a_1\otimes a'_1\otimes b_2 )&=\lambda \mu a_1\otimes a'_1\otimes b_2.
			\end{align*}
			The computation is similar for $(i,j,k)=(2,2,1)$. \qedhere
		\end{itemize}
	\end{proof}

	\subsection{Sums and compositions}

	\begin{lemma}\label{lem2.2}
		\begin{enumerate}
			\item $\id_{D_E\otimes D_V}$ is tree-compatible. Moreover, $\Theta_{\id_{D_E\otimes D_V}}=\id_{\calT(D_E,D_V )}$.
			\item Let $\phi:D_E\otimes D_V\longrightarrow D_E\otimes D_V$ be a~bijection. If $\phi$ is tree-compatible, then $\phi^{-1}$ is tree-compatible. Moreover, $\Theta_\phi^{-1}=\Theta_{\phi^{-1}}$.
		\end{enumerate}
	\end{lemma}

	\begin{proof}
		Direct verifications.
	\end{proof}

	\begin{proposition}\label{prop2.3}
		Let $\phi,\psi:D_E\otimes D_V\longrightarrow D_E\otimes D_V$ be two tree-compatible maps, such that
		\[
			\phi_{13}\circ \psi_{23}=\psi_{23}\circ \phi_{13}.
		\]
		Then $\phi\circ \psi$ is tree-compatible and for any $(\alpha,\beta)\in \K^2$, $\alpha\phi+\beta\psi$ is tree-compatible.
	\end{proposition}

	\begin{proof}
		Let us first prove that $\psi_{13}\circ \phi_{23}=\phi_{23}\circ \psi_{13}$. We shall use the map
		\[
			\tau:\left\{
			\begin{array}{rcl}
				D_E\otimes D_E\otimes D_V &\longrightarrow&D_E\otimes D_E\otimes D_V\\
				a\otimes a'\otimes b&\longmapsto&a'\otimes a\otimes b.
			\end{array}
			\right.
		\]
		Then
		\begin{align*}
			\phi_{13}\circ \psi_{23}&=\tau\circ \phi_{23}\circ \tau\circ \tau\circ \psi_{13}\circ \tau\\
			&=\tau\circ \phi_{23}\circ \psi_{13}\circ \tau\\
			&=\tau\circ \psi_{13}\circ \phi_{23}\circ \tau\\
			&=\tau\circ \psi_{13}\circ \tau\circ \tau\circ \phi_{23}\circ \tau=\psi_{23}\circ \phi_{13}.
		\end{align*}
		Therefore,
		\begin{align*}
			(\phi\circ \psi)_{13}\circ (\phi\circ \psi)_{23}&=\phi_{13}\circ \psi_{13}\circ \phi_{23}\circ \psi_{23}\\
			&=\phi_{13}\circ \phi_{23}\circ \psi_{13}\circ \psi_{23}\\
			&=\phi_{23}\circ \phi_{13}\circ \psi_{23}\circ \psi_{13}\\
			&=\phi_{23}\circ \psi_{23}\circ \phi_{13}\circ \psi_{13}=(\phi\circ \psi)_{23}\circ (\phi\circ \psi)_{13}.
		\end{align*}
		So $\phi\circ \psi$ is tree-compatible.

		If $\Phi=\alpha\phi+\beta\psi$,
		\begin{align*}
			\Phi_{13}\circ \Phi_{23}&=\alpha^2 \phi_{13}\circ \phi_{23}+\alpha\beta \phi_{13}\circ \psi_{23}+\alpha\beta \psi_{13}\circ \phi_{23}+\beta^2 \psi_{13}\circ \psi_{23}\\
			&=\alpha^2 \phi_{23}\circ \phi_{13}+\alpha\beta \psi_{23}\circ \phi_{13}+\alpha\beta \phi_{23}\circ \psi_{13}+\beta^2 \psi_{23}\circ \psi_{13}=\Phi_{23}\circ \Phi_{13}. \qedhere
		\end{align*}
	\end{proof}

	\begin{proposition}\label{prop2.4}
		Let $\phi$ be tree-compatible map.
		\begin{enumerate}
			\item For any $P\in \K[X]$, $P(\phi)$ is tree-compatible.
			\item Let us assume that $\phi$ is locally nilpotent, that is to say
			\begin{align*}
				&\forall x\in D_E\otimes D_V,\: \exists n\geq 1,\: \phi^n (x)=0.
			\end{align*}
			Then, for any $P\in \K[[X]]$, $P(\phi)$ is tree-compatible.
		\end{enumerate}
	\end{proposition}

	\begin{proof}
		The tree-compatibility of $\phi$ implies that for any $P,Q\in \K[X]$,
		\[
			P(\phi)_{13}\circ Q(\phi)_{23}=P(\phi_{13})\circ Q(\phi_{23})=Q(\phi_{23})\circ P(\phi_{13})=Q(\phi)_{23}\circ P(\phi)_{13}.
		\]
		In particular, for $P=Q$, $P(\phi)$ is tree-compatible.

		If $\phi$ is locally nilpotent and $P\in \K[[X]]$, as $P(\phi)$ is equal to a~polynomial in $\phi$ when applied to any element of $D_E\otimes D_V$, the first point implies the result for $P(\phi)$.
	\end{proof}

	\subsection{Tensor product}

	\begin{proposition}
		Let $\phi:D_E\otimes D_V\to D_E\otimes D_V$ and $\phi':D'_E\otimes D'_V\to D'E\otimes D'_V$ be two tree-compatible maps. Then $\Phi=\phi_{13}\circ \phi'_{24}:(D_E\otimes D_E ')\otimes (D_V\otimes D_V ')\to (D_E\otimes D_E ')\otimes (D_V\otimes D_V ')$ is tree-compatible. With Sweedler's notation,
		\[
			\Phi(a\otimes a'\otimes b\otimes b')=\sum_i\sum_j \phi_E^i (a)\otimes \phi'^j_E (a')\otimes \phi_V^i (b)\otimes \phi'^j_V (b').
		\]
	\end{proposition}

	\begin{proof}
		In this case, $\Phi_{13}=\phi_{15}\circ \psi_{26}$, and $\Phi_{23}=\phi_{35}\circ \psi_{46}$. Moreover, $\phi_{15}$, $\psi_{26}$, $\phi_{35}$ and $\psi_{46}$ pairwise commute, either as they do not act on the same factors of the tensor product $(D_E\otimes D_E ')^{\otimes 2}\otimes (D_V\otimes D_V ')$, or because of the tree-compatibility of $\phi$ and $\psi$ (for $\phi_{15}$ and $\phi_{35}$, $\psi_{26}$ and $\psi_{46}$). So $\Phi_{13}$ and $\Phi_{23}$ commute.
	\end{proof}

	\section{Pre- and post-Lie structures associated to a~tree-compatible map}

	\subsection{Multiple pre-Lie algebras}
	\begin{definition}[see \cite{Bruned2023-1,Foissy47,Zhang2020}]
		A $D_E$-multiple (or $D_E$-matching, or shortly, multiple) pre-Lie algebra is a~pair $(V,\triangleleft)$ where $V$ is a~vector space and
		\[
			\triangleleft:\left\{
			\begin{array}{rcl}
				D_E\otimes V\otimes V&\longrightarrow&V\\
				a\otimes x\otimes y&\longmapsto&x\triangleleft_a y
			\end{array}
			\right.
		\]
		is a~linear map such that
		\begin{align*}
			x\triangleleft_a (y\triangleleft_{a'}z)-(x\triangleleft_a y)\triangleleft_{a'}z&=y\triangleleft_{a'}(x\triangleleft_a z)-(y\triangleleft_{a'} x)\triangleleft_a z
		\end{align*}
        for all $a,a'\in D_E$ and $x,y,z\in V$. In particular, for any $a\in D_E$, $(V,\triangleleft_a )$ is a~pre-Lie algebra.
	\end{definition}
	Let us recall the following Lemma, proved in~\cite[Proposition 2.4]{Foissy48} in a~more general setting.

	\begin{lemma}\label{lem3.2}
		Let $(V,\triangleleft)$ be a~$D_E$-multiple pre-Lie algebra. Then $D_E\otimes V$ is a~pre-Lie algebra, with the product defined by
		\begin{align*}
			&\forall a,a'\in D_E,\: \forall x,y\in V,& a\otimes x\triangleleft a'\otimes y&=a' \otimes x\triangleleft_a y.
		\end{align*}
	\end{lemma}

	\begin{proof}
		Let $a,a',a''\in D_E$ and $x,y,z\in V$.
		\begin{align*}
			&a\otimes x\triangleleft (a'\otimes y\triangleleft a''\otimes z)-(a\otimes x\triangleleft a'\otimes y)\triangleleft a''\otimes z\\
			&=a\otimes x \triangleleft (a'' \otimes y\triangleleft_{a'} z)-(a'\otimes x\triangleleft_a y)\triangleleft a''\otimes z\\
			&=a'' \otimes \left(x\triangleleft_a (y\triangleleft_{a'}z)-(x\triangleleft_a y)\triangleleft_{a'}z\right),
		\end{align*}
		whereas
		\begin{align*}
			&a'\otimes y\triangleleft (a\otimes x\triangleleft a''\otimes z)-(a'\otimes y\triangleleft a\otimes z)\triangleleft a''\otimes z\\
			&=a'' \otimes \left(y\triangleleft_{a'}(x\triangleleft_a z)-(y\triangleleft_{a'} x)\triangleleft_a z\right).
		\end{align*}
		Therefore, the $D_E$-multiple relation for $V$ implies (and in fact, is equivalent to) the pre-Lie relation for $D_E\otimes V$.
	\end{proof}

	\begin{notation}
		Let $T,T'\in \calT$ and $v\in V(T)$. We denote by $T \curvearrowright_v T'$ the rooted tree obtained by grafting $T$ and the vertex $v$ of $T'$, creating a~new edge $e$ with $s(e)=v$ and $t(e)$ equal to the root of $T$.
	\end{notation}
	From~\cite[Corollary 2.7]{Foissy47} and~\cite[Theorem 2.2]{Foissy48} for a~more general setting, the free $D_E$-multiple pre-Lie algebra generated by $D_V$ is, as a~vector space, $\calT(D_E,D_V )$, with the product $\triangleleft$ defined by
	\begin{align*}
		&\left(T\otimes \bigotimes_{e\in E(T)} a_e\otimes \bigotimes_{v\in V(T)}b_v \right)\triangleleft_a\left(T'\otimes \bigotimes_{e\in E(T')} a'_e\otimes \bigotimes_{v\in V(T')}b'_v\right)\\
		&=\sum_{v\in V(T)} T\curvearrowright_v T'\otimes a\otimes \bigotimes_{e\in E(T)} a_e\otimes \bigotimes_{e\in E(T')} a'_e\otimes \bigotimes_{v\in V(T)}b_v\otimes \bigotimes_{v\in V(T')}b'_v,
	\end{align*}
	where the factor $a$ in this tensor is understood as the decoration of the edge created when $T$ is grafted on the vertex $v$ of $T'$. In a~more compact way,
	\[
		T\otimes W_T \triangleleft_a T'\otimes W_{T'}=\sum_{v\in V(T)} T\curvearrowright_v T' \otimes a\otimes W_T\otimes W_{T'}.
	\]

	\begin{example}
		If $a,a_1,a_2,a_3\in D_E$ and $b_1,b_2,b_3,b_4,b_5\in D_V$,
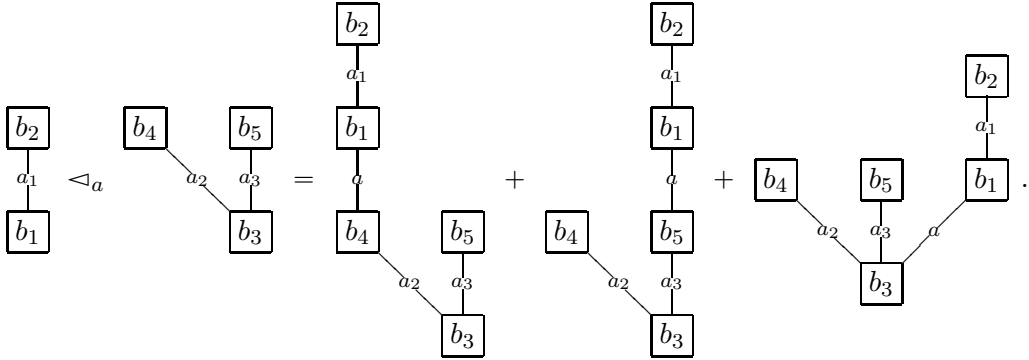
		\begin{align*}
			\begin{array}{c}\xymatrix{\boite{b_2}\ar@{-}[d]|{a_1}\\ \boite{b_1}}
			\end{array}
			\triangleleft_a
			\begin{array}{c}\xymatrix{\boite{b_4}\ar@{-}[rd]|{a_2}&\boite{b_5}\ar@{-}[d]|{a_3}\\&\boite{b_3}}
			\end{array}
			&=
			\begin{array}{c}\xymatrix{\boite{b_2}\ar@{-}[d]|{a_1}&\\ \boite{b_1}\ar@{-}[d]|{a}&\\\boite{b_4}\ar@{-}[rd]|{a_2}&\boite{b_5}\ar@{-}[d]|{a_3}\\&\boite{b_3}}
			\end{array}
			+
			\begin{array}{c}\xymatrix{&\boite{b_2}\ar@{-}[d]|{a_1}\\&\boite{b_1}\ar@{-}[d]|{a}\\\boite{b_4}\ar@{-}[rd]|{a_2}&\boite{b_5}\ar@{-}[d]|{a_3}\\&\boite{b_3}}
			\end{array}
			+
			\begin{array}{c}\xymatrix{&&\boite{b_2}\ar@{-}[d]|{a_1}\\\boite{b_4}\ar@{-}[rd]|{a_2}&\boite{b_5}\ar@{-}[d]|{a_3}&\boite{b_1}\ar@{-}[ld]|{a}\\&\boite{b_3}&}
			\end{array}
			.
		\end{align*}
	\end{example}

	\begin{theorem}\label{theo3.3}
		Let $\phi:D_E\otimes D_V\longrightarrow D_E\otimes D_V$ be a~map. We then define a~map
		\[
			\triangleleft^\phi:D_E\otimes \calT(D_E,D_V )\otimes \calT(D_E,D_V )\longrightarrow \calT(D_E,D_V )
		\]
		by
		\begin{align*}
			T\otimes W_T\triangleleft^\phi_a T'\otimes W_{T'}=\sum_{v\in V(T)} \sum T\curvearrowright_v T'\otimes \phi_{1,v}\left(a\otimes W_T\otimes W_{T'}\right).
		\end{align*}
		Here, $\phi_{1,v}$ means that $\phi$ acts on $a\otimes d'_v$. Then $(\calT(D_E,D_V ),\triangleleft^\phi)$ is a~$D_E$-multiple pre-Lie algebra if, and only if, $\phi$ is tree-compatible.
	\end{theorem}

	\begin{example}
		If $a,a_1,a_2,a_3\in D_E$ and $b_1,b_2,b_3,b_4,b_5\in D_V$,
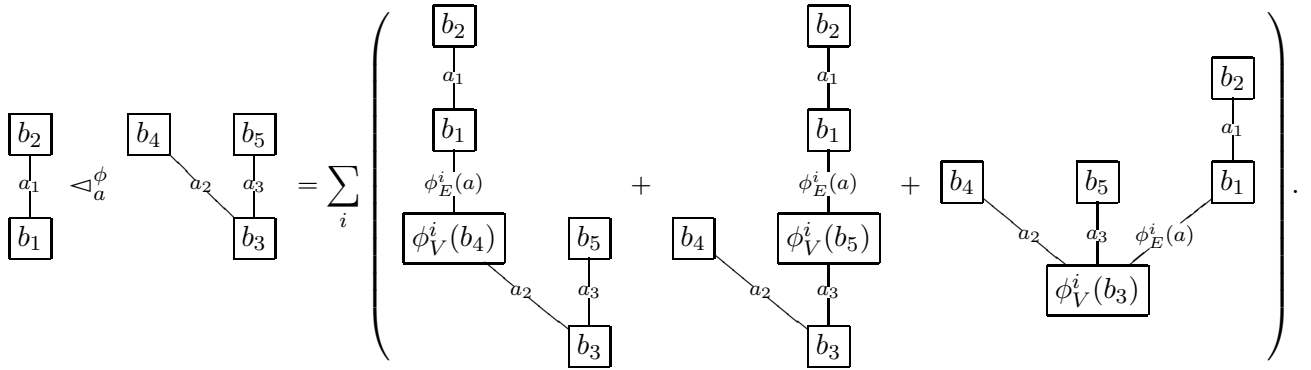
		\begin{align*}
			\begin{array}{c}\xymatrix{\boite{b_2}\ar@{-}[d]|{a_1}\\ \boite{b_1}}
			\end{array}
			{ }&{ }\triangleleft^\phi_a
			\begin{array}{c}\xymatrix{\boite{b_4}\ar@{-}[rd]|{a_2}&\boite{b_5}\ar@{-}[d]|{a_3}\\&\boite{b_3}}
			\end{array} \\
			&=\sum_i\left(
			\begin{array}{c}\xymatrix{\boite{b_2}\ar@{-}[d]|{a_1}&\\ \boite{b_1}\ar@{-}[d]|{\phi_E^i (a)}&\\\boite{\phi_V^i (b_4 )}\ar@{-}[rd]|{a_2}&\boite{b_5}\ar@{-}[d]|{a_3}\\&\boite{b_3}}
			\end{array}
			+
			\begin{array}{c}\xymatrix{&\boite{b_2}\ar@{-}[d]|{a_1}\\&\boite{b_1}\ar@{-}[d]|{\phi_E^i (a)}\\\boite{b_4}\ar@{-}[rd]|{a_2}&\boite{\phi_V^i (b_5 )}\ar@{-}[d]|{a_3}\\&\boite{b_3}}
			\end{array}
			+
			\begin{array}{c}\xymatrix{&&\boite{b_2}\ar@{-}[d]|{a_1}\\\boite{b_4}\ar@{-}[rd]|{a_2}&\boite{b_5}\ar@{-}[d]|{a_3}&\boite{b_1}\ar@{-}[ld]|{\phi_E^i (a)}\\&\boite{\phi_V^i (b_3 )}&}
			\end{array}
			\right).
		\end{align*}
	\end{example}

	\begin{proof}
		$\Longrightarrow$. Let us fix nonzero $b,b'\in D_V$. Let $a,a'\in D_E$ and $b''\in D_V$.
		\begin{align*}
			&\xymatrix{\boite{b}}\triangleleft^\phi_a\left(\xymatrix{\boite{b'}}\triangleleft^\phi_{a'}\xymatrix{\boite{b''}}\right)
			-\left(\xymatrix{\boite{b}}\triangleleft^\phi_a\xymatrix{\boite{b'}}\right)\triangleleft^\phi_{a'}\xymatrix{\boite{b''}}\\
			&=\xymatrix{\boite{b}}\triangleleft^\phi_a\left(\sum
			\begin{array}{c}\xymatrix{\boite{b'}\ar@{-}[d]|{\phi_E^i (a')}\\ \boite{\phi_V^i (b'')}}
			\end{array}
			\right)
			-\left(\sum_i
			\begin{array}{c}\xymatrix{\boite{b}\ar@{-}[d]|{\phi_E^i (a)}\\ \boite{\phi_V^i (b')}}
			\end{array}
			\right)\triangleleft^\phi_{a'}\xymatrix{\boite{b''}}\\
			&=\sum_i\sum_j
			\begin{array}{c}\xymatrix{\boite{b}\ar@{-}[dr]|{\phi_E^j (a)}&&\boite{b'}\ar@{-}[ld]|{\phi_E^i (a')}\\&\boite{\phi_V^j\circ \phi_V^i (b'')}&}
			\end{array}
			+\sum_i\sum_j
			\begin{array}{c}\xymatrix{\boite{b}\ar@{-}[d]|{\phi^j_E (a)}\\ \boite{\phi^j_V (b')} \ar@{-}[d]|{\phi^i_E (a')}\\ \boite{\phi_V^i (b'')}}
			\end{array}
			-\sum_i\sum_j
			\begin{array}{c}\xymatrix{\boite{b}\ar@{-}[d]|{\phi^i_E (a)}\\ \boite{\phi^i_V (b')} \ar@{-}[d]|{\phi^j_E (a')}\\ \boite{\phi_V^j (b'')}}
			\end{array}
			\\
			&=\sum_i\sum_j
			\begin{array}{c}\xymatrix{\boite{b}\ar@{-}[dr]|{\phi_E^j (a)}&&\boite{b'}\ar@{-}[ld]|{\phi_E^i (a')}\\&\boite{\phi_V^j\circ \phi_V^i (b'')}&}
			\end{array}
			.
		\end{align*}
		By symmetry,
		\begin{align*}
			\xymatrix{\boite{b'}}\triangleleft^\phi_{a'}\left(\xymatrix{\boite{b}}\triangleleft^\phi_a\xymatrix{\boite{b''}}\right)
			-\left(\xymatrix{\boite{b'}}\triangleleft^\phi_{a'}\xymatrix{\boite{b}}\right)\triangleleft^\phi_a\xymatrix{\boite{b''}}
			&=\sum_i\sum_j
			\begin{array}{c}\xymatrix{\boite{b}\ar@{-}[dr]|{\phi_E^i (a)}&&\boite{b'}\ar@{-}[ld]|{\phi_E^j (a')}\\&\boite{\phi_V^j\circ \phi_V^i (b'')}&}
			\end{array}
			.
		\end{align*}
		We obtain from the multiple pre-Lie relation, identifying the decorations of the vertices and of the edges, that
		\[
			\phi_{15}\circ \phi_{25}(a\otimes a' \otimes b\otimes b'\otimes b'')=\phi_{25}\circ \phi_{15}(a\otimes a' \otimes b\otimes b'\otimes b'').
		\]
		As $b,b'\neq 0$, this gives $\phi_{13}\circ\phi_{23}=\phi_{23}\circ \phi_{13}$, so $\phi$ is tree-compatible.

		$\Longleftarrow$. Let $x=T\otimes W_T$, $x'=T'\otimes W_{T'}$ and $x''=T''\otimes W_{T''}$ be elements of $\calT(D_E,D_T )$. Let $a,a'\in D_E$.
		\begin{align*}
			&x\triangleleft^\phi_a\left(x' \triangleleft^\phi_{a'} x''\right)-\left( x\triangleleft^\phi_a x'\right) \triangleleft^\phi_{a'} x''\\
			&=\sum_{v'' \in V(T'')} x\triangleleft^\phi_a T'\curvearrowright_{v''} T''\otimes \phi_{1,v''}(a'\otimes W_{T'}\otimes W_{T''})\\
			&-\sum_{v'\in V(T')} T\curvearrowright_{v'}T'\otimes \phi_{1,v'}(a\otimes W_T\otimes W_{T'})\triangleleft^\phi_{a'}x''\\
			&=\sum_{v''_1,v''_2 \in V(T'')}T\curvearrowright_{v''_1}(T'\curvearrowright_{v''_2} T'')\otimes \phi_{1,v''_1}\circ \phi_{2,v''_2}(a\otimes a' \otimes W_T\otimes W_{T'}\otimes W_{T''})\\
			&+\sum_{\substack{v''\in V(T''),\\ v'\in V(T')}}T\curvearrowright_{v'}(T'\curvearrowright_{v''} T'')\otimes \phi_{1,v'}\circ \phi_{2,v''}(a\otimes a' \otimes W_T\otimes W_{T'}\otimes W_{T''})\\
			&-\sum_{\substack{v''\in V(T''),\\ v'\in V(T')}}(T\curvearrowright_{v'}T')\curvearrowright_{v''} T''\otimes \phi_{2,v''}\circ \phi_{1,v''}(a\otimes a' \otimes W_T\otimes W_{T'}\otimes W_{T''}).
		\end{align*}
		Let $v'\in V(T')$ and $v''\in V(T'')$. Then the two following rooted trees are equal:
		\begin{align*}
			T_1 &=T\curvearrowright_{v'}(T'\curvearrowright_{v''} T''),&T_2 &=(T\curvearrowright_{v'}T')\curvearrowright_{v''} T''.
		\end{align*}
		Moreover, $\phi_{1,v'}$ and $\phi_{2,v''}$ commute, are they apply on different factors of the tensor product $D_{T_1}=D_{T_2}$. Therefore,
		\begin{align*}
			&x\triangleleft^\phi_a\left(x' \triangleleft^\phi_{a'} x''\right)-\left( x\triangleleft^\phi_a x'\right) \triangleleft^\phi_{a'} x''\\
			&=\sum_{v''_1,v''_2 \in V(T'')}T\curvearrowright_{v''_1}(T'\curvearrowright_{v''_2} T'')\otimes \phi_{1,v''_1}\circ \phi_{2,v''_2}(a\otimes a' \otimes W_T\otimes W_{T'}\otimes W_{T''}).
		\end{align*}
		Similarly,
		\begin{align*}
			&x'\triangleleft^\phi_{a'}\left(x \triangleleft^\phi_a x''\right)-\left( x'\triangleleft^\phi_{a'}x\right) \triangleleft^\phi_a x''\\
			&=\sum_{v''_1,v''_2 \in V(T'')}T'\curvearrowright_{v''_2}(T\curvearrowright_{v''_1} T'')\otimes \phi_{2,v''_2}\circ \phi_{1,v''_1}(a\otimes a' \otimes W_T\otimes W_{T'}\otimes W_{T''}).
		\end{align*}
		Let $v''_1,v''_2\in V(T'')$. The rooted tree $T_1 =T\curvearrowright_{v''_1}(T'\curvearrowright_{v''_2} T'')$ is equal to the rooted tree $T_2 =T'\curvearrowright_{v''_2}(T\curvearrowright_{v''_1} T'')$. Moreover:
		\begin{itemize}
			\item If $v''_1\neq v''_2$, $\phi_{1,v''_1}$ and $\phi_{2,v''_2}$ commute, as they apply on different factors of the tensor product $D_{T_1}=D_{T_2}$.
			\item If $v''_1 =v''_2$, as $\phi$ is tree-compatible, $\phi_{1,v''_1}$ and $\phi_{2,v''_1}$ commute.
		\end{itemize}
		As a~conclusion, in any case, $ \phi_{1,v''_1}\circ \phi_{2,v''_2}=\phi_{2,v''_2}\circ \phi_{1,v''_1}$. We obtain
		\begin{align*}
			x\triangleleft^\phi_a\left(x' \triangleleft^\phi_{a'} x''\right)-\left( x\triangleleft^\phi_a x'\right) \triangleleft^\phi_{a'} x''&=x'\triangleleft^\phi_{a'}\left(x \triangleleft^\phi_a x''\right)-\left( x'\triangleleft^\phi_{a'}x\right) \triangleleft^\phi_a x''.
		\end{align*}
		So $(\calT(D_E,D_D ),\triangleleft^\phi)$ is $D_E$-multiple pre-Lie.
	\end{proof}

	\begin{proposition}
		Let $\phi,\psi$ be two tree-compatible maps on $D_E\otimes D_V$ such that
		\[
			\psi_{13}\circ \phi_{23}=\phi_{23}\circ \psi_{13}.
		\]
		Then $\Theta_\phi: (\calT(D_E,D_V ),\triangleleft^\psi) \to (\calT(D_E,D_V ),\triangleleft^{\phi\circ \psi})$ is a~$D_E$-multiple pre-Lie algebra map.
	\end{proposition}

	\begin{proof}
		By Proposition~\ref{prop2.3}, $\phi\circ \psi$ is tree-compatible. Let us prove that $\Theta_\phi$ is a~multiple pre-Lie morphism. Let $x=T\otimes W_T$ and $x'=T'\otimes W_{T'}$ be two elements of $\calT(D_E,D_V )$, and let $a\in D_E$.
		\begin{align*}
			\Theta_\phi(x \triangleleft^\psi_a x')
			&=\sum_{v\in V(T')} \Theta_\phi(T\curvearrowright_v T'\otimes \psi_{1,v}(a\otimes W_T\otimes W_{T'}))\\
			&=\sum_{v\in V(T')}T\curvearrowright_v T'\otimes \prod^\circ_{e\in E(T)} \phi_{e,s(e)}\circ \prod^\circ_{e'\in E(T')}\phi_{e',s(e')} \circ \phi_{1,v}\circ \psi_{1,v}(a\otimes W_T \otimes W_{T'})\\
			&=\sum_{v\in V(T')}T\curvearrowright_v T'\otimes\phi_{1,v}\circ \psi_{1,v}\circ \prod^\circ_{e\in E(T)} \phi_{e,s(e)}\circ \prod^\circ_{e'\in E(T')}\phi_{e',s(e')} (a\otimes W_T \otimes W_{T'})\\
			&=\sum_{v\in V(T')}T\curvearrowright_v T'\otimes(\phi\circ \psi)_{1,v}\left(a\otimes \prod^\circ_{e\in E(T)} \phi_{e,s(e)}(W_T ) \otimes \prod^\circ_{e'\in E(T')}\phi_{e',s(e')}(W_{T'})\right)\\
			&=\Theta_\phi(x)\triangleleft^{\phi\circ \psi}_a \Theta_\phi(x').
		\end{align*}
		The symbols $\displaystyle\prod^\circ$ have to be understood as compositions, without any constraint on the order by the tree-compatibility of $\phi$. The commutation on the third equality comes from the fact that $\phi$ is tree-compatible (for the commutation of $\phi_{1,v}$) and from the hypothesis on $\phi$ and $\psi$ (for the commutation of $\psi_{1,v}$). So $\Theta_\phi$ is a~$D_E$-multiple pre-Lie algebra morphism from $(\calT(D_E,D_V ),\triangleleft^\psi)$ to $(\calT(D_E,D_V ),\triangleleft^{\phi\circ\psi})$.
	\end{proof}

	In the particular case where $\psi=\id_{D_E\otimes D_V}$, which is tree-compatible by Lemma~\ref{lem2.2}, note that $\triangleleft^{\id_{D_E\otimes D_V}}=\triangleleft$. We obtain:

	\begin{corollary}\label{cor3.5}
		\begin{enumerate}
			\item The map $\Theta_\phi: (\calT(D_E,D_V ),\triangleleft) \to (\calT(D_E,D_V ),\triangleleft^\phi)$ is a~multiple pre-Lie algebra morphism.
            
			\item If $\phi$ is a~bijective tree-compatible map, then 
            \[
                \Theta_\phi:(\calT(D_E,D_V ),\triangleleft)\to(\calT(D_E,D_V ),\triangleleft^\phi)
            \]
            is an isomorphism, of inverse $\Theta_{\phi^{-1}}$.
		\end{enumerate}
	\end{corollary}

	\begin{corollary}\label{cor3.6}
		Let $\phi$ be tree-compatible. Then $(\calT(D_E,D_V ),\triangleleft^\phi)$ is generated by the trees $\tun\otimes b=\xymatrix{\boite{b}}$, with $b\in D_V$, if and only if, $\phi$ is surjective.
	\end{corollary}

	\begin{proof}
		Let us fix $b'\neq 0$ in $D_V$.

		$\Longrightarrow$. Let $a\in D_E$ and $b\in D_V$, and let us consider
		\[
			x=\tdeux\otimes a\otimes b\otimes b'=
			\begin{array}{c}\xymatrix{\boite{b'}\ar@{-}[d]|{a}\\\boite{b}}
			\end{array}
			.
		\]
		As $(\calT(D_E,D_V ),\triangleleft^\phi)$ is generated by the trees with one vertex, by homogeneity, one can write $\tdeux\otimes a\otimes b\otimes b'$ under the form
		\[
			\begin{array}{c}\xymatrix{\boite{b'}\ar@{-}[d]|{a}\\\boite{b}}
			\end{array}
			=\sum_{k=1}^n \xymatrix{\boite{b'_k}}\triangleleft^\phi_{a_k} \xymatrix{\boite{b_k}}=\sum_{k=1}^n\sum_i
			\begin{array}{c}\xymatrix{\boite{b'_k}\ar@{-}[d]|{\phi_E^i (a_k )}\\\boite{\phi_V^i (b_k )}}
			\end{array}
			.
		\]
		This gives
		\[
			a\otimes b\otimes b'=\sum_{k=1}^n \sum_i \phi_E^i (a_k )\otimes \phi_V^i (b_k )\otimes b'_k =\sum_{k=1}^n (\phi\otimes \id_{D_V})(a_k\otimes b_k\otimes b'_k ).
		\]
		Let $\lambda \in D_V^*$, such that $\lambda(b')=1$. Applying $\id_{D_E}\otimes \id_{D_V}\otimes \lambda$, we obtain
		\[
			a\otimes b=\sum_{k=1}^n \lambda(b'_k ) \phi(a_k\otimes b_k ),
		\]
		so $a\otimes b \in \im(\phi)$: $\phi$ is surjective.

		$\Longleftarrow$. Then $\Theta_\phi$ is surjective. As it is a~surjective $D_E$-multiple pre-Lie morphism from $(\calT(D_E,D_V ),\triangleleft)$ to $(\calT(D_E,D_V ),\triangleleft^\phi)$ sending $\xymatrix{\boite{b}}$ to itself for any $b\in D_V$, and as $(\calT(D_E,D_V ),\triangleleft)$ is generated by these elements, so is $(\calT(D_E,D_V ),\triangleleft^\phi)$.
	\end{proof}

	\begin{corollary}\label{cor3.7}
		Let $\phi$ be tree-compatible. Then $(\calT(D_E,D_V ),\triangleleft^\phi)$ is freely generated by the trees $\tun\otimes b=\xymatrix{\boite{b}}$, with $b\in D_V$, if and only if, $\phi$ is bijective.
	\end{corollary}

	\begin{proof}
		$\Longrightarrow$. By Corollary~\ref{cor3.6}, $\phi$ is surjective. The unique $D_E$-multiple pre-Lie morphism from $(\calT(D_E,D_V ),\triangleleft)$ to $(\calT(D_E,D_V ),\triangleleft^\phi)$ sending $\xymatrix{\boite{b}}$ to itself for any $b\in D_V$, is bijective. This morphism is obviously $\Theta_\phi$. Let $x=\sum_{k=1}^n a_k \otimes b_k \in \Ker(\phi)$. Then
		\begin{align*}
			\Theta_\phi\left(\sum_{k=1}^n
			\begin{array}{c}\xymatrix{\boite{b'}\ar@{-}[d]|{a_k}\\\boite{b_k}}
			\end{array}
			\right)&=\sum_i \sum_{k=1}^n
			\begin{array}{c}\xymatrix{\boite{b'}\ar@{-}[d]|{\phi_E^i (a_k )}\\\boite{\phi_V^i (b_k )}}
			\end{array}
			=0.
		\end{align*}
		As $\Theta_\phi$ is an isomorphism, we deduce that
		\[
			\sum_{k=1}^n
			\begin{array}{c}\xymatrix{\boite{b'}\ar@{-}[d]|{a_k}\\\boite{b_k}}
			\end{array}
			=0.
		\]
		As $b'\neq 0$, this gives $\sum_{k=1}^n a_k\otimes b_k =0$, so $\phi$ is injective.

		$\Longleftarrow$. Then $\Theta_\phi$ is an isomorphism from $(\calT(D_E,D_V ),\triangleleft)$ to $(\calT(D_E,D_V ),\triangleleft^\phi)$ sending $\xymatrix{\boite{b}}$ to itself for any $b\in D_V$, and as $(\calT(D_E,D_V ),\triangleleft)$ is freely generated by these elements, so is $(\calT(D_E,D_V ),\triangleleft^\phi)$.
	\end{proof}

	\subsection{Associated pre- and post-Lie algebras}

	Let $\phi$ be tree-compatible. By Lemma~\ref{lem3.2} and Theorem~\ref{theo3.3}, $D_E\otimes \calT(D_E,D_V )$ is a~pre-Lie algebra. Elements $a\otimes T\otimes W_T \in D_E\otimes \calT(D_E,D_V )$, where $T$ is a~decorated rooted tree, are identified with planted tree by adding a~non decorated root to $T$, the created edge being decorated by $a$, as in~\cite{Bruned2023}. For example, if $a_1,a_2,a\in D_E$ and $b_1,b_2,b_3\in D_V$, we shall identify the two elements
	\begin{align*}
		a\otimes
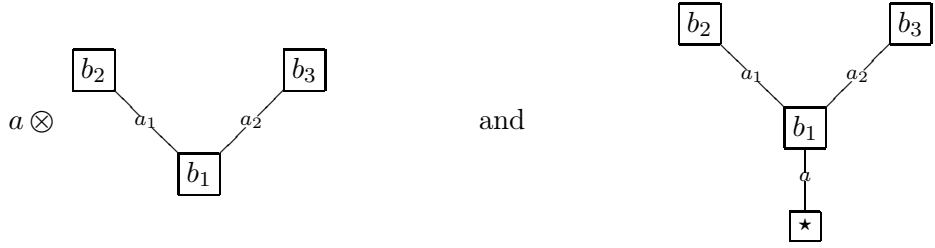
		\begin{array}{c}
			\xymatrix{\boite{b_2}\ar@{-}[rd]|{a_1}&&\boite{b_3}\ar@{-}[ld]|{a_2}\\&\boite{b_1}&}
		\end{array}
		&&\mbox{and}&&
		\begin{array}{c}
			\xymatrix{\boite{b_2}\ar@{-}[rd]|{a_1}&&\boite{b_3}\ar@{-}[ld]|{a_2}\\&\boite{b_1}\ar@{-}[d]|{a}&\\&\boite{\blackdiamond}&}
		\end{array}
	\end{align*}
	The absence of decoration on the added root is symbolized by a~$\blackdiamond$. Let us describe this pre-Lie structure. Let $a,a'\in D_E$, and $x=T\otimes W_T$, $x'=T'\otimes W_{T'}$ two elements of $\calT(D_E,D_V )$. Then
	\begin{align*}
		a\otimes x\triangleleft^\phi a'\otimes x'&=\sum_{v_0\in V(T')}a'\otimes T \curvearrowright_{v_0}T'\otimes \phi_{1,v_0}\left(a\otimes W_T\otimes W_{T'}\right).
	\end{align*}

	\begin{example}
		If $a,a',a_1,a_2,a_3\in D_E$ and $b_1,b_2,b_3,b_4,b_5\in D_V$,
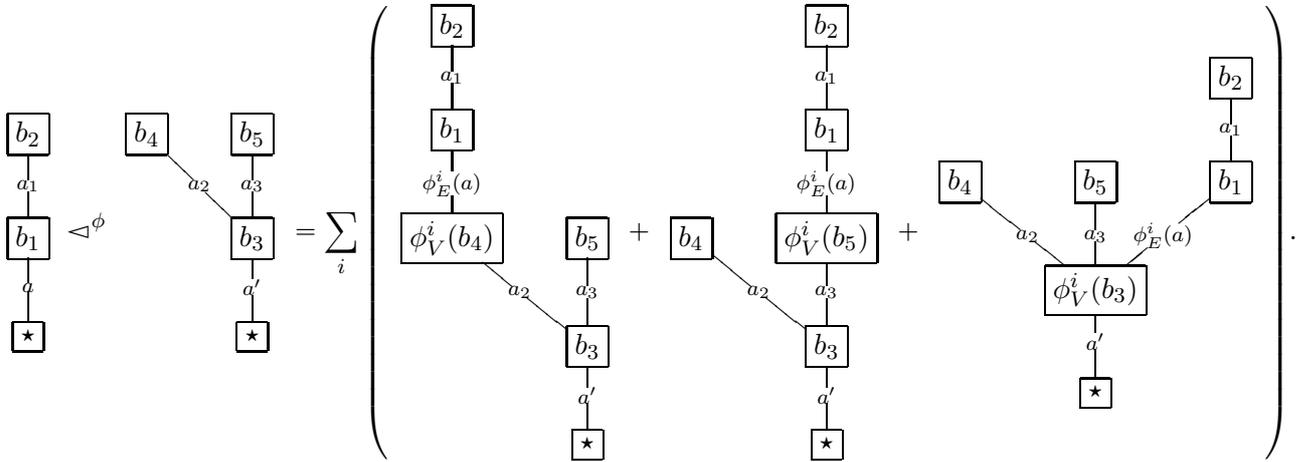
		\begin{align*}
			\begin{array}{c}\xymatrix{\boite{b_2}\ar@{-}[d]|{a_1}\\ \boite{b_1}\\ \boite{\blackdiamond} \ar@{-}[u]|{a}}
			\end{array}
			{ }&{ }\triangleleft^\phi
			\begin{array}{c}\xymatrix{\boite{b_4}\ar@{-}[rd]|{a_2}&\boite{b_5}\ar@{-}[d]|{a_3}\\&\boite{b_3}\\&\boite{\blackdiamond} \ar@{-}[u]|{a'}}
			\end{array} \\
			&=\sum_i\left(
			\begin{array}{c}\xymatrix{\boite{b_2}\ar@{-}[d]|{a_1}&\\ \boite{b_1}\ar@{-}[d]|{\phi_E^i (a)}&\\\boite{\phi_V^i (b_4 )}\ar@{-}[rd]|{a_2}&\boite{b_5}\ar@{-}[d]|{a_3}\\&\boite{b_3}\\&\boite{\blackdiamond} \ar@{-}[u]|{a'}}
			\end{array}
			+
			\begin{array}{c}\xymatrix{&\boite{b_2}\ar@{-}[d]|{a_1}\\&\boite{b_1}\ar@{-}[d]|{\phi_E^i (a)}\\\boite{b_4}\ar@{-}[rd]|{a_2}&\boite{\phi_V^i (b_5 )}\ar@{-}[d]|{a_3}\\&\boite{b_3}\\&\boite{\blackdiamond} \ar@{-}[u]|{a'}}
			\end{array}
			+
			\begin{array}{c}\xymatrix{&&\boite{b_2}\ar@{-}[d]|{a_1}\\\boite{b_4}\ar@{-}[rd]|{a_2}&\boite{b_5}\ar@{-}[d]|{a_3}&\boite{b_1}\ar@{-}[ld]|{\phi_E^i (a)}\\&\boite{\phi_V^i (b_3 )}&\\&\boite{\blackdiamond} \ar@{-}[u]|{a'}&}
			\end{array}
			\right).
		\end{align*}
		In particular, if $\phi$ is the identity of $D_E\otimes D_V$,
		\begin{align*}
			\begin{array}{c}\xymatrix{\boite{b_2}\ar@{-}[d]|{a_1}\\ \boite{b_1}\\ \boite{\blackdiamond} \ar@{-}[u]|{a}}
			\end{array}
			\triangleleft
			\begin{array}{c}\xymatrix{\boite{b_4}\ar@{-}[rd]|{a_2}&\boite{b_5}\ar@{-}[d]|{a_3}\\&\boite{b_3}\\&\boite{\blackdiamond} \ar@{-}[u]|{a'}}
			\end{array}
			&=
			\begin{array}{c}\xymatrix{\boite{b_2}\ar@{-}[d]|{a_1}&\\ \boite{b_1}\ar@{-}[d]|{a}&\\\boite{b_4}\ar@{-}[rd]|{a_2}&\boite{b_5}\ar@{-}[d]|{a_3}\\&\boite{b_3}\\&\boite{\blackdiamond} \ar@{-}[u]|{a'}}
			\end{array}
			+
			\begin{array}{c}\xymatrix{&\boite{b_2}\ar@{-}[d]|{a_1}\\&\boite{b_1}\ar@{-}[d]|{a}\\\boite{b_4}\ar@{-}[rd]|{a_2}&\boite{b_5}\ar@{-}[d]|{a_3}\\&\boite{b_3}\\&\boite{\blackdiamond} \ar@{-}[u]|{a'}}
			\end{array}
			+
			\begin{array}{c}\xymatrix{&&\boite{b_2}\ar@{-}[d]|{a_1}\\\boite{b_4}\ar@{-}[rd]|{a_2}&\boite{b_5}\ar@{-}[d]|{a_3}&\boite{b_1}\ar@{-}[ld]|{a}\\&\boite{b_3}&\\&\boite{\blackdiamond} \ar@{-}[u]|{a'}&}
			\end{array}
			.
		\end{align*}
	\end{example}

	\begin{definition}[see \cite{Vallette2007}] \label{defi3.8}
		A post-Lie algebra is a~triple $(\g,\{-,-\},\triangleleft)$ where $\g$ is a~vector space and $\{-,-\},\triangleleft$ are bilinear products on $\g$ such that
		\begin{align}\label{EQ2}&\forall x,y,z\in \g,&&\{\{x,y\},z\}+\{\{y,z\},x\}+\{\{z,x\},y\}=0.\\
			\label{EQ3}&&&x\triangleleft\{y,z\}=\{x\triangleleft y,z\}+\{y,x\triangleleft z\}.\\
			\label{EQ4}&&&\{x,y\}\triangleleft z=x\triangleleft (y\triangleleft z)-(x\triangleleft y)\triangleleft z-y\triangleleft (x\triangleleft z)+(y\triangleleft x)\triangleleft z.
		\end{align}
		In particular, pre-Lie algebras are post-Lie algebras with a~zero bracket $\{-,-\}$.
	\end{definition}
	We now assume that $(P,\{-,-\},\triangleleft)$ is a~post-Lie algebra. Our aim in this section is to extend this post-Lie structure and the pre-Lie structure of $E\otimes \calT(D_E,D_V )$ to a~post-Lie algebra structure on $\left(E\otimes \calT(D_E,D_V )\right)\oplus P$.

	\begin{proposition}\label{prop3.9}
		Let $P$ be a~post-Lie algebra and let two maps
		\begin{align*}
			\psi_V &:\left\{
			\begin{array}{rcl}
				P&\longrightarrow&\lin(D_V )\\
				p&\longmapsto&\left\{
				\begin{array}{rcl}
				D_V &\longrightarrow&D_V\\
				b&\longmapsto&\psi_V (p)(b),
			\end{array}
			\right.
			\end{array}
			\right.&
			\psi_E &:\left\{
			\begin{array}{rcl}
				P&\longrightarrow&\lin(D_E )\\
				p&\longmapsto&\left\{
				\begin{array}{rcl}
				D_E &\longrightarrow&D_E\\
				a&\longmapsto&\psi_E (p)(a).
			\end{array}
			\right.
			\end{array}
			\right.
		\end{align*}
		We extend the structure maps of $P$ as a~bilinear map $\triangleleft$ and an antisymmetric bilinear map $\{-,-\}$ on $\calT'_P (D_E,D_V )=\left(D_E\otimes \calT(D_E,D_V )\right)\oplus P$ by the following: for any $p,p'\in P$, for any $a,a'\in D_E$ and for any
		$x=T\otimes W_T$ and $x'=T'\otimes W_{T'}\in \calT(D_E,D_V )$,we put
		\begin{align*}
			a\otimes x\triangleleft a'\otimes x'&=\sum_{v_0\in V(T')}a'\otimes T \curvearrowright_{v_0}T'\otimes \phi_{1,v_0}\left(a\otimes W_T\otimes W_{T'}\right),\\
			p\triangleleft a'\otimes x'&=\sum_{v_0\in V(T')} a'\otimes T'\otimes \psi_V (p)_{v_0}(W_{T'}),\\
			a\otimes x\triangleleft p'&=0,\\
			\{a\otimes x,a'\otimes x'\}&=0,\\
			\{a\otimes x,p'\}&=\psi_E (p')(a)\otimes x.
		\end{align*}
		Then $(\calT'_P (D_E,D_V ),\triangleleft,\{-,-\})$ is a~post-Lie algebra if, and only if,
		\begin{align}\label{EQ5}&\forall p,p'\in P,&\psi_E (\{p,p'\})&=\psi_E (p')\circ \psi_E (p)-\psi_E (p)\circ \psi_E (p'),\\
			\label{EQ6}&&\psi_E (p\triangleleft p')&=0,\\
			\nonumber&&\psi_V (\{p,p'\})&=\psi_V (p)\circ \psi_V (p')-\psi_V (p')\circ \psi_V (p)\\
			\label{EQ7}&&&-\psi_V (p\triangleleft p')+\psi_V (p'\triangleleft p),\\
			\label{EQ8}&\forall p\in P,&\phi\circ (\psi_E (p)\otimes \id_{D_V})&=\phi\circ(\id_{D_E}\otimes \psi_V (p))-(\id_{D_E}\otimes \psi_V (p))\circ \phi.
		\end{align}
	\end{proposition}

	\begin{proof}
		\textit{Jacobi relation \eqref{EQ2} for $\{-,-\}$}. Observe that $(P,\{-,-\})$ is a~Lie algebra and $D_E\otimes \calT(D_E,D_V )$ is an abelian Lie algebra. Therefore, to obtain the Jacobi relation, it is enough to consider two cases.
		\begin{enumerate}
			\item First case: $x,y\in D_E\otimes \calT(D_E,D_V )$ and $z\in P$. With immediate notations,
			\begin{align*}
				\{\{a\otimes x,a'\otimes x' \},p''\}+\{\underbrace{\{a'\otimes x',p''\}}_{\in D_E\otimes \calT(D_E,D_V )},a\otimes x\}
				+\{\underbrace{\{p'',a\otimes x\}}_{\in D_E\otimes \calT(D_E,D_V )},a'\otimes x'\}=0.
			\end{align*}
			\item Second case: $x\in D_E\otimes \calT(D_E,D_V )$, $y,z\in P$. With immediate notations,
			\begin{align*}
				&\{\{a\otimes x,p'\},p''\}+\{\{p',p''\},a\otimes x\}+\{\{p'',a\otimes x\},p'\}\\
				&=\{\psi_E (p')(a)\otimes x,p''\}-\psi_E (\{p',p''\})(a)\otimes x-\{\psi_E (p'')(a)\otimes x,p'\}\\
				&=\psi_E (p'')\circ \psi_E (p')(a)\otimes x-\psi_E (\{p',p''\})(a)\otimes x-\psi_E (p')\circ \psi_E (p'')(a)\otimes x.
			\end{align*}
		\end{enumerate}
		Therefore, \eqref{EQ2} is satisfied if, and only if, \eqref{EQ5} is satisfied.

		\textit{First post-Lie relation \eqref{EQ3}}. By hypothesis, it is satisfied on the post-Lie algebra $P$. As the bracket is zero on $D_E$ and on $D_E\otimes \calT(D_E,D_V )$, and using its anti-symmetry, we can restrict ourselves to four cases.
		\begin{enumerate}
			\item $x,y\in D_E\otimes \calT(D_E,D_V )$, $z\in P$. Then
			\begin{align*}
				a\otimes x\triangleleft\{a'\otimes x',p''\}&=a\otimes x\triangleleft \psi_E (p'')(a')\otimes x'\\
				&=\psi_E (p'')(a')\otimes x\triangleleft_a^\phi x',\\
				\{a\otimes x\triangleleft a'\otimes x',p''\}+\{a'\otimes x',a\otimes x\triangleleft p''\}&=\{a'\otimes x\triangleleft_a^\phi x',p''\}+0\\
				&=\psi_E (p'')(a')\otimes x\triangleleft_a^\phi x'.
			\end{align*}
			\item $x\in P$, $y,z\in D_E\otimes \calT(D_E,D_V )$. Then
			\begin{align*}
				p\triangleleft\{a'\otimes x',a''\otimes x''\}&=0,\\
				\{\underbrace{p\triangleleft a'\otimes x'}_{\in D_E\otimes \calT(D_E,D_V )},a''\otimes x''\}
				+\{a'\otimes x'',\underbrace{p\triangleleft a''\otimes x''}_{\in D_E\otimes \calT(D_E,D_V )}\}&=0+0=0.
			\end{align*}
			\item $x\in D_E\otimes \calT(D_E,D_V )$, $y,z\in P$. Then
			\begin{align*}
				a\otimes x \triangleleft\{p',p''\}&=0,\\
				\{a\otimes x\triangleleft p',p''\}+\{p',a\otimes x\triangleleft p''\}&=0+0=0.
			\end{align*}
			\item $x,z\in P$, $y\in D_E\otimes \calT(D_E,D_V )$. Then
			\begin{align*}
				p\triangleleft \{a'\otimes x',p''\}&=p\triangleleft \psi_E (p'')(a')\otimes x'\\
				&=\sum_{v_0\in V(T')} \psi_E (p'')(a')\otimes \psi_V (p)_{v_0}(x),\\
				\{p\triangleleft a'\otimes x',p''\}+\{a'\otimes x',p\triangleleft p''\}&=\sum_{v_0\in V(T')}\{a'\otimes \psi_V (p)_{v_0}(x'),v''\}\\
				&+\sum_{v_0\in V(T')}\psi_E (p\triangleleft p'')(a')\otimes x'\\
				&=\sum_{v_0\in V(T')} \psi_E (p'')(a')\otimes \psi_V (p)_{v_0}(x)\\
				&+\sum_{v_0\in V(T')}\psi_E (p\triangleleft p'')(a')\otimes x'.
			\end{align*}
		\end{enumerate}
		So \eqref{EQ3} is satisfied if, and only if \eqref{EQ6} is satisfied.

		\emph{Second post-Lie relation \eqref{EQ4}}. We already notice that $(D_E\otimes \calT(D_E,D_V ),\triangleleft)$ is pre-Lie. As the bracket is zero on this subspace, the second post-Lie relation is satisfied on $D_E\otimes \calT(D_E,D_V )$. It is also satisfied on the post-Lie algebra $P$. Using the anti-symmetry on $a,b$, we can restrict ourselves to four cases.
		\begin{enumerate}
			\item $x,y \in D_E\otimes \calT(D_E,D_V )$ and $z\in P$. Then
			\begin{align*}
				\{a\otimes x,a'\otimes x'\}\triangleleft p''&=0,\\
				a\otimes x\triangleleft (a'\otimes x' \triangleleft p'')-(\underbrace{a\otimes x\triangleleft a'\otimes x'}_{\in D_E\otimes \calT(D_E,D_V )})\triangleleft p''&=0+0=0.
			\end{align*}
			\item $y,z\in P$, $x\in D_E\otimes \calT(D_E,D_V )$. Then
			\begin{align*}
				\underbrace{\{a\triangleleft x,p'\}}_{\in D_E\otimes \calT(D_E,D_V )}\triangleleft p''&=0,
			\end{align*}
			whereas
			\begin{align*}
				&a\otimes x\triangleleft(p'\otimes p'')-(a\otimes x \triangleleft p')\triangleleft p''-p'\triangleleft (a\otimes x\triangleleft p'')+(\underbrace{p'\triangleleft a\otimes x}_{\in D_E\otimes \calT(D_E,D_V )})\triangleleft p''\\
				&=0+0+0+0=0.
			\end{align*}
			\item $x,y\in P$, $z\in D_E\otimes \calT(D_E,D_V )$. Then
			\begin{align*}
				\{p,p'\}\triangleleft a''\otimes x''&=\sum_{v_0\in V(T'')} a''\otimes \psi_V (\{p',p''\})_{v_0}(x''),
			\end{align*}
			whereas
			\begin{align*}
				&p\triangleleft (p'\triangleleft a''\otimes x'')-(p\triangleleft p')\triangleleft a''\otimes x''-p'\triangleleft (p\triangleleft a''\otimes x'')+(p'\triangleleft p)\triangleleft a''\otimes x''\\
				&=\sum_{v_0,v_1\in V(T'')}a''\otimes \psi_V (p)_{v_1}\circ \psi_V (p')_{v_0}(x'')-\sum_{v_0\in V(T'')}a''\otimes \psi_V (p\triangleleft p')_{v_0}(x'')\\
				&-\sum_{v_0,v_1\in V(T'')}a''\otimes \psi_V (p')_{v_0}\circ \psi_V (p)_{v_1}(x'')
				+\sum_{v_0\in V(T'')}a''\otimes \psi_V (p'\triangleleft p'')_{v_0}(x'').
			\end{align*}
			If $v_0\neq v_1$, obviously $\psi_V (p)_{v_0}$ and $\psi_V (p)_{v_1}$ commute. Therefore,
			\begin{align*}
				&p\triangleleft (p'\triangleleft a''\otimes x'')-(p\triangleleft p')\triangleleft a''\otimes x''-p'\triangleleft (p\triangleleft a''\otimes x'')+(p'\triangleleft p)\triangleleft a''\otimes x''\\
				&=\sum_{v_0\in V(T'')}a''\otimes \psi_V (p)_{v_0}\circ \psi_V (p')_{v_0}(x'')-\sum_{v_0\in V(T'')}a''\otimes \psi_V (p\triangleleft p')_{v_0}(x'')\\
				&-\sum_{v_0\in V(T'')}a''\otimes \psi_V (p')_{v_0}\circ \psi_V (p)_{v_0}(x'')
				+\sum_{v_0\in V(T'')}a''\otimes \psi_V (p'\triangleleft p'')_{v_0}(x'').
			\end{align*}
			Hence, in this third case, \eqref{EQ4} is equivalent to \eqref{EQ7}.

			\item $x,z \in D_E\otimes \calT(D_E,D_V )$, $y\in P$. Then, with immediate notations,
			\begin{align*}
				\{a\otimes x,p'\}\triangleleft a''\otimes x''&=\psi_E (p')(a)\otimes x\triangleleft a''\otimes x''\\
				&=a'' \otimes x\triangleleft_{\psi_E (p')(a)}x''\\
				&=\sum_{v_0\in V(T'')}a''\otimes T\curvearrowright_{v_0}T''\otimes \phi_{1,v_0}\circ \psi_E (p')_0 (a\otimes W_T\otimes W_{T''}).
			\end{align*}
			On the other side,
			\begin{align*}
				a\otimes x\triangleleft(p'\triangleleft a''\otimes x'')&=\sum_{v_0\in V(T'')}a\otimes x\triangleleft p''\otimes T'' \otimes \psi_V (p')_{v_0}\left(W_{T''}\right)\\
				&=\!\!\sum_{v_0,v_1\in V(T'')}a''\otimes T\curvearrowright_{v_1}T''\otimes \phi_{1,v_1}\circ \psi_V (a')_{v_0}(a\otimes W_T\otimes W_{T''}),\\
				(a\otimes x\triangleleft p')\triangleleft a''\otimes x''&=0,\\
				p'\triangleleft (a\otimes x\triangleleft a''\otimes x'')&=\sum_{v_1\in V(T'')}p'\triangleleft a''\otimes T\curvearrowright_{v_1}T'' \otimes \phi_{1,v_1}(a\otimes W_T\otimes W_{T''})\\
				&=\sum_{v_0,v_1\in V(T'')} a'' \otimes T\curvearrowright_{v_1}T'' \otimes \psi_V (p')_{v_0}\circ \phi_{1,v_1})\\
				&+\sum_{\substack{v_0\in V(T),\\v_1\in V(T'')}}a'' \otimes T\curvearrowright_{v_1}T'' \otimes \psi_V (p')_{v_0}\circ \phi_{1,v_1}(a\otimes W_T\otimes W_{T''}),\\
				(p'\triangleleft a\otimes x)\triangleleft a''\otimes x''&=\sum_{v_0\in V(T)}a\otimes T\otimes \psi_V (p')_{v_0}\left(W_T\right)\triangleleft a''\otimes x''\\
				&=\sum_{\substack{v_0\in V(T),\\v_1\in V(T'')}}a'' \otimes T\curvearrowright_{v_1}T'' \otimes \phi_{1,v_1}\circ \psi_V (p')_{v_0}(a\otimes W_T\otimes W_{T''}).
			\end{align*}
			If $v_0\in V(T)$ and $v_1\in V(T'')$ or if $v_0,v_1$ are two different vertices of $T''$, obviously $\phi_{1,v_1}$ and $\psi_V (p')_{v_0}$ commute. Therefore, we obtain finally
			\begin{align*}
				&a\otimes x\triangleleft(p'\triangleleft a''\otimes x'')-(a\otimes x\triangleleft p')\triangleleft a''\otimes x''\\
                &-p'\triangleleft (a\otimes x\triangleleft a''\otimes x'')+(p'\triangleleft a\otimes x)\triangleleft a''\otimes x''\\
				&=\sum_{v_0\in V(T'')}a''\otimes T\curvearrowright_{v_0}T''\otimes \phi_{2,v_0}\circ \psi_V (p')_{v_0}(a\otimes W_T\otimes W_{T''}).
			\end{align*}
			As a~consequence, we obtain that \eqref{EQ4} is satisfied in this last case if and only if, for any $v_0\in V(T'')$,
			\[
				\phi_{1,v_0}\circ \psi_E (p')_0 =\phi_{1,v_0}\circ \psi_V (p'_{v_0})-\psi_V (p')_{v_0}\circ \phi_{1,v_0},
			\]
			which is equivalent to \eqref{EQ8}. \qedhere
		\end{enumerate}
	\end{proof}

	\begin{example}
		If $p\in P$ and $a_1,a_2,a_3\in D_E$, $b_1,b_2,b_3\in D_V$,
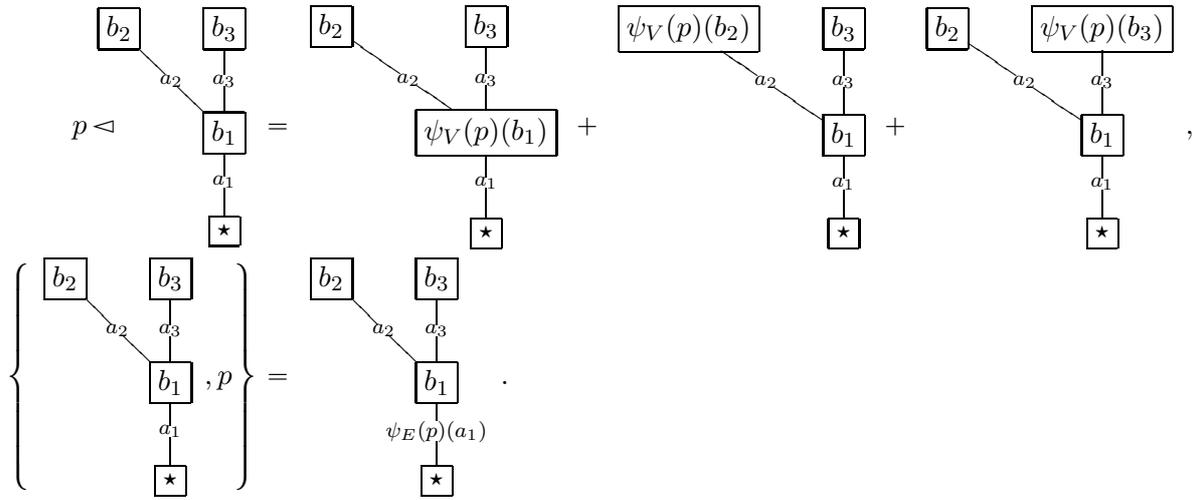
		\begin{align*}
			p\triangleleft
			\begin{array}{c}\xymatrix{\boite{b_2}\ar@{-}[rd]|{a_2}&\boite{b_3}\ar@{-}[d]|{a_3}\\&\boite{b_1}\\&\boite{\blackdiamond} \ar@{-}[u]|{a_1}}
			\end{array}
			&=
			\begin{array}{c}\xymatrix{\boite{b_2}\ar@{-}[rd]|{a_2}&\boite{b_3}\ar@{-}[d]|{a_3}\\&\boite{\psi_V (p)(b_1 )}\\&\boite{\blackdiamond} \ar@{-}[u]|{a_1}}
			\end{array}
			+\!\!
			\begin{array}{c}\xymatrix{\boite{\psi_V (p)(b_2 )}\ar@{-}[rd]|{a_2}&\boite{b_3}\ar@{-}[d]|{a_3}\\&\boite{b_1}\\&\boite{\blackdiamond} \ar@{-}[u]|{a_1}}
			\end{array}
			+\!\!
			\begin{array}{c}\xymatrix{\boite{b_2}\ar@{-}[rd]|{a_2}&\boite{\psi_V (p)(b_3 )}\ar@{-}[d]|{a_3}\\&\boite{b_1}\\&\boite{\blackdiamond} \ar@{-}[u]|{a_1}}
			\end{array}
			,\\
			\left\{
			\begin{array}{c}\xymatrix{\boite{b_2}\ar@{-}[rd]|{a_2}&\boite{b_3}\ar@{-}[d]|{a_3}\\&\boite{b_1}\\&\boite{\blackdiamond} \ar@{-}[u]|{a_1}}
			\end{array}
			,p\right\}&=
			\begin{array}{c}\xymatrix{\boite{b_2}\ar@{-}[rd]|{a_2}&\boite{b_3}\ar@{-}[d]|{a_3}\\&\boite{b_1}\\&\boite{\blackdiamond} \ar@{-}[u]|{\psi_E (p)(a_1 )}}
			\end{array}
			.
		\end{align*}
	\end{example}

	\begin{remark}
		\begin{enumerate}
			\item Relation \eqref{EQ5} means that $-\psi_E:(P,\{-,-\})\longrightarrow (\lin(D_E ),[-,-])$ is a~Lie algebra morphism, where $[-,-]$ is the usual bracket of endomorphisms.
			\item The post-Lie algebra $P$ has a~second Lie bracket, defined by
			\begin{align*}
				&\forall p,p'\in P,&\{\{p,p'\}\}=\{p,p'\}+p\triangleleft p'-p'\triangleleft p.
			\end{align*}
			Relation \eqref{EQ7} means that $\psi_V:(P,\{\{-,-\}\})\longrightarrow (\lin(D_V ),[-,-])$ is a~Lie algebra morphism.
		\end{enumerate}
	\end{remark}

	In the case where $P$ is trivial, this simplifies:

	\begin{corollary}\label{cor3.10}
		Let $P$ be a~vector space and two maps
		\begin{align*}
			\psi_V &:\left\{
			\begin{array}{rcl}
				P&\longrightarrow&\lin(D_V )\\
				p&\longmapsto&\left\{
				\begin{array}{rcl}
				D_V &\longrightarrow&D_V\\
				b&\longmapsto&\psi_V (p)(b),
			\end{array}
			\right.
			\end{array}
			\right.&
			\psi_E &:\left\{
			\begin{array}{rcl}
				P&\longrightarrow&\lin(D_E )\\
				p&\longmapsto&\left\{
				\begin{array}{rcl}
				D_V &\longrightarrow&D_V\\
				a&\longmapsto&\psi_E (p)(a).
			\end{array}
			\right.
			\end{array}
			\right.
		\end{align*}
		We define a~bilinear map $\triangleleft$ and an antisymmetric bilinear map $\{-,-\}$ on $\calT'_P (D_E,D_V )=\left(D_E\otimes \calT(D_E,D_V )\right)\oplus P$ by the following: for any $p,p'\in P$ and for any $x=T\otimes W_T$ and $x'=T'\otimes W_{T'} \in \calT(D_E,D_V )$, we put
		\begin{align*}
			a\otimes x\triangleleft a'\otimes x'&=\sum_{v_0\in V(T')}a'\otimes T \curvearrowright_{v_0}T'\otimes \phi_{1,v_0}\left(a\otimes W_T\otimes W_{T'}\right),\\
			p\triangleleft a'\otimes x'&=\sum_{v_0\in V(T')} a'\otimes T'\otimes \psi_V (p)(a')_{v_0}(W_{T'}),\\
			a\otimes x\triangleleft p'&=0,\\
			p\triangleleft p'&=0,\\
			\{a\otimes x,a'\otimes x'\}&=0,\\
			\{a\otimes x,p'\}&=\psi_E (p')(a)\otimes x,\\
			\{p,p'\}&=0.
		\end{align*}
		Then $(\calT'_P (D_E,D_V ),\triangleleft,\{-,-\})$ is a~post-Lie algebra if, and only if,
		\begin{align}
			&\forall p,p'\in P,&\psi_E (p')\circ \psi_E (p)&=\psi_E (p)\circ \psi_E (p'),\\
			&\forall p,p'\in P,&\psi_V (p)\circ \psi_V (p')&=\psi_V (p')\circ \psi_V (p),\\
			\label{eq11}&\forall p\in P,&\phi\circ (\psi_E (p)\otimes \id_{D_V})&=\phi\circ(\id_{D_E}\otimes \psi_V (p))-(\id_{D_E}\otimes \psi_V (p))\circ \phi.
		\end{align}
		If $(\psi_E,\psi_V )$ satisfies these three conditions, we shall say that $(\psi_E,\psi_V )$ is $\phi$-compatible.
	\end{corollary}

	\subsection{Guin-Oudom construction}

	We now apply the Guin-Oudom construction~\cite{Oudom2005,Oudom2008} to the pre-Lie algebra $(D_E\otimes \calT(D_E,D_V ),\triangleleft^\phi)$, when $\phi$ is tree-compatible. The extension of $\triangleleft^\phi$ to $S((D_E\otimes \calT(D_E,D_V ))$ is also denoted by $\triangleleft^\phi$ and the associated product by $\star^\phi$. Then $S(D_E\otimes \calT(D_E,D_V ))$, with the product $\star^\phi$ and its usual coproduct is a~Hopf algebra, isomorphic to the enveloping algebra of the Lie algebra underlying the pre-Lie algebra $(D_E\otimes \calT(D_E,D_V ),\triangleleft^\phi)$

	\begin{proposition}\label{prop3.11}
		\begin{enumerate}
			\item Let $a_1\otimes x_1,\ldots,a_k\otimes x_k,a'\otimes x'\in D_E\otimes \calT(D_E,D_V )$. Then
			\begin{align*}
				a_1{ }&{ }\otimes x_1\cdots a_k\otimes x_k \triangleleft^\phi a'\otimes x' \\
                &=\sum_{v_1,\ldots,v_k\in V(T')} a'\otimes T_1\curvearrowright_{v_1}(\cdots (T_k\curvearrowright_{v_k} T')\cdots)\\
				&{\qquad\qquad\qquad\quad}\otimes \phi_{1,v_1}\circ \cdots \circ \phi_{k,v_k}(a_1\otimes \cdots \otimes a_k \otimes W_{T_1}\otimes \cdots\otimes W_{T_k}\otimes W_{T'}).
			\end{align*}
			\item Let $a_1\otimes x_1,\ldots,a_k\otimes x_k,a'_1\otimes x'_1,\ldots,a'_l\otimes x'_l\in D_E\otimes \calT(D_E,D_V )$. Then
			\begin{align*}
				a_1\otimes x_1\cdots a_k\otimes x_k { }&{ }\triangleleft^\phi a'_1\otimes x'_1\cdots\otimes a'_l\otimes x'_l \\
                &=
				\sum_{[k]=I_1\sqcup \cdots \sqcup I_l}\prod_{j=1}^l\left(\left(\prod_{i\in I_j} a_i\otimes x_i\right)\triangleleft^\phi a'_j\otimes x'_j\right).
			\end{align*}
			\item Let $a_1\otimes x_1,\ldots,a_k\otimes x_k,a'_1\otimes x'_1,\ldots,a'_l\otimes x'_l\in D_E\otimes \calT(D_E,D_V )$. Then
			\begin{align*}
				&a_1\otimes x_1\cdots a_k\otimes x_k \star^\phi a'_1\otimes x'_1\cdots\otimes a'_l\otimes x'_l\\
				&=\sum_{[k]=I_0\sqcup I_1\sqcup \cdots \sqcup I_l}\left(\prod_{i\in I_0} a_i\otimes x_i\right)\prod_{j=1}^l\left(\left(\prod_{i\in I_j} a_i\otimes x_i\right)\triangleleft^\phi a'_j\otimes x'_j\right).
			\end{align*}
		\end{enumerate}
	\end{proposition}

	\begin{proof}
		The second and third points are direct applications of the Guin-Oudom construction. Let us prove the first point by induction on $k$. If $k=1$, this is immediately implied by the definition of $\triangleleft^\phi$ by graftings. Let us assume the result at rank $k-1$. In order to lighten the writing, we put
		\begin{align*}
			W&=a_1\otimes \cdots \otimes a_k \otimes W_{T_1}\otimes \cdots\otimes W_{T_k}\otimes W_{T'},\\
			W'&=a_2\otimes \cdots \otimes a_k \otimes W_{T_2}\otimes \cdots\otimes W_{T_k}\otimes W_{T'}.
		\end{align*}
		Then
		\begin{align*}
			&a_1\otimes x_1\cdots a_k\otimes x_k \triangleleft^\phi a'\otimes x'\\
			&=a_1\otimes x_1 \triangleleft^\phi\left(\sum_{v_2,\ldots,v_k \in V(T')}a'\otimes T_2\curvearrowright_{v_2}(\cdots (T_k\curvearrowright_{v_k} T')\cdots)\otimes \phi_{2,v_2}\circ \cdots \circ \phi_{k,v_k}(W')\right)\\
			&-\sum_{i=2}^k \sum_{\substack{v_2,\ldots,v_k\in V(T'),\\v_1\in V(T_i )}} a'\otimes T_2\curvearrowright_{v_2}(\cdots (T_1\curvearrowright_{v_1}T_i )\curvearrowright_{v_i}(\cdots (T_k\curvearrowright_{v_k}T'))\cdots)\\
			&{\qquad\qquad\qquad\qquad}\otimes \phi_{2,v_2}\circ \cdots \circ \phi_{k,v_k}\circ \phi_{1,v_1}(W)\\
			&=\sum_{v_1,\ldots,v_k\in V(T')}a'\otimes T_1\curvearrowright_{v_1}(\cdots (T_k\curvearrowright_{v_k} T')\cdots)\otimes \phi_{1,v_1}\circ \cdots \circ \phi_{k,v_k}(W)\\
			&+\sum_{i=2}^k \sum_{\substack{v_2,\ldots,v_k\in V(T'),\\v_1\in V(T_i )}}a'\otimes T_1\curvearrowright_{v_1}(\cdots (T_k\curvearrowright_{v_k} T')\cdots)\otimes \phi_{1,v_1}\circ \cdots \circ \phi_{k,v_k}(W)\\
			&-\sum_{i=2}^k \sum_{\substack{v_2,\ldots,v_k\in V(T'),\\v_1\in V(T_i )}} a'\otimes T_2\curvearrowright_{v_2}(\cdots (T_1\curvearrowright_{v_1}T_i )\curvearrowright_{v_i}(\cdots (T_k\curvearrowright_{v_k}T'))\cdots)\\
			&{\qquad\qquad\qquad\qquad}\otimes \phi_{2,v_2}\circ \cdots \circ \phi_{k,v_k}\circ \phi_{1,v_1}(W).
		\end{align*}
		We conclude with the equality of the trees
		\[
			T_2\curvearrowright_{v_2}(\cdots (T_1\curvearrowright_{v_1}T_i )\curvearrowright_{v_i}(\cdots (T_k\curvearrowright_{v_k}T'))=T_1\curvearrowright_{v_1}(\cdots (T_k\curvearrowright_{v_k} T')\cdots)
		\]
		when $v_2,\ldots,v_k\in V(T')$ and $v_1\in V(T_i )$ with $i\geq 2$ and by the commutation of the $\phi_{i,v_i}$ by the tree-compatibility of $\phi$. It remains
		\begin{align*}
			&a_1\otimes x_1\cdots a_k\otimes x_k \triangleleft^\phi a'\otimes x'\\
			&=\sum_{v_1,\ldots,v_k\in V(T')}a'\otimes T_1\curvearrowright_{v_1}(\cdots (T_k\curvearrowright_{v_k} T')\cdots)\otimes \phi_{1,v_1}\circ \cdots \circ \phi_{k,v_k}(W). \qedhere
		\end{align*}
	\end{proof}

	\begin{example}
		Let $a_1,a_2,a_3,a_4\in D_E$ and $b_1,b_2,b_3,b_4\in D_V$. We identify elements of $a\otimes x \in D_E\otimes \calT(D_E,D_V )$ with planted trees.
		\begin{align*}
			\begin{array}{c}
				\xymatrix{\boite{b_1}\ar@{-}[d]|{a_1}\\ \boite{\blackdiamond}}
			\end{array}
			\begin{array}{c}
				\xymatrix{\boite{b_2}\ar@{-}[d]|{a_a}\\ \boite{\blackdiamond}}
			\end{array}
			\triangleleft^\phi
			\begin{array}{c}
				\xymatrix{\boite{b_4}\ar@{-}[d]|{a_4}\\\boite{b_3}\ar@{-}[d]|{a_3}\\ \boite{\blackdiamond}}
			\end{array}
			&=\sum_{i,j}\left(
			\begin{array}{c}
				\begin{array}{c}
				\xymatrix{\boite{b_1}\ar@{-}[rd]|{\phi_E^i (a_1 )}&\boite{b_2}\ar@{-}[d]|{\phi_E^i (a_2 )}&\boite{b_4}\ar@{-}[ld]|{a_4}\\ &\boite{\phi_V^i\circ \phi_V^j (b_3 )}\ar@{-}[d]|{a_3}&\\ &\boite{\blackdiamond}&}
			\end{array}
			+
			\begin{array}{c}
				\xymatrix{&\boite{b_2}\ar@{-}[d]|{\phi_E^j (a_2 )}\\\boite{b_1}\ar@{-}[rd]|{\phi_E^i (a_1 )}&\boite{\phi_V^j (b_4 )}\ar@{-}[d]|{a_4}\\&\boite{\phi_V^i (b_3 )}\ar@{-}[d]|{a_3}\\&\boite{\blackdiamond}}
			\end{array}
			\\
			+
			\begin{array}{c}
				\xymatrix{&\boite{b_1}\ar@{-}[d]|{\phi_E^i (a_1 )}\\\boite{b_2}\ar@{-}[rd]|{\phi_E^j (a_2 )}&\boite{\phi_V^i (b_4 )}\ar@{-}[d]|{a_4}\\&\boite{\phi_V^j (b_3 )}\ar@{-}[d]|{a_3}\\&\boite{\blackdiamond}}
			\end{array}
			+
			\begin{array}{c}
				\xymatrix{\boite{b_1}\ar@{-}[rd]|{\phi_E^i (a_1 )}&\boite{b_2}\ar@{-}[d]|{\phi_E^j (a_2 )}\\&\boite{\phi_V^i\circ \phi_V^j (b_4 )}\ar@{-}[d]|{a_4}\\&\boite{b_3}\ar@{-}[d]|{a_3}\\&\boite{\blackdiamond}}
			\end{array}
			\end{array}
			\right),
		\end{align*}
		whereas
		\begin{align*}
			\begin{array}{c}
				\xymatrix{\boite{b_1}\ar@{-}[d]|{a_1}\\ \boite{\blackdiamond}}
			\end{array}
			\begin{array}{c}
				\xymatrix{\boite{b_2}\ar@{-}[d]|{a_a}\\ \boite{\blackdiamond}}
			\end{array}
			\star^\phi
			\begin{array}{c}
				\xymatrix{\boite{b_4}\ar@{-}[d]|{a_4}\\\boite{b_3}\ar@{-}[d]|{a_3}\\ \boite{\blackdiamond}}
			\end{array}
			&=\sum_{i,j}\left(
			\begin{array}{c}
				\begin{array}{c}
				\xymatrix{\boite{b_1}\ar@{-}[rd]|{\phi_E^i (a_1 )}&\boite{b_2}\ar@{-}[d]|{\phi_E^i (a_2 )}&\boite{b_4}\ar@{-}[ld]|{a_4}\\ &\boite{\phi_V^i\circ \phi_V^j (b_3 )}\ar@{-}[d]|{a_3}&\\ &\boite{\blackdiamond}&}
			\end{array}
			+
			\begin{array}{c}
				\xymatrix{&\boite{b_2}\ar@{-}[d]|{\phi_E^j (a_2 )}\\\boite{b_1}\ar@{-}[rd]|{\phi_E^i (a_1 )}&\boite{\phi_V^j (b_4 )}\ar@{-}[d]|{a_4}\\&\boite{\phi_V^i (b_3 )}\ar@{-}[d]|{a_3}\\&\boite{\blackdiamond}}
			\end{array}
			\\
			+
			\begin{array}{c}
				\xymatrix{&\boite{b_1}\ar@{-}[d]|{\phi_E^i (a_1 )}\\\boite{b_2}\ar@{-}[rd]|{\phi_E^j (a_2 )}&\boite{\phi_V^i (b_4 )}\ar@{-}[d]|{a_4}\\&\boite{\phi_V^j (b_3 )}\ar@{-}[d]|{a_3}\\&\boite{\blackdiamond}}
			\end{array}
			+
			\begin{array}{c}
				\xymatrix{\boite{b_1}\ar@{-}[rd]|{\phi_E^i (a_1 )}&\boite{b_2}\ar@{-}[d]|{\phi_E^j (a_2 )}\\&\boite{\phi_V^i\circ \phi_V^j (b_4 )}\ar@{-}[d]|{a_4}\\&\boite{b_3}\ar@{-}[d]|{a_3}\\&\boite{\blackdiamond}}
			\end{array}
			\end{array}
			\right)\\
			&+\sum_i \left(
			\begin{array}{c}
				\xymatrix{\boite{b_1}\ar@{-}[d]|{a_1}\\ \boite{\blackdiamond}}
			\end{array}
			\begin{array}{c}
				\xymatrix{\boite{b_2}\ar@{-}[rd]|{\phi_E^i (a_2 )}&\boite{b_4}\ar@{-}[d]|{a_4}\\&\boite{\phi_V^i (b_3 )}\ar@{-}[d]|{a_3}\\&\boite{\blackdiamond}}
			\end{array}
			+
			\begin{array}{c}
				\xymatrix{\boite{b_1}\ar@{-}[d]|{a_1}\\ \boite{\blackdiamond}}
			\end{array}
			\begin{array}{c}
				\xymatrix{\boite{b_2}\ar@{-}[d]|{\phi_E^i (a_2 )}\\\boite{\phi_V^i (b_4 )}\ar@{-}[d]|{a_4}\\\boite{b_3}\ar@{-}[d]|{a_3}\\\boite{\blackdiamond}}
			\end{array}
			\right)\\
			&+\sum_i \left(
			\begin{array}{c}
				\xymatrix{\boite{b_2}\ar@{-}[d]|{a_2}\\ \boite{\blackdiamond}}
			\end{array}
			\begin{array}{c}
				\xymatrix{\boite{b_1}\ar@{-}[rd]|{\phi_E^i (a_1 )}&\boite{b_4}\ar@{-}[d]|{a_4}\\&\boite{\phi_V^i (b_3 )}\ar@{-}[d]|{a_3}\\&\boite{\blackdiamond}}
			\end{array}
			+
			\begin{array}{c}
				\xymatrix{\boite{b_2}\ar@{-}[d]|{a_2}\\ \boite{\blackdiamond}}
			\end{array}
			\begin{array}{c}
				\xymatrix{\boite{b_1}\ar@{-}[d]|{\phi_E^i (a_1 )}\\\boite{\phi_V^i (b_4 )}\ar@{-}[d]|{a_4}\\\boite{b_3}\ar@{-}[d]|{a_3}\\\boite{\blackdiamond}}
			\end{array}
			\right)+
			\begin{array}{c}
				\xymatrix{\boite{b_1}\ar@{-}[d]|{a_1}\\ \boite{\blackdiamond}}
			\end{array}
			\begin{array}{c}
				\xymatrix{\boite{b_2}\ar@{-}[d]|{a_a}\\ \boite{\blackdiamond}}
			\end{array}
			\begin{array}{c}
				\xymatrix{\boite{b_4}\ar@{-}[d]|{a_4}\\\boite{b_3}\ar@{-}[d]|{a_3}\\ \boite{\blackdiamond}}
			\end{array}
			.
		\end{align*}
	\end{example}

	Let us give a~combinatorial description of $\star^\phi$. We shall need the following notion:
	\begin{definition}
		Let $F$ and $G$ be two forests of planted rooted trees. The set of trees of $F$ is denoted by $\bfT(F)$. A grafting of $F$ over $G$ is a~map $g:\bfT(F)\longrightarrow V(G)\sqcup \{0\}$. The set of graftings of $F$ over $G$ is denoted by $\bfG(F,G)$. If $g\in \bfG(F,G)$, then $F\curvearrowright_g G$ is the forest of planted rooted trees obtained from $FG$ by identifying the undecorated root of $T$ with $g(T)$ for any $T\in \bfT(F)$ such that $g(T)\neq 0$.
	\end{definition}
	Then:

	\begin{proposition}
		Let $x=F\otimes W_F$ and $y=G\otimes W_G$ in $S(D_E\otimes \calT(D_E\otimes D_V ))$. Then
		\[
			x\star^\phi y=\sum_{g\in\bfG(F,G)} F\curvearrowright_g G \otimes \prod^\circ_{\substack{\in \bfT(F),\\g(T)\neq 0}} \phi_{e_T,g(T)}(W_F\otimes W_G ),
		\]
		where $e_T$ is the decoration of the unique edge whose source is the undecorated root of $T$.
	\end{proposition}

	We extend $\Theta_\phi$ to $D_E\otimes \calT(D_E,D_V )$ by $\id_{D_E}\otimes \calT(D_E,D_V )$, which in turn is extended to $S(D_E\otimes \calT(D_E,D_V ))$ as an algebra morphism. This extension is denoted by $\overline{\Theta}_\phi$.

	\begin{example}
		Let $a_1,a_2,a_3\in D_E$ and $b_1,b_2,b_3\in D_E$.
		\begin{align*}
			\overline{\Theta}_\phi\left(
			\begin{array}{c}
				\xymatrix{\boite{b_1}\ar@{-}[rd]|{a_1}&\boite{b_2}\ar@{-}[d]|{a_2}\\&\boite{b_3}\ar@{-}[d]|{a_3}\\&\boite{\blackdiamond}}
			\end{array}
			\right)&=
			\sum_{i,j}
			\begin{array}{c}
				\xymatrix{\boite{b_1}\ar@{-}[rd]|{\phi_E^i (a_1 )}&\boite{b_2}\ar@{-}[d]|{\phi_E^j (a_2 )}\\&\boite{\phi_V^i\circ \phi_V^j (b_3 )}\ar@{-}[d]|{a_3}\\&\boite{\blackdiamond}}
			\end{array}
			.
		\end{align*}
	\end{example}

	The functoriality of the Guin-Oudom construction gives:

	\begin{proposition}
		Let $\phi,\psi$ be two tree-compatible maps on $D_E\otimes D_V$ such that
		\[
			\psi_{13}\circ \phi_{23}=\phi_{23}\circ \psi_{13}.
		\]
		The map $\overline{\Theta}_\phi$ is a~morphism from the bialgebra $(S(D_E\otimes \calT(D_E,D_V )),\star^\psi,\Delta)$ to the bialgebra $(S(D_E\otimes \calT(D_E,D_V )),\star^{\phi\circ \psi},\Delta)$. It is an isomorphism if, and only if, $\phi$ is bijective.
	\end{proposition}

	A particular case is $\phi=\id_{D_E\otimes D_V}$. In this case, we simply write $\star$ instead of $\star^{\id_{D_E\otimes D_V}}$. We obtain:

	\begin{corollary}
		Let $\phi$ be a~tree-compatible map $D_E\otimes D_V$. The map $\overline{\Theta}_\phi$ is a~bialgebra morphism from $(S(D_E\otimes \calT(D_E,D_V )),\star,\Delta)$ to
		$(S(D_E\otimes \calT(D_E,D_V )),\star^\phi,\Delta)$. It is an isomorphism if, and only if, $\phi$ is bijective.
	\end{corollary}

	\subsection{Butcher-Connes-Kreimer construction}

	We here define a~Butcher-Connes-Kreimer-like bialgebra structure associated to a~tree-compatible map $\phi$. We shall need the following combinatorial notions:
	\begin{definition}
		Let $F$ be a~forest of planted rooted trees. Its set of vertices (which do not include the roots decorated by $\star$) is denoted by $V(F)$ and its set of edges by $E(F)$.
		\begin{enumerate}
			\item An upper part of $F$ is a~subset $I$ of $V(F)$ such that for any $x\in I$ and $y\in V(F)$, if there exists an edge $e\in E(F)$ such that $s(e)=x$ and $t(e)=y$, then $y\in I$. The set of upper parts of $F$ is denoted by $\Up(F)$.
			\item Let $I\subseteq V(F)$. The forest of planted rooted trees $F_{\mid I}$ is defined by
			\begin{align*}
				V(F_{\mid I})&=I,&
				E(F_{\mid I})&=\{e\in E(F)\mid t(e)\in I\}.
			\end{align*}
			This is a~forest of planted rooted trees (note that it is necessary to add an undecorated root to each component part). Moreover, if $F\in S(D_E\otimes \calT(D_E,D_V ))$, then by restriction of the decorations, $F_{\mid I}\in S(D_E\otimes \calT(D_E,D_V ))$.
		\end{enumerate}
	\end{definition}

	\begin{example}
		Let us consider the planted rooted tree, which we decorate:
		\begin{align*}
			T&=
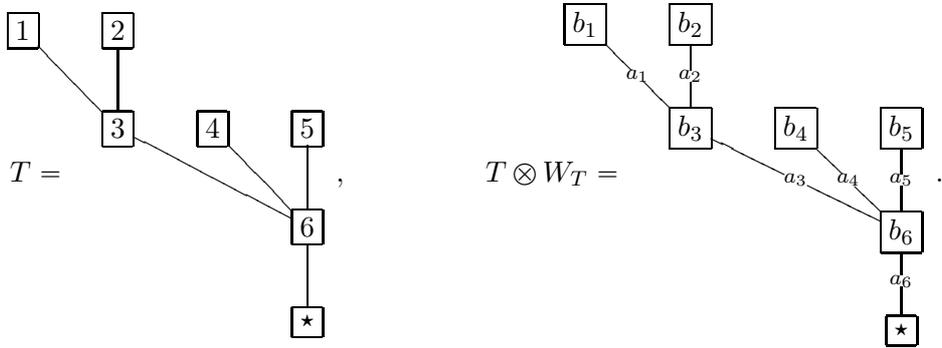
			\begin{array}{c}
				\xymatrix{\boite{1}\ar@{-}[rd]&\boite{2}\ar@{-}[d]&&\\&\boite{3}\ar@{-}[rrd]&\boite{4}\ar@{-}[rd]&\boite{5}\ar@{-}[d]\\&&&\boite{6}\ar@{-}[d]\\&&&\boite{\blackdiamond}}
			\end{array}
			,&
			T\otimes W_T &=
			\begin{array}{c}
				\xymatrix{\boite{b_1}\ar@{-}[rd]|{a_1}&\boite{b_2}\ar@{-}[d]|{a_2}&&\\&\boite{b_3}\ar@{-}[rrd]|{a_3}&\boite{b_4}\ar@{-}[rd]|{a_4}&\boite{b_5}\ar@{-}[d]|{a_5}\\&&&\boite{b_6}\ar@{-}[d]|{a_6}\\&&&\boite{\blackdiamond}}
			\end{array}
			.
		\end{align*}
		Then $I=\{1,2,3,4\}$ is an upper part of $T$. Moreover,
		\begin{align*}
			(T\otimes W_T )_{\mid \{1,2,3,4\}}&=
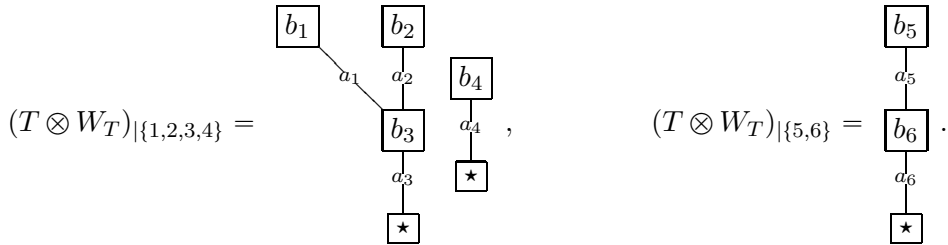
			\begin{array}{c}
				\xymatrix{\boite{b_1}\ar@{-}[rd]|{a_1}&\boite{b_2}\ar@{-}[d]|{a_2}\\&\boite{b_3}\ar@{-}[d]|{a_3}\\&\boite{\blackdiamond}}
			\end{array}
			\begin{array}{c}
				\xymatrix{\boite{b_4}\ar@{-}[d]|{a_4}\\\boite{\blackdiamond}}
			\end{array}
			,&(T\otimes W_T )_{\mid \{5,6\}}&=
			\begin{array}{c}
				\xymatrix{\boite{b_5}\ar@{-}[d]|{a_5}\\\boite{b_6}\ar@{-}[d]|{a_6}\\\boite{\blackdiamond}}
			\end{array}
			.
		\end{align*}
	\end{example}

	\begin{notation}
		Any forest of planted rooted trees of $S(D_E\otimes \calT(D_E\otimes D_V ))$ will be denoted under the form
		\[
			x=F\otimes\underbrace{\bigotimes_{e\in E(F)} d_e\otimes \bigotimes_{v\in V(F)}d_v}_{=W_F},
		\]
		where $F$ is the underlying, non-decorated forest of planted rooted trees. To simplify the notation, we denote tensor products $x\otimes x'$ of elements of $S(D_E\otimes \calT(D_E\otimes D_V ))^{\otimes 2}$ under the form
		\[
			x\otimes x'=F\otimes F'\otimes W_F\otimes W_{F'},
		\]
		rather than
		\[
			x\otimes x'=(F\otimes W_F )\otimes (F'\otimes W_{F'}),
		\]
		and we shall use similar notations for superior tensor products of $S(D_E\otimes \calT(D_E\otimes D_V ))$.
	\end{notation}

	\begin{proposition}
		Let $\phi$ be a~tree-compatible map. We give the algebra $S(D_E\otimes \calT(D_E\otimes D_V ))$ a~coproduct $\Delta^\phi$, defined as follows, making it a~bialgebra:
		for any forest $x=F\otimes W_F$ of planted decorated rooted trees,
		\[
			\Delta^\phi(x)=\sum_{I\in \Up(F)} F_{\mid I}\otimes F_{\mid V(F)\setminus I}
			\otimes \prod_{\substack{e\in E(F),\\ s(e)\notin I,\: t(e)\in I}}^\circ \phi_{e,s(e)} (W_F ).
		\]
	\end{proposition}

	\begin{remark}
		The tree-compatibility is needed to correctly define this coproduct, as the order in $\prod_{\substack{e\in E(F),\\ s(e)\notin I,\: t(e)\in I}}^\circ \phi_{e,s(e)}$ does not matter if $\phi$ is tree-compatible.
	\end{remark}

	\begin{example}
		Let $a_1,a_2,a_3\in D_E$ and $b_1,b_2,b_3\in D_V$.
		\begin{align*}
			\Delta^\phi\left(
			\begin{array}{c}
				\xymatrix{\boite{b_1}\ar@{-}[rd]|{a_1}&\boite{b_2}\ar@{-}[d]|{a_2}\\&\boite{b_3}\ar@{-}[d]|{a_3}\\&\boite{\blackdiamond}}
			\end{array}
			\right)&=x\otimes 1+1\otimes x+\sum_i
			\begin{array}{c}
				\xymatrix{\boite{b_1}\ar@{-}[d]|{\phi_E^i (a_1 )}\\ \boite{\blackdiamond}}
			\end{array}
			\otimes
			\begin{array}{c}
				\xymatrix{\boite{b_2}\ar@{-}[d]|{a_2}\\\boite{\phi_V^i (b_3 )}\ar@{-}[d]|{a_3}\\ \boite{\blackdiamond}}
			\end{array}
			+\sum_i
			\begin{array}{c}
				\xymatrix{\boite{b_2}\ar@{-}[d]|{\phi_E^i (a_2 )}\\ \boite{\blackdiamond}}
			\end{array}
			\otimes
			\begin{array}{c}
				\xymatrix{\boite{b_1}\ar@{-}[d]|{a_1}\\\boite{\phi_V^i (b_3 )}\ar@{-}[d]|{a_3}\\ \boite{\blackdiamond}}
			\end{array}
			\\[2mm]
			&+\sum_{i,j}
			\begin{array}{c}
				\xymatrix{\boite{b_1}\ar@{-}[d]|{\phi_E^i (a_1 )}\\ \boite{\blackdiamond}}
			\end{array}
			\begin{array}{c}
				\xymatrix{\boite{b_2}\ar@{-}[d]|{\phi_E^j (a_2 )}\\ \boite{\blackdiamond}}
			\end{array}
			\otimes
			\begin{array}{c}
				\xymatrix{\boite{\phi_V^i\circ \phi_V^j (b_3 )}\ar@{-}[d]|{a_1}\\ \boite{\blackdiamond}}
			\end{array}
			.
		\end{align*}
	\end{example}

	\begin{proof}
		Let $x=F\otimes W_F\in S(D_E\otimes \calT(D_E\otimes D_V ))$.
		\begin{align*}
			(\id\otimes \Delta^\phi){ }&{ }\circ \Delta^\phi(x) \\
            &=\sum_{\substack{V(F)=I\sqcup J\sqcup K,\\ I\in \Up(F),\\ J\in \Up(F_{\mid J\sqcup K})}}F_{\mid I}\otimes F_{\mid J}\otimes F_{\mid K}\otimes \prod^\circ_{\substack{e\in E(F),\\ s(e)\in K,\: t(e)\in J}}
			\phi_{e,s(e)}\circ  \prod^\circ_{\substack{e\in E(F),\\ s(e)\in J\sqcup K,\: t(e)\in I}}\phi_{e,s(e)}(W_F ),
            \\
			(\Delta^\phi\otimes \id){ }&{ }\circ \Delta^\phi(x) \\
            &=\sum_{\substack{V(F)=I\sqcup J\sqcup K,\\ I\sqcup J\in \Up(F),\\ I\in \Up(F_{\mid I\sqcup J})}}F_{\mid I}\otimes F_{\mid J}\otimes F_{\mid K}\otimes \prod^\circ_{\substack{e\in E(F),\\ s(e)\in J,\: t(e)\in I}}
			\phi_{e,s(e)}\circ  \prod^\circ_{\substack{e\in E(F),\\ s(e)\in K,\: t(e)\in I\sqcup J}}\phi_{e,s(e)}(W_F ).
		\end{align*}
		If $V(F)=I\sqcup J\sqcup K$, then it is not difficult to show that
		\[
		I\in \Up(F) \mbox{ and }J\in \Up(F_{\mid J\sqcup K})\Longleftrightarrow
		I\sqcup J\in \Up(F) \mbox{ and }I\in \Up(F_{\mid I\sqcup J}).
		\]
		The tree-compatibility of $\phi$ implies the commutation of the maps $\phi_{e,s(e)}$, so
		\begin{align*}
			\prod^\circ_{\substack{e\in E(F),\\ s(e)\in K,\: t(e)\in J}}
			\phi_{e,s(e)}{ }&{ }\circ \prod^\circ_{\substack{e\in E(F),\\ s(e)\in J\sqcup K,\: t(e)\in I}}\phi_{e,s(e)} \\
			&=\prod^\circ_{\substack{e\in E(F),\\ s(e)\in K,\: t(e)\in J}}
			\phi_{e,s(e)}\circ  \prod^\circ_{\substack{e\in E(F),\\ s(e)\in K,\: t(e)\in I}}\phi_{e,s(e)}\circ  \prod^\circ_{\substack{e\in E(F),\\ s(e)\in J,\: t(e)\in I}}\phi_{e,s(e)}\\
			&=\prod^\circ_{\substack{e\in E(F),\\ s(e)\in J,\: t(e)\in I}}
			\phi_{e,s(e)}\circ  \prod^\circ_{\substack{e\in E(F),\\ s(e)\in K,\: t(e)\in I\sqcup J}}\phi_{e,s(e)}.
		\end{align*}
		So $\Delta^\phi$ is coassociative. If $F,G$ are two forests of planted rooted trees, then
		\[
			\Up(FG)=\{I\sqcup J\mid (I,J)\in \Up(F)\otimes \Up(G)\}.
		\]
		This gives the multiplicativity of $\Delta^\phi$.
	\end{proof}

	\begin{proposition}
		Let $\phi,\psi$ be two tree-compatible maps on $D_E\otimes D_V$ such that
		\[
			\psi_{13}\circ \phi_{23}=\phi_{23}\circ \psi_{13}.
		\]
		The map $\overline{\Theta}_\phi$ is a~morphism from the bialgebra $(S(D_E\otimes \calT(D_E,D_V )),m,\Delta^{\psi\circ \phi})$ to the bialgebra $(S(D_E\otimes \calT(D_E,D_V )),m,\Delta^\psi)$. It is an isomorphism if, and only if, $\phi$ is bijective.
	\end{proposition}

	\begin{proof}
		The map $\overline{\Theta}_\phi$ is obviously an algebra morphism for the product $m$. Let $F\otimes W_F$ be an element in $(S(D_E\otimes \calT(D_E,D_V ))$. Then,
		\begin{align*}
			\Delta^\psi\circ \overline{\Theta}_\phi(F\otimes W_F )&=\sum_{I\in \Up(F)} F_{\mid I}\otimes F_{\mid V(F)\setminus I}\otimes \underbrace{\prod^\circ_{\substack{e\in E(F),\\ s(e)\notin I,\:t(e)\in I}} \psi_{e,s(e)}\circ \prod^\circ_{e\in E(F)} \phi_{e,s(e)}}_{F_I}(W_F ).
		\end{align*}
		By definition of an upper part, if $e\in E(F)$, then if $s(e)\in I$, necessarily $t(e)\in I$. This gives, using the different commutations between the $\phi_{e,s(e)}$ and $\psi_{e,s(e)}$,
		\begin{align*}
			F_I &=\prod^\circ_{\substack{e\in E(F),\\ s(e)\notin I,\:t(e)\in I}} \psi_{e,s(e)}\circ \prod^\circ_{\substack{e\in E(F),\\ s(e)\notin I,\:t(e)\in I}} \phi_{e,s(e)} \circ \prod^\circ_{\substack{e\in E(F),\\ s(e)\in I,\:t(e)\in I}} \phi_{e,s(e)}
			\circ \prod^\circ_{\substack{e\in E(F),\\ s(e)\notin I,\:t(e)\notin I}} \phi_{e,s(e)}\\
			&=\prod^\circ_{\substack{e\in E(F),\\ s(e)\notin I,\:t(e)\in I}} (\psi\circ \phi)_{e,s(e)}\circ \prod^\circ_{\substack{e\in E(F),\\ s(e)\in I,\:t(e)\in I}} \phi_{e,s(e)}
			\circ \prod^\circ_{\substack{e\in E(F),\\ s(e)\notin I,\:t(e)\notin I}} \phi_{e,s(e)}\\
			&=\prod^\circ_{\substack{e\in E(F),\\ s(e)\in I,\:t(e)\in I}} \phi_{e,s(e)}\circ \prod^\circ_{\substack{e\in E(F),\\ s(e)\notin I,\:t(e)\notin I}} \phi_{e,s(e)}\circ \prod^\circ_{\substack{e\in E(F),\\ s(e)\notin I,\:t(e)\in I}} (\psi\circ \phi)_{e,s(e)}.
		\end{align*}
		Consequently,
		\begin{align*}
			\Delta^\psi\circ \overline{\Theta}_\phi(F\otimes W_F )&=\sum_{I\in \Up(F)} F_{\mid I}\otimes F_{\mid V(F)\setminus I}\\
			&\otimes\prod^\circ_{\substack{e\in E(F),\\ s(e)\in I,\:t(e)\in I}} \phi_{e,s(e)}\circ \prod^\circ_{\substack{e\in E(F),\\ s(e)\notin I,\:t(e)\notin I}} \phi_{e,s(e)}\circ \prod^\circ_{\substack{e\in E(F),\\ s(e)\notin I,\:t(e)\in I}} (\psi\circ \phi)_{e,s(e)}(W_F )\\
			&=(\overline{\Theta}_\phi\otimes \overline{\Theta}_\phi)\left(\sum_{I\in \Up(F)} F_{\mid I}\otimes F_{\mid V(F)\setminus I}\otimes \prod^\circ_{\substack{e\in E(F),\\ s(e)\notin I,\:t(e)\in I}} (\psi\circ \phi)_{e,s(e)}(W_F )\right)\\
			&=(\overline{\Theta}_\phi\otimes \overline{\Theta}_\phi)\circ \Delta^{\psi\circ \phi}(F\otimes W_F ).
		\end{align*}
		So $\overline{\Theta}_\phi$ is a~coalgebra morphism.
	\end{proof}

	A particular case is $\phi=\id_{D_E\otimes D_V}$. In this case, we simply write $\Delta$ instead of $\Delta^{\id_{D_E\otimes D_V}}$. We obtain:

	\begin{corollary}
		Let $\phi$ be a~tree-compatible map $D_E\otimes D_V$. The map $\overline{\Theta}_\phi$ is a~bialgebra morphism from $(S(D_E\otimes \calT(D_E,D_V )),m,\Delta^\phi)$ to $(S(D_E\otimes \calT(D_E,D_V )),m,\Delta)$. It is an isomorphism if, and only if, $\phi$ is bijective.
	\end{corollary}

	\subsection{Duality}

	When $D_E$ and $D_V$ are finite-dimensional, we can identify $(D_E\otimes D_V )^*$ and $D_E^*\otimes D_V^*$. We can then transpose any map $\phi:D_E\otimes D_V\longrightarrow D_E\otimes D_V$ into a~map
	$\phi^*:D_E^*\otimes D_V^*\longrightarrow D_E^*\otimes D_V^*$. As $(\phi_{13})^* =\phi^*_{13}$ and $(\phi_{23})^* =\phi^*_{23}$, $\phi$ is tree-compatible if, and only if, $\phi^*$ is tree-compatible. If so, $(S(D_E\otimes \calT(D_E\otimes D_V ),\star^\phi,\Delta)$ is a~graded bialgebra, whose graduation is given by the number of vertices. If $D_E$ and $D_V$ are finite-dimensional, the components of this graded bialgebra are finite-dimensional. As a~consequence, the graded dual, which we identify with $S(D_E^*\otimes \calT(D_E^*\otimes D_V^* ))$, is also a~graded bialgebra, isomorphic to $(S(D_E^*\otimes \calT(D_E^*\otimes D_V^* )),m,\Delta^{\phi^*})$, as we shall see in Theorem~\ref{theodualite}.

	More generally, we now consider two tree-compatible maps $\phi:D_E\otimes D_V\longrightarrow D_E\otimes D_V$ and $\phi':D'_E\otimes D'_V\longrightarrow D'_E\otimes D'_V$, and a~pairing $\langle-,-\rangle:(D'_E\otimes D'_V )\times (D_E\otimes D_V )\longrightarrow \K$ such that for any $a\in D_E$, $b\in D_V$, $a'\in D'_E$, $b'\in D'_V$,
	\[
		\langle \phi'(a'\otimes b'),a\otimes b\rangle=\langle a'\otimes b',\phi(a\otimes b)\rangle.
	\]
	We define a~pairing between $S(D'_E\otimes \calT(D'_E,D'_V ))$ and $S(D_E\otimes \calT(D_E,D_V ))$ as follows. If $x=F\otimes W_F \in S(D_E\otimes \calT(D_E,D_V ))$ and $x'=F'\otimes W'_F \in S(D'_E\otimes \calT(D'_E,D'_V ))$,
	\[
		\langle x',x\rangle=\sum_{\sigma \in \iso(F',F)}\prod_{e\in E(F')}\underbrace{\langle d'_e\otimes d'_{t(e)}, d_{\sigma(e)}\otimes d_{\sigma(t(e))}\rangle}_{\langle W'_F,W_F\rangle_\sigma}.
	\]
	Note that in a~planted forest $F'$, $t$ is a~bijection from $E(F')$ to $V(F')$; hence, in $\langle W'_F,W_F\rangle_\sigma$, each $d'_e$ for $e\in E(F')$, $d_e$ for $e\in E(F)$, $d'_v$ for $v\in V(F')$, $d_v$ for $v\in V(F)$, appears exactly once. If the pairing $\langle-,-\rangle$ between $D'_E\otimes D'_V$ and $D_E\otimes D_V$ is non-degenerate, then this extension is non-degenerate.

	\begin{example}
		Let $a_1,a_2,a_3\in D_E$, $a'_1,a'_2,a'_3 \in D'_E$, $b_1,b_2,b_3 \in D_V$, and $b'_1,b'_2,b'_3\in D'_V$.
		\begin{align*}
			\langle
			\begin{array}{c}
				\xymatrix{\boite{b'_1}\ar@{-}|{a'_1}[d]\\\boite{b'_2}\ar@{-}|{a'_2}[d]\\\boite{b'_3}\ar@{-}|{a'_3}[d]\\\boite{\blackdiamond}}
			\end{array}
			,
			\begin{array}{c}
				\xymatrix{\boite{b_1}\ar@{-}|{a_1}[d]\\\boite{b_2}\ar@{-}|{a_2}[d]\\\boite{b_3}\ar@{-}|{a_3}[d]\\\boite{\blackdiamond}}
			\end{array}
			\rangle&=\langle a'_1\otimes b'_1,a_1\otimes b_1\rangle\langle a'_2\otimes b'_2,a_2\otimes b_2\rangle\langle a'_3\otimes b'_3,a_3\otimes b_3\rangle,\\
			\langle
			\begin{array}{c}
				\xymatrix{\boite{b'_1}\ar@{-}|{a'_1}[rd]&\boite{b'_2}\ar@{-}|{a'_2}[d]\\&\boite{b'_3}\ar@{-}|{a'_3}[d]\\&\boite{\blackdiamond}}
			\end{array}
			,
			\begin{array}{c}
				\xymatrix{\boite{b_1}\ar@{-}|{a_1}[rd]&\boite{b_2}\ar@{-}|{a_2}[d]\\&\boite{b_3}\ar@{-}|{a_3}[d]\\&\boite{\blackdiamond}}
			\end{array}
			\rangle&=\langle a'_1\otimes b'_1,a_1\otimes b_1\rangle\langle a'_2\otimes b'_2,a_2\otimes b_2\rangle\langle a'_3\otimes b'_3,a_3\otimes b_3\rangle\\
			&+\langle a'_1\otimes b'_1,a_2\otimes b_2\rangle\langle a'_2\otimes b'_2,a_1\otimes b_1\rangle\langle a'_3\otimes b'_3,a_3\otimes b_3\rangle.
		\end{align*}
	\end{example}

	\begin{theorem}\label{theodualite}
		The pairing $\langle-,-\rangle$ is a~Hopf pairing between $(S(D'_E\otimes \calT(D'_E\otimes D'_V )),\star^{\phi'},\Delta)$ and $(S(D_E\otimes \calT(D_E\otimes D_V )),m,\Delta^\phi)$, that is to say for any $x',y'\in S(D'_E\otimes \calT(D'_E\otimes D'_V ))$, for any $x,y \in (S(D_E\otimes \calT(D_E\otimes D_V ))$,
		\begin{align*}
			\langle 1,x\rangle&=\varepsilon(x),&\langle x',1\rangle&=\varepsilon(x'),\\
			\langle x'\star^{\phi'} y',x\rangle&=\langle x'\otimes y',\Delta^\phi(x)\rangle,&\langle \Delta(x'),x\otimes y\rangle&=\langle x',xy\rangle.
		\end{align*}
	\end{theorem}

	\begin{proof}
		Let us consider $x'=F'\otimes W'_F$ and $y'=G'\otimes W'_G$ in $S(D'_E\otimes \calT(D'_E\otimes D'_V ))$ and $x=H\otimes W_H\in S(D_E\otimes \calT(D_E\otimes D_V ))$. We consider the two sets
		\begin{align*}
			A&=\{(I,\sigma_F,\sigma_G )\mid I\in \Up(H),\: \sigma_F\in \iso(F',H_{\mid I}),\: \sigma_G\in \iso(G',H_{\mid V(H)\setminus I})\},\\
			B&=\{(g,\sigma)\mid g\in \bfG(F',G'),\: \sigma \in \iso(F'\curvearrowright_g G'),H)\}.
		\end{align*}
		If $g\in \bfG(F',G')$, then $V(F')$ is an upper part of $F'\curvearrowright_g G'$. Therefore, we define a~map $\Upsilon:A\longrightarrow B$, sending $(g,\sigma)$ to $(\sigma(V(F')),\sigma_{\mid V(F')},\sigma_{\mid V(G')})$. Its is clearly injective. If $(I,\sigma_F,\sigma_F )\in A$, let us define $g\in \bfG(F',G')$ as follows. If $T'\in \bfT(F')$, let us denote its undecorated root by $r(T)$. Two cases are possible.
		\begin{enumerate}
			\item If $\sigma_F (r(T))$ is a~root of $H$, we put $g(T)=0$.
			\item If $\sigma_F (r(T))$ is not a~root of $H$, there exists a~unique edge $e\in E(H)$ such that $t(e)=\sigma_F (r(T))$. As $I$ is an upper part of $H$, $s(e)\notin I$. We put $g(T)=\sigma_G^{-1}(s(e))$.
		\end{enumerate}
		Then $\sigma=\sigma_F\sqcup \sigma_G:V(F')\sqcup V(G')=V(F'\curvearrowright G)'\longrightarrow V(H)$ is an isomorphism and $\Upsilon(g,\sigma)=(I,\sigma_F,\sigma_G )$. So $\Upsilon$ is a~bijection. We obtain
		\begin{align*}
			\langle x'\star^{\phi'} y',x\rangle&=\sum_{g\in \bfG(F',G')} \langle F\curvearrowright_g G,\prod^\circ_{\substack{T\in \bfT(F'),\\ g(T)\neq 0}} \phi'_{e(T),g(T)}(W'_F\otimes W'_G ),H\otimes W_H\rangle\\
			&=\sum_{(g,\sigma)\in B} \langle\prod^\circ_{\substack{T\in \bfT(F'),\\ g(T)\neq 0}} \phi'_{e(T),g(T)}(W'_F\otimes W'_G ),W_H \rangle_\sigma\\
			&=\sum_{(g,\sigma)\in B} \langle W'_F\otimes W'_G, \prod^\circ_{\substack{T\in \bfT(F'),\\ g(T)\neq 0}}\phi_{\sigma(e(T)),\sigma(g(T))}(W_H )\rangle_\sigma\\
			&=\sum_{(I,\sigma_F,\sigma_G )\in A}\langle W'_F\otimes W'_G,\prod^\circ_{\substack{e\in E(H),\\ s(e)\notin I,\:t(e)\in I}} \phi_{e,s(e)}(W_H )\rangle\\
			&=\sum_{I\in \Up(H)} \langle F'\otimes W'_F\otimes G'\otimes W'_G, H_{\mid H}\otimes H_{\mid V(H)\setminus I}\otimes \prod^\circ_{\substack{e\in E(H),\\ s(e)\notin I,\:t(e)\in I}} \phi_{e,s(e)}(W_H )\rangle\\
			&=\langle x'\otimes y',\Delta^\phi(x)\rangle.
		\end{align*}
		Let $x'=F'\otimes W'_F\in S(D'_E\otimes \calT(D'_E\otimes D'_V ))$, $y=G\otimes W_G \in S(D_E\otimes \calT(D_E\otimes D_V ))$, and $x=F\otimes W_F$. Then,
		\begin{align*}
            \langle x',m(x\otimes y)\rangle&=\langle F'\otimes W'_F,FG\otimes W_F\otimes W_G\rangle\\
            &=\sum_{\sigma \in \iso(F',FG)}\langle W'_F,W_FW_G\rangle\\
            &=\sum_{\bfT(F')=\bfT_1\sqcup \bfT_2}\langle \prod_{T'\in \bfT_1} T'\otimes W'_T,F\otimes W_F\rangle\langle \prod_{T'\in \bfT_2} T'\otimes W'_T,G\otimes W_G\rangle\\
            &=\langle \Delta(x'),x\otimes y\rangle.
        \end{align*}
		So $\langle-,-\rangle$ is a~Hopf pairing.
	\end{proof}

	\section{Examples}

	\subsection{An example for stochastic PDEs, without noise}

	In this paragraph, we consider the case described in~\cite{Bruned2023}, firstly without noise.
	\begin{notation}
		We fix $d\in \N$ and put
		\[
			\DEPDE=\DVPDE=\vect\left(\N^{d+1}\right).
		\]
		Elements of $\N^{d+1}$ will be written under the form $a=(a_0,\ldots,a_d )$. For $j\in \llbracket 0;d\rrbracket$, we denote by $\epsilon^{(j)}$ the element of $\N^{d+1}$ defined by
		\begin{align*}
			&\forall i\in \llbracket 0;d\rrbracket,&\epsilon^{(j)}_i &=\delta_{i,j}.
		\end{align*}
	\end{notation}

	\begin{definition}
		The map $\partial^{(j)}$ is defined by
		\[
			\partial^{(j)}:\left\{
			\begin{array}{rcl}
				\DEPDE\otimes \DVPDE&\longrightarrow&\DEPDE\otimes \DVPDE\\
				a\otimes b&\longmapsto&b_j\left(a-\epsilon^{(j)}\right)\otimes \left(b-\epsilon^{(j)}\right),
			\end{array}
			\right.
		\]
		with the convention that for any $c\in \N^{d+1}$, $c-\epsilon^{(j)}=0$ if $c_j =0$.
	\end{definition}

	\begin{lemma}
		For any $i,j\in \llbracket 0;d\rrbracket$, $\partial^{(i)}\circ \partial^{(j)}=\partial^{(j)}\circ \partial^{(i)}$.
	\end{lemma}

	\begin{proof}
		Let us assume that $i\neq j$. Let $a,b\in \N^{d+1}$.
		\begin{align*}
			\partial^{(i)}\circ \partial^{(j)}(a\otimes b)&=b_j \partial^{(i)}\left(\left(a-\epsilon^{(j)}\right)\otimes \left(b-\epsilon^{(j)}\right)\right)\\
			&=b_i b_j\left(a-\epsilon^{(j)}-\epsilon^{(i)}\right)\otimes \left(b-\epsilon^{(j)}-\epsilon^{(i)}\right)=\partial^{(j)}\circ \partial^{(i)}(a\otimes b). \qedhere
		\end{align*}
	\end{proof}

	\begin{proposition}
		For any $\lambda=(\lambda_0,\ldots,\lambda_d )\in \K^{d+1}$, $\partial^\lambda=\sum_{i=0}^d \lambda_i \partial^{(i)}$ is tree-compatible.
	\end{proposition}

	\begin{proof}
		Let us first prove that for any $i,j\in \llbracket 0;d\rrbracket$, $\partial^{(i)}_{13}\circ \partial^{(j)}_{23}=\partial^{(j)}_{23}\circ \partial^{(i)}_{13}$. Let $a,a',b\in \N^{d+1}$.
		\begin{align*}
			\partial^{(i)}_{13}\circ \partial^{(j)}_{23}(a\otimes a'\otimes b)&=b_j\partial_{13}^{(i)}\left(a\otimes (a'-\epsilon^{(j)})\otimes (b-\epsilon^{(j)})\right)\\
			&=
			\begin{cases}
				b_i b_j \left(a-\epsilon^{(i)}\right)\otimes (a'-\epsilon^{(j)})\otimes \left(b-\epsilon^{(j)}-\epsilon^{(i)}\right) \mbox{ if }i\neq j,\\
				b_i (b_i -1) \left(a-\epsilon^{(i)}\right)\otimes \left(a'-\epsilon^{(i)}\right)\otimes \left(b-2\epsilon^{(j)}\right) \mbox{ if }i=j;
			\end{cases}
			\\
			\partial^{(j)}_{23}\circ \partial^{(i)}_{13}(a\otimes a'\otimes b)&=b_i\partial_{23}^{(j)}\left((a-\epsilon^{(i)})\otimes a'\otimes (b-\epsilon^{(i)})\right)\\
			&=
			\begin{cases}
				b_j b_i \left(a-\epsilon^{(i)}\right)\otimes (a'-\epsilon^{(j)})\otimes \left(b-\epsilon^{(j)}-\epsilon^{(i)}\right) \mbox{ if }i\neq j,\\
				b_i (b_i -1) \left(a-\epsilon^{(i)}\right)\otimes \left(a'-\epsilon^{(i)}\right)\otimes \left(b-2\epsilon^{(j)}\right) \mbox{ if }i=j.
			\end{cases}
		\end{align*}
		So $\partial^{(i)}_{13}\circ \partial^{(j)}_{23}=\partial^{(j)}_{23}\circ \partial^{(i)}_{13}$. By Proposition~\ref{prop2.3}, $\phi^\lambda$ is tree-compatible for any $\lambda\in \K^{d+1}$.
	\end{proof}

	Let us compute the maps $(\partial^\lambda)^n$, which are tree-compatible by Proposition~\ref{prop2.4}.
	\begin{notation}
		\begin{enumerate}
			\item If $a,b \in \N^{d+1}$, we shall say that $a\leq b$ if for any $i\in \llbracket 0;d\rrbracket$, $a_i\leq b_i$. If so, we put
			\[
				\binom{b}a=\prod_{i=0}^d \binom{b_i}{a_i}=\prod_{i=0}^d \frac{b_i!}{a_i!(b_i -a_i )!}.
			\]
			\item If $a,b\in \N^{d+1}$, we put $\min(a,b)=(\min(a_0,b_0 ),\ldots,\min(a_n,b_n ))$.
			\item For any $a\in \N^{d+1}$, we put $|a|=a_0 +\cdots+a_d$.
			\item For any $\lambda=(\lambda_0,\ldots,\lambda_d )\in \K^{d+1}$ and any $a=(a_0,\ldots,a_d )\in \N^{d+1}$, we put
			\[
				\lambda^a =\prod_{i=0}^d\lambda_i^{a_i},
			\]
			with the convention $0^0 =1$.
		\end{enumerate}
	\end{notation}

	\begin{lemma}
		For any $\lambda\in \K^{d+1}$, for any $n\in \N$, and for any $ a,b\in \N^{d+1}$,
		\begin{align*}
			(\partial^\lambda)^n (a\otimes b)&=n!\sum_{\substack{l\in \N^{d+1},\\ l\leq \min(a,b),\\ |l|=n}} \lambda^l \binom{b}{l} (a-l)\otimes (b-l).
		\end{align*}
	\end{lemma}

	\begin{proof}
		By induction on $n$. If $n=0$, this is obvious. If $n=1$,
		\begin{align*}
			\partial^\lambda(a\otimes b)&=\sum_{\substack{i\in \llbracket 0;d\rrbracket,\\ a_i,b_i\geq 1}} \lambda_i b_i\left(a-\epsilon^{(i)}\right)\otimes \left(b-\epsilon^{(i)}\right)\\
			&=\sum_{\substack{i\in \llbracket 0;d\rrbracket,\\ a_i,b_i\geq 1}} \lambda^{\epsilon^{(i)}}\binom{b}{\epsilon^{(i)}}\left(a-\epsilon^{(i)}\right)\otimes \left(b-\epsilon^{(i)}\right)\\
			&=\sum_{\substack{l\in \N^{d+1},\\ l\leq \min(a,b),\\ |l|=1}} \lambda^l \binom{b}{l} (a-l)\otimes (b-l).
		\end{align*}
		So the results holds at rank $1$. Let us assume it at rank $n$.
		\begin{align*}
			(\partial^\lambda)^{n+1}(a\otimes b)&=n! \partial^\lambda\left(\sum_{\substack{l\in \N^{d+1},\\ l\leq \min(a,b),\\ |l|=n}} \lambda^l \binom{b}{l} (a-l)\otimes (b-l)\right)\\
			&=n!\sum_{\substack{l\in \N^{d+1},\\ l\leq \min(a,b),\\ |l|=n}} \sum_{\substack{i\in \llbracket 0;d\rrbracket,\\ a_i -l_i, b_i -l_i\geq 1}}\lambda^l \lambda_i\binom{b}{l}(b_i -l_i )\left(a-l-\epsilon^{(i)}\right)\otimes \left(b-l-\epsilon^{(i)}\right)\\
			&=n! \sum_{\substack{l\in \N^{d+1},\\ l\leq \min(a,b),\\ |l|=n+1}}\lambda^l\left(\sum_{\substack{i\in \llbracket 0;d\rrbracket,\\l_i\geq 1}}\binom{b}{l-\epsilon^{(i)}}(b_i -l_i +1)\right) (a-l)\otimes (b-l).
		\end{align*}
		Let $l\in \N^{d+1}$, such that $l\leq b$. Then
		\begin{align*}
			\sum_{\substack{i\in \llbracket 0;d\rrbracket,\\l_i\geq 1}}\binom{b}{l-\epsilon^{(i)}}(b_i -l_i +1)&=\sum_{\substack{i\in \llbracket 0;d\rrbracket,\\l_i\geq 1}}\prod_{j\neq i} \frac{b_j!}{l_j!(b_j -l_j )!} \frac{b_i!}{(l_i -1)!(b_i -l_i +1)!}(b_i -l_i +1)\\
			&=\sum_{\substack{i\in \llbracket 0;d\rrbracket,\\l_i\geq 1}}\prod_{j\neq i} \frac{b_j!}{l_j!(b_j -l_j )!} \frac{b_i!}{l_i!(b_i -l_i )!}l_i\\
			&=\binom{b}{l} \sum_{\substack{i\in \llbracket 0;d\rrbracket,\\l_i\geq 1}}l_i\\
			&=\binom{b}{l}|l|.
		\end{align*}
		We obtain
		\begin{align*}
			(\partial^\lambda)^{n+1}(a\otimes b)&=n! \sum_{\substack{l\in \N^{d+1},\\ l\leq \min(a,b),\\ |l|=n+1}}\lambda^l (n+1)\binom{b}{l}(a-l)\otimes (b-l)\\
			&=(n+1)! \sum_{\substack{l\in \N^{d+1},\\ l\leq \min(a,b),\\ |l|=n+1}}\lambda^l\binom{b}{l}(a-l)\otimes (b-l). \qedhere
		\end{align*}
	\end{proof}

	The maps $\partial^{(i)}$ are locally nilpotent: if $a\in \N^{d+1}$,
	\[
		\left(\partial^{(i)}\right)^{a_i +1}(a)=0.
	\]
	As they pairwise commute, any linear span of $\partial^{(i)}$ is also locally nilpotent, so for any $\lambda\in \K^{d+1}$, $\partial^\lambda$ is locally nilpotent: we can use Proposition~\ref{prop2.4} with $\phi=\partial^\lambda$.

	\begin{proposition}
		For any $\lambda\in \K^{d+1}$, we put
		\[
			\phi^\lambda=\exp\left(\partial^\lambda\right)=\exp\left(\lambda_0\partial^{(0)}\right)\circ \cdots \circ \exp\left(\lambda_d\partial^{(d)}\right).
		\]
		Then $\phi^\lambda$ is tree-compatible. Moreover, for any $a,b\in \K^{d+1}$,
		\[
			\phi^\lambda(a\otimes b)=\sum_{l\leq \min(a,b)}\lambda^l \binom{b}{l}(a-l)\otimes (b-l).
		\]
		For any $\lambda,\mu\in \K^{d+1}$, $\phi^\lambda\circ \phi^\mu=\phi^{\lambda+\mu}$. The map $\phi^\lambda$ is invertible, of inverse $\phi^{-\lambda}$.
	\end{proposition}

	\begin{proof}
		By Proposition~\ref{prop2.4}, $\phi^\lambda$ is tree-compatible for any $\lambda\in \K^{d+1}$. If $\lambda,\mu\in \K^{d+1}$, as the $\partial^{(i)}$ pairwise commute,
		\begin{align*}
			\phi^\lambda \circ \phi^\mu&=\exp\left(\lambda_0\partial^{(0)}\right)\circ \cdots \circ \exp\left(\lambda_d\partial^{(d)}\right)\circ \exp\left(\mu_0\partial^{(0)}\right)\circ \cdots \circ \exp\left(\mu_d\partial^{(d)}\right)\\
			&=\exp\left(\lambda_0\partial^{(0)}\right)\circ \exp\left(\mu_0\partial^{(0)}\right)\circ \cdots \circ \exp\left(\lambda_d\partial^{(d)}\right)\circ \exp\left(\mu_d\partial^{(d)}\right)\\
			&=\exp\left((\lambda_0 +\mu_0 )\partial^{(0)}\right)\circ \cdots \circ \exp\left((\lambda_d +\mu_d )\partial^{(d)}\right)\\
			&=\phi^{\lambda+\mu}.
		\end{align*}
		As $\phi^0 =\id_{\DEPDE\otimes \DVPDE}$, $\phi^\lambda$ is invertible and $(\phi^\lambda)^{-1}=\phi^{-\lambda}$.
	\end{proof}

	From Theorem~\ref{theo3.3}:

	\begin{theorem}\label{theo4.6}
		The $\DEPDE$-multiple pre-Lie algebra $(\calT(\DEPDE,\DVPDE),\triangleleft^{\phi^\lambda})$ is freely generated by the elements $\tun\otimes b=\xymatrix{\boite{b}}$, with $b\in \DVPDE$. Moreover, $\Theta_{\phi^\lambda}$ is an isomorphism from $(\calT(\DEPDE,\DVPDE),\triangleleft)$ to $(\calT(\DEPDE,\DVPDE),\triangleleft^{\phi^\lambda})$, of inverse $\Theta_{\phi^{-\lambda}}$.
	\end{theorem}

	When $\lambda=(1,\ldots,1)$, we recover Theorem 2.7 of~\cite{Bruned2023-1}, with an explicit isomorphism.

	\begin{example}
		If $a,a_1,a_2,a_3,b_1,b_2,b_3,b_4,b_5\in \N^{d+1}$,
		\begin{align*}
			\begin{array}{c}\xymatrix{\boite{b_2}\ar@{-}[d]|{a_1}\\ \boite{b_1}}
			\end{array}
			\triangleleft_a^{\phi^\lambda}
			\begin{array}{c}\xymatrix{\boite{b_4}\ar@{-}[rd]|{a_2}&\boite{b_5}\ar@{-}[d]|{a_3}\\&\boite{b_3}}
			\end{array}
			&=\sum_{l\leq \min(a,b_4 )}\lambda^l \binom{b_4}{l}
			\begin{array}{c}\xymatrix{\boite{b_2}\ar@{-}[d]|{a_1}&\\ \boite{b_1}\ar@{-}[d]|{a-l}&\\\boite{b_4 -l}\ar@{-}[rd]|{a_2}&\boite{b_5}\ar@{-}[d]|{a_3}\\&\boite{b_3}}
			\end{array} \\
			&{\quad}+\sum_{l\leq \min(a,b_5 )}\lambda^l \binom{b_5}{l}
			\begin{array}{c}\xymatrix{&\boite{b_2}\ar@{-}[d]|{a_1}\\&\boite{b_1}\ar@{-}[d]|{a-l}\\\boite{b_4}\ar@{-}[rd]|{a_2}&\boite{b_5 -l}\ar@{-}[d]|{a_3}\\&\boite{b_3}}
			\end{array} \\
			&{\quad}+\sum_{l\leq \min(a,b_3 )}\lambda^l \binom{b_3}{l}
			\begin{array}{c}\xymatrix{&&\boite{b_2}\ar@{-}[d]|{a_1}\\\boite{b_4}\ar@{-}[rd]|{a_2}&\boite{b_5}\ar@{-}[d]|{a_3}&\boite{b_1}\ar@{-}[ld]|{a-l}\\&\boite{b_3 -l}&}
			\end{array}
			.
		\end{align*}
		Moreover,
		\begin{align*}
			\Theta_{\phi^\lambda}\left(
			\begin{array}{c}\xymatrix{\boite{b_2}\ar@{-}[d]|{a_1}\\\boite{b_1}}
			\end{array}
			\right)
			&=\sum_{l\leq \min(a_1,b_1 )} \lambda^l\binom{b_1}{l}
			\begin{array}{c}\xymatrix{\boite{b_2}\ar@{-}[d]|{a_1 -l}\\\boite{b_1 -l}}
			\end{array}
			,\\[3mm]
			\Theta_{\phi^\lambda}\left(
			\begin{array}{c}\xymatrix{\boite{b_3}\ar@{-}[d]|{a_2}\\\boite{b_2}\ar@{-}[d]|{a_1}\\\boite{b_1}}
			\end{array}
			\right)
			&=\sum_{\substack{l_1\leq \min(a_1,b_1 ),\\ l_2\leq \min(a_2,b_2 )}} \lambda^{l_1 +l_2}\binom{b_1}{l_1}\binom{b_2}{l_2}
			\begin{array}{c}\xymatrix{\boite{b_3}\ar@{-}[d]|{a_2 -l_2}\\\boite{b_2 -l_2}\ar@{-}[d]|{a_1 -l}\\\boite{b_1 -l_1}}
			\end{array}
			,\\[3mm]
			\Theta_{\phi^\lambda}\left(
			\begin{array}{c}\xymatrix{\boite{b_2}\ar@{-}[d]|{a_1}&\boite{b_3}\ar@{-}[dl]|{a_2}\\\boite{b_1}}
			\end{array}
			\right)
			&=\sum_{\substack{l_1\leq a_1,\\ l_2\leq a_1,\\ l_1 +l_2\leq b_1}} \lambda^{l_1 +l_2}\binom{b_1}{l_1}\binom{b_1 -l_1}{b_2}
			\begin{array}{c}\xymatrix{\boite{b_2}\ar@{-}[d]|{a_1 -l_1}&\boite{b_3}\ar@{-}[dl]|{a_2 -l_2}\\\boite{b_1 -l_1 -l_2}}
			\end{array}
			.
		\end{align*}
	\end{example}

	\begin{proposition}
		We consider the case where $\lambda_i =1$ for any $i\in \llbracket 0;d\rrbracket$. Let us consider $P=\vect(X_i\mid i\in \llbracket 0;d\rrbracket)$. We then define $\psi_V$ and $\psi_E$ by
		\begin{align*}
			&\forall i\in \llbracket 0;d\rrbracket,\:\forall a,b\in \N^{d+1},&\psi_V (X_i )(b)&=b+\epsilon^{(i)},&\psi_E (X_i )(a)&=a-\epsilon^{(i)},
		\end{align*}
		with the convention that $a-\epsilon^{(i)}=0$ if $a_i =0$. Then $(\psi_E,\psi_V )$ is $\phi^\lambda$-compatible, in the sense of Corollary~\ref{cor3.10}.
	\end{proposition}

	\begin{proof}
		Only \eqref{eq11} is not immediate. Let $a,b\in \N^{d+1}$ and let $i\in \llbracket 0;d\rrbracket$.
		\begin{align*}
			&\phi\circ (\id_{\DEPDE}\otimes \psi_V (X_i ))(a\otimes b)-(\id_{\DEPDE}\otimes \psi_V (X_i ))\circ \phi(a\otimes b)\\
			&=\sum_{l\leq \min(a,b+\epsilon^{(i)})}\binom{b+\epsilon^{(i)}}{l}(a-l)\otimes(b-l+\epsilon^{(i)})-\sum_{l\leq \min(a,b)} \binom{b}{l} (a-l)\otimes(b-l+\epsilon^{(i)})\\
			&=\sum_{l\leq \min(a,b)}\left(\binom{b+\epsilon^{(i)}}{l}- \binom{b}{l}\right)(a-l)\otimes(b-l+\epsilon^{(i)})\\
			&+\sum_{\substack{l\leq \min(a,b+\epsilon^{(i)}),\\ l_i =b_i +1}} \binom{b+\epsilon^{(i)}}{l}(a-l)\otimes(b-l+\epsilon^{(i)}).
		\end{align*}
		If $l\leq \min(a,b)$,
		\begin{align*}
			\binom{b+\epsilon^{(i)}}{l}- \binom{b}{l}&=\prod_{j\neq i}\binom{b_j}{l_j}\left(\frac{(b_i +1)!}{l_i!(b_i -l_i +1)!}-\frac{b_i!}{l_i!(b_i -l_i )!}\right)\\
			&=\binom{b}{l}\left(\frac{b_i +1}{b_i -l_i +1}-1\right)\\
			&=\frac{l_i}{b_i -l_i +1}\binom{b}{l}.
		\end{align*}
		If $l\leq \min\left(a,b+\epsilon^{(i)}\right)$, with $l_i =b_i +1$,
		\[
		\binom{b+\epsilon^{(i)}}{l}=\prod_{j\neq i}\binom{b_j}{l_j}.
		\]
		We finally obtain
		\begin{align}
			\nonumber&\phi\circ (\id_{\DEPDE}\otimes \psi_V (X_i ))(a\otimes b)-(\id_{\DEPDE}\otimes \psi_V (X_i ))\circ \phi(a\otimes b)\\
			\nonumber&=\sum_{l\leq \min(a,b)}\frac{l_i}{b_i -l_i +1}\binom{b}{l} (a-l)\otimes(b-l+\epsilon^{(i)})\\
			\nonumber&+\sum_{\substack{l\leq \min(a,b+\epsilon^{(i)}),\\ l_i =b_i +1}} \prod_{j\neq i}\binom{b_j}{l_j} (a-l)\otimes(b-l+\epsilon^{(i)})\\
			\nonumber&=\sum_{\substack{l\leq \min(a,b),\\ l_i\geq 1}}\frac{l_i}{b_i -l_i +1}\binom{b}{l} (a-l)\otimes(b-l+\epsilon^{(i)})\\
			\label{EQ9}&+\sum_{\substack{l\leq \min(a,b+\epsilon^{(i)}),\\ l_i =b_i +1}} \prod_{j\neq i}\binom{b_j}{l_j} (a-l)\otimes(b-l+\epsilon^{(i)}).
		\end{align}
		Let us firstly assume that $a_i =0$. Then, in \eqref{EQ9}, both sums are equal to $0$. Moreover, $\psi_E (X_i )(a)=0$, so
		\[
			\phi \circ (\psi_E (X_i )\otimes \id_{\DVPDE})(a\otimes b)=0.
		\]
		Let us now assume that $a_i\geq 1$. Then
		\begin{align*}
			\phi \circ (\psi_E (X_i )\otimes \id_{\DVPDE})(a\otimes b)&=\sum_{\substack{l\leq \min(a,b+\epsilon^{(i)}),\\ l_i\geq 1}} \binom{b}{l-\epsilon^{(i)}}(a-l)\otimes (b-l+\epsilon^{(i)}).
		\end{align*}
		If $l\leq \min(a,b)$ and $l_i\geq 1$, then
		\begin{align*}
			\binom{b}{l-\epsilon^{(i)}}&=\prod_{j\neq i}\binom{b_j}{l_j} \frac{b_i!}{(l_i -1)!(b_i -l_i +1)!}=\frac{l_i}{b_i -l_i +1}\binom{b}{l}.
		\end{align*}
		If $l\leq \min(a,b)$ and $l_i =b_i +1$, then
		\begin{align*}
			\binom{b}{l-\epsilon^{(i)}}&=\prod_{j\neq i}\binom{b_j}{l_j}.
		\end{align*}
		We obtain
		\begin{align}
			\nonumber \phi \circ (\psi_E (X_i )\otimes \id_{\DVPDE})(a\otimes b)&=\sum_{\substack{l\leq \min(a,b),\\ l_i\geq 1}}\frac{l_i}{b_i -l_i +1}\binom{b}{l} (a-l)\otimes(b-l+\epsilon^{(i)})\\
			\label{EQ10}&+\sum_{\substack{l\leq \min(a,b+\epsilon^{(i)}),\\ l_i =b_i +1}} \prod_{j\neq i}\binom{b_j}{l_j} (a-l)\otimes(b-l+\epsilon^{(i)}).
		\end{align}
		Comparing \eqref{EQ9} and \eqref{EQ10}, we obtain the result.
	\end{proof}

	Therefore, we obtain a~post-Lie algebra, as mentioned in~\cite{Bruned2023}.

	\begin{example}
		If $p\in P$ and $a_1,a_2,a_3,b_1,b_2,b_3\in \N^{d+1}$ and $i\in \llbracket 0;d\rrbracket$,
		\begin{align*}
			X_i\triangleleft
			\begin{array}{c}\xymatrix{\boite{b_2}\ar@{-}[rd]|{a_2}&\boite{b_3}\ar@{-}[d]|{a_3}\\&\boite{b_1}\\&\boite{\blackdiamond} \ar@{-}[u]|{a_1}}
			\end{array}
			&=
			\begin{array}{c}\xymatrix{\boite{b_2}\ar@{-}[rd]|{a_2}&\boite{b_3}\ar@{-}[d]|{a_3}\\&\boite{b_1 +\epsilon^{(i)}}\\&\boite{\blackdiamond} \ar@{-}[u]|{a_1}}
			\end{array}
			+
			\begin{array}{c}\xymatrix{\boite{b_2 +\epsilon^{(i)}}\ar@{-}[rd]|{a_2}&\boite{b_3}\ar@{-}[d]|{a_3}\\&\boite{b_1}\\&\boite{\blackdiamond} \ar@{-}[u]|{a_1}}
			\end{array}
			+
			\begin{array}{c}\xymatrix{\boite{b_2}\ar@{-}[rd]|{a_2}&\boite{b_3 +\epsilon^{(i)}}\ar@{-}[d]|{a_3}\\&\boite{b_1}\\&\boite{\blackdiamond} \ar@{-}[u]|{a_1}}
			\end{array}
			,\\
			\left\{
			\begin{array}{c}\xymatrix{\boite{b_2}\ar@{-}[rd]|{a_2}&\boite{b_3}\ar@{-}[d]|{a_3}\\&\boite{b_1}\\&\boite{\blackdiamond} \ar@{-}[u]|{a_1}}
			\end{array}
			,X_i\right\}&=
			\begin{cases}
				\begin{array}{c}
				\xymatrix{\boite{b_2}\ar@{-}[rd]|{a_2}&\boite{b_3}\ar@{-}[d]|{a_3}\\&\boite{b_1}\\&\boite{\blackdiamond} \ar@{-}[u]|{a_1 -\epsilon^{(i)}}}
				\end{array}
				\mbox{ if }a_1 \geq \epsilon^{(i)},\\
				0\mbox{ otherwise}.
			\end{cases}
		\end{align*}
	\end{example}

	\begin{remark}
		We define a~pairing between $\DEPDE\otimes \DVPDE$ and itself as follows:
		\begin{align*}
			&\forall a,b,a',b'\in \N^{d+1},&\langle a'\otimes b',a\otimes b\rangle&=\delta_{a,a'}\delta_{b,b'}.
		\end{align*}
		It is possible to transpose $\partial^{(j)}$: define
		\[
			d^{(j)}:\left\{
			\begin{array}{rcl}
				\DEPDE\otimes \DVPDE&\longrightarrow&\DEPDE\otimes \DVPDE\\
				a\otimes b&\longmapsto&(b_j +1)\left(a+\epsilon^{(j)}\right)\otimes \left(b+\epsilon^{(j)}\right).
			\end{array}
			\right.
		\]
		Then, for any $x,x'\in \DEPDE\otimes \DVPDE$,
		\[
			\langle \partial^{(j)}(x'),x\rangle=\langle x,d^{(j)}(x)\rangle.
		\]
		Note that $d^{(j)}$ is not locally nilpotent. It is not possible to transpose $\phi^\lambda$: such a~transpose should satisfy
		\begin{align*}
			&\forall a,b\in \N^{d+1},& \left(\phi^\lambda\right)^* (a\otimes b)&=\sum_{l\in \N^{d+1}} \lambda^l \binom{b+l}{l}(a+l)\otimes (b+l),
		\end{align*}
		which does not belong to $\DEPDE\otimes \DVPDE$ but rather to a~completion of this vector space. The associated Butcher-Connes-Kreimer would require infinite sums of tensors. Such objects are considered in~\cite{Bruned2019}.
	\end{remark}

	\subsection{An example for stochastic PDEs, with noise}

	In order to treat noises, we augment the spaces of decorations $\DEPDE$ and $\DVPDE$ by adding two symbols:
	\begin{align*}
		\overline{\DEPDE}&=\vect(\N^{d+1}\sqcup \{\Xi\}),&\overline{\DVPDE}&=\vect(\N^{d+1}\sqcup \{\star\}).
	\end{align*}
	We extend $\phi^\lambda$ as follows:
	\begin{align*}
		&\forall a,b\in \N^{d+1},&\overline{\phi}^\lambda(a\otimes \star)&=0,&\overline{\phi}^\lambda(\Xi\otimes b)&=\Xi \otimes b,&\overline{\phi}^\lambda(\Xi\otimes \star)&=0.
	\end{align*}

	\begin{proposition}\label{prop4.8}
		This extension $\overline{\phi}^\lambda$ of $\phi^\lambda$ is tree-compatible.
	\end{proposition}

	\begin{proof}
		We put $D'_E =\vect(\Xi)$ and $D'_V =\vect(\star)$, and we equipped $D'_E\otimes D'_V$ with the zero map $\phi'$, which is obviously tree-compatible. Then, with the notations of Proposition~\ref{prop2.1}, $\overline{\phi}^\lambda=\phi^\lambda \oplus_{0,1}\phi'$.
	\end{proof}

	\begin{proposition}\label{prop4.9}
		For any $i\in\llbracket 0;d\rrbracket$, we extend $\psi_V (X_i )$ to $\overline{\DVPDE}$ and $\psi_E (X_i )$ to $\overline{\DEPDE}$ by
		\begin{align*}
			\overline{\psi}_V (X_i )(\star)&=0,&\overline{\psi}_E (X_i )(\Xi)&=0.
		\end{align*}
		The pair $(\overline{\psi}_E,\overline{\psi}_V )$ is $\overline{\phi}^\lambda$-compatible, in the sense of Corollary~\ref{cor3.10}.
	\end{proposition}

	\begin{proof}
		The first two conditions are obvious. Let us prove the last one. Let us fix $i\in \llbracket 0;d\rrbracket$. It is enough to prove that
		\[
			\overline{\phi^\lambda}\circ (\overline{\psi}_E (X_i )\otimes \id_{\overline{\DVPDE}})(x\otimes y)=\overline{\phi^\lambda}\circ(\id_{\overline{\DEPDE}}\otimes \overline{\psi}_V (X_i ))(x\otimes y)-(\id_{\overline{\DEPDE}}\otimes \overline{\psi}_V (X_i ))\circ \overline{\phi^\lambda}(x\otimes y)
		\]
		in the following cases:
		\begin{itemize}
			\item $y=\star$. Then every term is zero.
			\item $x=\Xi$ and $y\in \N^{d+1}$. Then every term is zero. \qedhere
		\end{itemize}
	\end{proof}

	\begin{example}
		If $a_1,a_2,b_1,b_2\in \N^{d+1}$ and $i\in \llbracket 0;d\rrbracket$,
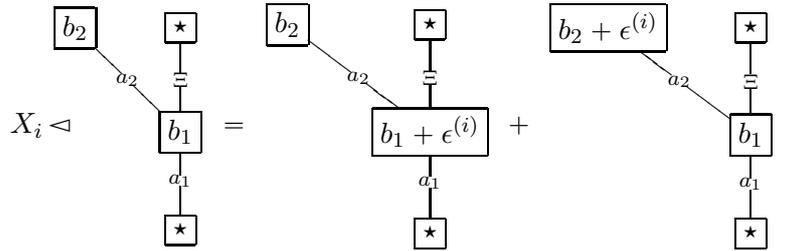
		\begin{align*}
			X_i\triangleleft
			\begin{array}{c}\xymatrix{\boite{b_2}\ar@{-}[rd]|{a_2}&\boite{\star}\ar@{-}[d]|{\Xi}\\&\boite{b_1}\\&\boite{\blackdiamond} \ar@{-}[u]|{a_1}}
			\end{array}
			&=
			\begin{array}{c}\xymatrix{\boite{b_2}\ar@{-}[rd]|{a_2}&\boite{\star}\ar@{-}[d]|{\Xi}\\&\boite{b_1 +\epsilon^{(i)}}\\&\boite{\blackdiamond} \ar@{-}[u]|{a_1}}
			\end{array}
			+
			\begin{array}{c}\xymatrix{\boite{b_2 +\epsilon^{(i)}}\ar@{-}[rd]|{a_2}&\boite{\star}\ar@{-}[d]|{\Xi}\\&\boite{b_1}\\&\boite{\blackdiamond} \ar@{-}[u]|{a_1}}
			\end{array}
			.
		\end{align*}
	\end{example}

	\begin{remark}
		Let $x=T\otimes W_T$ and $x'=T'\otimes W_{T'} \in \calT(\overline{\DEPDE},\overline{\DVPDE})$ be two rooted trees such that any vertex is decorated by an element of $\N^{d+1}$ or by $\star$, and any edge is decorated by an element of $\N^{d+1}$ or by $\Xi$. As $\overline{\phi^\lambda}(x\otimes \star)=0$ for any $x\in \overline{\DEPDE}$, we can rewrite the definition of the products in this way: for any $a\in \N^{d+1}$,
		\begin{align*}
			x\triangleleft^{\overline{\phi^\lambda}}_a x'&=\sum_{\substack{v\in V(T'),\\b_v\neq \star}} T\curvearrowright_v T'\otimes \overline{\phi^\lambda}_{0,v}\left(a\otimes W_T\otimes W_{T'}\right).
		\end{align*}
		For example, if $a,a_1,a_2,b_1,b_2,b_3,b_4\in \vect\left(\N^{d+1}\right)$,
		\begin{align*}
			\begin{array}{c}\xymatrix{\boite{b_2}\ar@{-}[d]|{a_1}\\ \boite{b_1}}
			\end{array}
			\triangleleft_a^{\overline{\phi^\lambda}}
			\begin{array}{c}\xymatrix{\boite{b_4}\ar@{-}[rd]|{a_2}&\boite{\star}\ar@{-}[d]|{\Xi}\\&\boite{b_3}}
			\end{array}
			&=\sum_{l\leq \min(a,b_4 )}\lambda^l \binom{b_4}{l}
			\begin{array}{c}\xymatrix{\boite{b_2}\ar@{-}[d]|{a_1}&\\ \boite{b_1}\ar@{-}[d]|{a-l}&\\\boite{b_4 -l}\ar@{-}[rd]|{a_2}&\boite{\star}\ar@{-}[d]|{\Xi}\\&\boite{b_3}}
			\end{array}
			\\
			&+\sum_{l\leq \min(a,b_3 )}\lambda^l \binom{b_3}{l}
			\begin{array}{c}\xymatrix{&&\boite{b_2}\ar@{-}[d]|{a_1}\\\boite{b_4}\ar@{-}[rd]|{a_2}&\boite{\star}\ar@{-}[d]|{\Xi}&\boite{b_1}\ar@{-}[ld]|{a-l}\\&\boite{b_3 -l}&}
			\end{array}
			.
		\end{align*}
	\end{remark}

	\begin{proposition}\label{prop4.10}
		Let us denote by $\g$ the pre-Lie subalgebra of $\overline{\DEPDE}\otimes \calT(\overline{\DEPDE},\overline{\DVPDE})$ generated by the elements
		\begin{align*}
			\Xi\otimes\xymatrix{\boite{\star}}&=
			\begin{array}{c}\xymatrix{\boite{\star}\ar@{-}[d]|{\Xi}\\ \boite{\blackdiamond}}
			\end{array}
			,& a\otimes\xymatrix{\boite{\star}}&=
			\begin{array}{c}\xymatrix{\boite{\star}\ar@{-}[d]|{a}\\ \boite{\blackdiamond}}
			\end{array}
			,&a\otimes\xymatrix{\boite{b}}&=
			\begin{array}{c}\xymatrix{\boite{b}\ar@{-}[d]|{a}\\ \boite{\blackdiamond}}
			\end{array}
			,
		\end{align*}
		where $a,b\in \N^{d+1}$. As a~vector space, it is generated by the set of tensors $a\otimes T\otimes W_T$, where $a\in \N^{d+1}\sqcup \{\Xi\}$ and $T\otimes W_T$ is a~rooted whose vertices are decorated by elements of $\N^{d+1}\sqcup \{\star\}$ and edges by elements of $\N^{d+1}\sqcup \{\Xi\}$, with the following conditions:
		\begin{itemize}
			\item If $a=\Xi$, then $T=\xymatrix{\boite{\star}}$.
			\item If $e\in E(T)$ is decorated by $\Xi$, then $d_{t(e)}=\star$.
			\item If $v\in V(T)$ is decorated by $\star$, then it is a~leaf of $T$.
		\end{itemize}
		These elements will be called $\Xi$-admissible.
	\end{proposition}

	\begin{proof}
		Let us observe firstly that the pre-Lie product is given by
		\begin{align*}
			a\otimes x\triangleleft a'\otimes x'=\sum_{\substack{v\in V(T'),\\b_v\neq \star}} a'\otimes T\curvearrowright_v T'\otimes \overline{\phi^\lambda}_{0,v}\left(a\otimes W_T\otimes W_{T'}\right).
		\end{align*}
		As $\overline{\phi^\lambda}(\Xi\otimes b)=\Xi\otimes b$ for any $b\in \N^{d+1}$, we obtain that if $a\otimes x$ and $a'\otimes x'$ are $\Xi$-admissible, then any term in $a\otimes x\triangleleft a'\otimes x'$ is also $\Xi$-admissible:
		therefore, these conditions define a~pre-Lie subalgebra. As it contains the generators of $\g$, it contains $\g$.

		Let $a\otimes x$ be $\Xi$-admissible, and let us show that it belongs to $\g$. If $a=\Xi$, then $a\otimes x=\Xi \otimes \star \in \g$. Let us assume that $a\in \N^{d+1}$. We proceed by the number $n$ of vertices of the underlying tree. If $n=1$, then $a\otimes x$ is a~generator of $\g$. Let us put $T=B^+ (T_1\ldots T_k )$ and let us proceed by induction on $k$. If $k=0$, then $a\otimes x$ is a~generator of $\g$. Otherwise, let us put $x'$ the decorated rooted tree corresponding to $T_1$ and $x''$ the rooted tree corresponding to $B^+ (T_2\ldots T_k )$. If $c$ is the decoration of the root of $T$ (note that $b\in \N^{d+1}$), and $b$ the decoration of the edge between the root of $T$ and the root of $T_1$, let us consider
		\[
			\left(\phi^\lambda\right)^{-1}(b\otimes c)=\phi^{-\lambda}(b\otimes c)=\sum_i \mu_i b_i\otimes c_i.
		\]
		Let us denote by $x''_i$ the element obtained from $x''$ be changing the decoration of the root into $c_i$, instead of $c$. Then
		\begin{align*}
			\sum_i { }&{ }\mu_i b_i\otimes x'\triangleleft a\otimes x''_i \\
            &=a\otimes x+\mbox{a sum of terms of $\Xi$-admissible trees with $n$ vertices and $k'=k-1$}.
		\end{align*}
		By the induction hypothesis on $n$, $b_i\otimes x'$ and $a\otimes x''_i$ belong to $\g$. By the induction hypothesis on $k$, we deduce that $a\otimes x\in \g$.
	\end{proof}

	\subsection{Examples in small dimensions}

	\begin{proposition}
		Let $f$ be an endomorphism of $D_E$ and $g$ be an endomorphism of $D_V$. Then $f\otimes g$ is tree-compatible.
	\end{proposition}

	\begin{proof}
		Indeed, if $\phi=f\otimes g$, $\phi_{13}=f\otimes \id_{D_E}\otimes g$, and $\phi_{23}=\id_{D_E}\otimes f\otimes g$. So $\phi_{13}$ and $\phi_{23}$ commute.
	\end{proof}

	\begin{corollary}
		Let us assume that $D_V$ or $D_E$ is 1-dimensional. Then any endomorphism of $D_E\otimes D_V$ is tree-compatible.
	\end{corollary}

	\begin{proof}
		If $D_E$ is one-dimensional, then any endomorphism of $D_E\otimes D_V$ is of the form $f\otimes \id_{D_V}$, where $f$ is an endomorphism of $D_E$. The proof is similar if $D_V$ is one-dimensional.
	\end{proof}

	When $D_E$ and $D_V$ are finite-dimensional, we can work with matrices:

	\begin{proposition}\label{prop4.13}
		Let us assume that $D_E$ and $D_V$ are finite-dimensional, with respective dimensions $m$ and $n$, and bases $\calB_E =(a_1,\ldots,a_m )$ and $\calB_V =(b_1,\ldots,b_n )$. The basis $\calB_E\otimes \calB_V$ of $D_E\otimes D_V$ is
		\[
			\calB_E\otimes \calB_V =(a_1\otimes b_1,\ldots,a_1\otimes b_n,\ldots,a_m\otimes b_1,\ldots,a_m\otimes b_n ).
		\]
		Let $\phi$ be an endomorphism of $D_E\otimes D_V$. Its matrix in the basis $\calB_E\otimes \calB_V$ is written under the form
		\[
			\begin{pmatrix}
				A_{11}&\ldots&A_{1m}\\
				\vdots&&\vdots\\
				A_{m1}&\ldots&A_{mm}
			\end{pmatrix}
			,
		\]
		with for any $i,j\in \llbracket 1;m\rrbracket$, $A_{i,j}\in M_n (\K)$. Then:
		\begin{enumerate}
			\item $\phi$ is tree-compatible if, and only if, for any $i,j,k,l\in \llbracket 1;m\rrbracket$, $A_{ij}A_{kl}=A_{kl}A_{ij}$.
			\item $\phi$ is the tensor product of an endomorphism of $D_E$ and an endomorphism of $D_V$ if, and only if,
			\[
				\rk(A_{11},\ldots,A_{mn})\leq 1.
			\]
		\end{enumerate}
	\end{proposition}

	\begin{proof}
		1. In order to simplify the writing, we shall identify any $b\in D_V$ with its coordinates vector in $\K^n$. Let $j\in \llbracket 1;m\rrbracket$ and $b\in D_V$. Then
		\[
			\phi(a_j\otimes b)=\sum_{i=1}^m a_i\otimes A_{ij}b,
		\]
		so, for any $i,j\in \llbracket 1;m\rrbracket$ and $b\in D_V$,
		\begin{align*}
			\phi_{23}\circ \phi_{13}(a_i\otimes a_j\otimes b)&=\sum_{k,l=1}^m a_k\otimes a_l\otimes A_{lj}A_{ki}b,\\
			\phi_{13}\circ \phi_{23}(a_i\otimes a_j\otimes b)&=\sum_{k,l=1}^m a_k\otimes a_l\otimes A_{ki}A_{lj}b.
		\end{align*}
		Therefore,
		\begin{align*}
			&\mbox{$\phi$ is tree-compatible}\\
			&\Longleftrightarrow \forall i,j\in \llbracket 1;m\rrbracket,\:\forall b\in D_V,\:\sum_{k,l=1}^m a_k\otimes a_l\otimes A_{lj}A_{ki}b=\sum_{k,l=1}^m a_k\otimes a_l\otimes A_{ki}A_{lj}b\\
			&\Longleftrightarrow \forall i,j,k,l\in \llbracket 1;m\rrbracket,\:\forall b\in D_V,\:A_{lj}A_{ki}b=A_{ki}A_{lj}b\\
			&\Longleftrightarrow \forall i,j,k,l\in \llbracket 1;m\rrbracket,\:A_{lj}A_{ki}=A_{ki}A_{lj}.
		\end{align*}
		2. $\Longrightarrow$. Let us assume that $\phi=\phi_E\otimes \phi_V$. Let us denote by $B=(b_{ij})_{1\leq i,j\leq m}$ the matrix of $\phi_E$ in the basis $\calB_E$ and by $A$ the matrix of $\phi_V$ in the basis $\calB_V$. Then, for any $i,j\in \llbracket 1;m\rrbracket$, $A_{ij}=b_{ij}A$, so $\rk(A_{11},\ldots,A_{mm})\leq 1$.

		$\Longleftarrow$. If so, there exist matrices $A\in M_n (\K)$ and $B=(b_{ij})_{1\leq i,j\leq m}\in M_m (\K)$ such that for any $i,j\in \llbracket 1;m\rrbracket$, $A_{ij}=b_{ij}A$. Let $\phi_E$ be the endomorphism of $D_E$ whose matrix in the basis $D_E$ is $B$ and $\phi_V$ the endomorphism of $D_V$ whose matrix in the basis $D_V$ is $A$. Then the matrix of $\phi_E\otimes \phi_V$ in the basis $\calB_E\otimes \calB_V$ is
		\[
			\begin{pmatrix}
				b_{11}A&\ldots&b_{1m}A\\
				\vdots&&\vdots\\
				b_{m1}A&\ldots&b_{mm}A
			\end{pmatrix}
			=
			\begin{pmatrix}
				A_{11}&\ldots&A_{1m}\\
				\vdots&&\vdots\\
				A_{m1}&\ldots&A_{mm}
			\end{pmatrix}
			,
		\]
		so $\phi=\phi_E\otimes \phi_V$.
	\end{proof}

	\begin{remark}
		The matrix of $\phi^*$ in the dual basis $(a_1^*\otimes b_1^*,\ldots,a_m^*\otimes b_n^* )$ is
		\[
			\begin{pmatrix}
				A_{11}&\ldots&A_{1,m}\\
				\vdots&&\vdots\\
				A_{m1}&\ldots&A_{mn}
			\end{pmatrix}
			^\top=
			\begin{pmatrix}
				A_{11}^\top&\ldots&A_{m1}^\top\\
				\vdots&&\vdots\\
				A_{1m}^\top&\ldots&A_{mm}^\top
			\end{pmatrix}
			.
		\]
	\end{remark}

	\begin{notation}
		For any $(a,b)\in \K^2$, we consider the $2\times 2$ matrices
		\begin{align*}
			J(a,b)&=
			\begin{pmatrix}
				a&b\\
				0&a
			\end{pmatrix}
			,& D(a,b)&=
			\begin{pmatrix}
				a&0\\
				0&b
			\end{pmatrix}
			.
		\end{align*}
	\end{notation}

	\begin{theorem}\label{theo4.14}
		Let $D_E$ be finite-dimensional $\C$-vector space of dimension $m$ and $D_V$ be a~2-dimensional $\C$-vector space, and let $\phi$ be an endomorphism of $D_E\otimes D_V$. Then $\phi$ is tree-compatible if, and only if, there exists a~basis $\calB_E$ of $D_E$ and $\calB_V$ of $D_V$ such that the matrix of $\phi$ in the basis $\calB_E\otimes \calB_V$ has one of the following forms:
		\begin{align}
			\nonumber
			J(A,B)&=
			\begin{pmatrix}
				J(a_{11},b_{11})&\ldots&J(a_{1m},b_{1m})\\
				\vdots&&\vdots\\
				J(a_{m1},b_{m1})&\ldots&J(a_{mm},b_{mm})
			\end{pmatrix}
			\\
			\label{EQ11}\mbox{or } D(A,B)&=
			\begin{pmatrix}
				D(a_{11},b_{11})&\ldots&D(a_{1m},b_{1m})\\
				\vdots&&\vdots\\
				D(a_{m1},b_{m1})&\ldots&D(a_{mm},b_{mm})
			\end{pmatrix}
			,
		\end{align}
		with $A=(a_{ij})_{1\leq i,j\leq m}$ and $B=(b_{ij})_{1\leq i,j\leq m}$ in $M_m (\C)$.
	\end{theorem}

	\begin{proof}
		For any $A\in M_2 (\C)$, we put
		\[
			\calC(A)=\{B\in M_2 (\C)\mid AB=BA\}.
		\]
		Let $a,b\in \mathbb{K}$. We leave as an exercise to prove that
		\[
			\calC(D(a,b))=
			\begin{cases}
				\{D(c,d)\mid (c,d)\in \C^2\}\mbox{ if }a\neq b,\\
				M_2 (\C) \mbox{ if }a=b,
			\end{cases}
        \]
        and
        \[
			\calC(J(a,b))=
			\begin{cases}
				\{J(c,d)\mid (c,d)\in \C^2\}\mbox{ if }b\neq 0,\\
				M_2 (\C) \mbox{ if }b=0.
			\end{cases}
        \]
		Therefore, Proposition~\ref{prop4.13} implies that if the matrix of $\phi$ in a~basis $\calB_E\otimes \calB_V$ has one of the form of \eqref{EQ11}, then $\phi$ is tree-compatible.

		Let us now assume that $\phi$ is tree-compatible. Let us choose arbitrary bases $\calB_E$ and $\calB_V$ of, respectively, $D_E$ and $D_V$. The matrix of $\phi$ in $\calB_E\otimes \calB_V$ is denoted by
		\[
			M=
			\begin{pmatrix}
				M_{11}&\ldots&M_{m1}\\
				\vdots&&\vdots\\
				M_{m1}&\ldots&M_{mm}
			\end{pmatrix}
			,
		\]
		with for any $i,j\in \llbracket 1;m\rrbracket$, $M_{ij}\in M_2 (\C)$. If all these $2\times 2$ matrices are multiples of the identity matrix, then
		\[
			M=
			\begin{pmatrix}
				D(a_{11},a_{11})&\ldots&D(a_{1m},a_{1m})\\
				\vdots&&\vdots\\
				D(a_{m1},a_{m1})&\ldots&D(a_{mm},a_{mm})
			\end{pmatrix}
			=
			\begin{pmatrix}
				J(a_{11},0)&\ldots&J(a_{1m},0)\\
				\vdots&&\vdots\\
				J(a_{m1},0)&\ldots&J(a_{mm},0)
			\end{pmatrix}
			,
		\]
		with $(a_{ij})_{1\leq i,j\leq m}\in M_m (\C)$. Otherwise, let us fix $i_0,j_0\in \llbracket 1;m\rrbracket$ such that $M_{i_0 j_0}$ is not a~multiple of the identity matrix. As we work over $\C$, there exists a~matrix $P\in \mathrm{GL}_2 (\C)$ such that
		$P^{-1}M_{i_0 j_0}P=J(a_{i_0 j_0},b_{i_0 j_0})$ with $b_{i_0 j_0}\neq 0$ or $D(a_{i_0 j_0},b_{i_0 j_0})$ with $a_{i_0 j_0}\neq b_{i_0 j_0}$. Changing the basis of $D_V$ according to $P$, $M$ is replaced by
		\[
			\begin{pmatrix}
				P^{-1}M_{11}P&\ldots&P^{-1}M_{m1}P\\
				\vdots&&\vdots\\
				P^{-1}M_{m1}P&\ldots&P^{-1}M_{mm}P
			\end{pmatrix}
			.
		\]
		Consequently, we can assume that $M_{i_0 j_0}=J(a_{i_0 j_0},b_{i_0 j_0})$ or $D(a_{i_0 j_0},b_{i_0 j_0})$. As $\phi$ is tree-compatible, for any $i,j\in \llbracket 1;m\rrbracket$, $M_{ij}\in \calC(M_{i_0 j_0})$. If $M_{i_0 j_0}=J(a_{i_0 j_0},b_{i_0 j_0})$, then for any $i,j\in \llbracket 1;m\rrbracket$, $M_{ij}$ is of the form $J(a_{ij},b_{ij})$, as $b_{i_0 j_0}\neq 0$;
		if $M_{i_0 j_0}=D(a_{i_0 j_0},b_{i_0 j_0})$, then for any $i,j\in \llbracket 1;m\rrbracket$, $M_{ij}$ is of the form $D(a_{ij},b_{ij})$, as $a_{i_0 j_0}\neq b_{i_0 j_0}$.
	\end{proof}

	\begin{remark}
		Of course, this form is not unique. It is not difficult to prove that:
		\begin{enumerate}
			\item If $\phi$ has a~matrix of the form $D(A,B)$, then the other possible forms for matrices of $\phi$ are $D(P^{-1}AP,P^{-1}BP)$ and $D(P^{-1}BP,P^{-1}AP)$, with $P\in \mathrm{GL}_m (\C)$.
			\item If $\phi$ has a~matrix of the form $J(A,B)$, then the other possible forms for matrices of $\phi$ are $J(P^{-1}AP,\alpha P^{-1}BP)$, with $P\in \mathrm{GL}_m (\C)$ and $\alpha\in \C\setminus\{0\}$.
		\end{enumerate}
		In both cases, $P$ corresponds to a~change of basis for $D_E$; in the first case, the permutation of $A$ and $B$ corresponds to a~change of basis of $D_V$; in the second case, the scalar $\alpha$ corresponds to a~change of basis of $D_V$.
	\end{remark}

	\begin{remark}
		If the matrix of $\phi$ in a~convenient basis is $J(A,B)$, respectively $D(A,B)$, then in a~convenient basis, the matrix of $\phi^*$ is $J(A^\top,B^\top)$, respectively $D(A^\top,B^\top)$.
	\end{remark}

	We end this section by an observation of the invertibility of $J(A,B)$ and $D(A,B)$:

	\begin{proposition}
		Let $A,B\in M_m (\K)$. Then:
		\begin{enumerate}
			\item $D(A,B)$ is invertible if, and only if, $A$ and $B$ are invertible.
			\item $J(A,B)$ is invertible if, and only if, $A$ is invertible.
		\end{enumerate}
	\end{proposition}

	\begin{proof}
		As the blocks $M_{ij}$ of $M=D(A,B)$ or $J(A,B)$ pairwise commute, we can apply the result on determinants by blocks of~\cite{Bourbaki} and we obtain that
		\begin{align*}
			\det(M)&=\det\left(\sum_{\sigma \in \sym_m} \varepsilon(\sigma) \prod_{j=1}^m M_{\sigma(j)j}\right).
		\end{align*}
		If $M=D(A,B)$, then $M_{ij}=D(a_{ij},b_{ij})$ and
		\begin{align*}
			\det(M)
            &=\det
			\begin{pmatrix}
				\displaystyle \sum_{\sigma\in \sym_m} \varepsilon(\sigma) \prod_{j=1}^m a_{\sigma(j)j}&0\\
				0&\displaystyle \sum_{\sigma\in \sym_m} \varepsilon(\sigma) \prod_{j=1}^m b_{\sigma(j)j}
			\end{pmatrix} \\
			&=
			\begin{pmatrix}
				\det(A)&0\\
				0&\det(B)
			\end{pmatrix} \\
			&=\det(A)\det(B).
		\end{align*}
		If $M=J(A,B)$, then $M_{ij}=J(a_{ij},b_{ij})$ and
		\begin{align*}
			\det(M)
            &=\det
			\begin{pmatrix}
				\displaystyle \sum_{\sigma\in \sym_m} \varepsilon(\sigma) \prod_{j=1}^m a_{\sigma(j)j}&?\\
				0&\displaystyle \sum_{\sigma\in \sym_m} \varepsilon(\sigma) \prod_{j=1}^m a_{\sigma(j)j}
			\end{pmatrix} \\
			&=
			\begin{pmatrix}
				\det(A)&?\\
				0&\det(A)
			\end{pmatrix} \\
			&=\det(A)^2. \qedhere
		\end{align*}
	\end{proof}

	\section{Appendix: freeness of $D_E\otimes \calT(D_E,D_V )$}

	We here prove that the pre-Lie algebra $(D_E\otimes \calT(D_E,D_V ))$ (without deformation by a~tree-compatible map $\phi$) is free. We use for this Livernet's rigidity theorem~\cite{Livernet2006}.

	\begin{proposition}
		The pre-Lie algebra $(D_E\otimes \calT(D_E,D_V ),\triangleleft)$ is freely generated by the elements $a\otimes \xymatrix{\boite{b}}$, with $a\in D_E$ and $b\in D_V$.
	\end{proposition}

	\begin{proof}
		We define a~non-associative, permutative coproduct (shortly, NAP) coproduct $\rho$ by the following:
		if $a\in D_E$ and $x=T\otimes W_T\in \calT(D_E,D_v )$,
		\[
			\rho(a\otimes x)=\sum_{e\in E(T),\: s(e)=\rt(T)}\left(d_e\otimes P^e (T)\otimes W_{P^e (T)}\right)\otimes \left(a\otimes R^e (T)\otimes W_{R_e}(T)\right),
		\]
		where:
		\begin{itemize}
			\item $\rt(T)$ is the root of $T$.
			\item For any $e\in E(T)$, $P^e (T)$ is the subtree of $T$ formed by the vertices $v \in V(T)$ with a~path from $t(e)$ to $v$, and $R^e (T)$ is the subtree of $T$ formed by the other vertices.
		\end{itemize}
		Note that, up to the order of the factors,
		\[
			W_T =d_E\otimes W_{P^e (T)}\otimes W_{R^e (T)},
		\]
		as $V(T)$ is the disjoint union of $V(P^e (T))$ and $V(R^e (T))$ and $E(T)$ is the disjoint union of $E(P^e (T))$ and $E(R^e (T))$, minus $\{e\}$. Then
		\begin{align*}
			&(\id\otimes \rho)\circ \rho(a\otimes x)\\
			&=\hspace*{-2ex}\sum_{\substack{e\neq e'\in E(T),\\ s(e)=s(e')=\rt(T)}}\hspace*{-3.5ex}
            (d_e\otimes P^e (T)\otimes W_{P^e (T)})\otimes (d_{e'}\otimes P^{e'}(T)\otimes W_{P^{e'}(T)})\otimes (a\otimes R^{e,e'}(T)\otimes W_{R^{e,e'}}(T)),
		\end{align*}
		which is invariant by permutation of the first two factors: $\rho$ is indeed a~NAP coproduct. A study of the edges of $T\curvearrowright_v T'$ proves that for any $a,a'\in D_E$, for any $x,x'\in \calT(D_E,D_V )$,
		\begin{align*}
			\rho((a\otimes x)\triangleleft (a'\otimes x')&=(a\otimes x)\otimes (a'\otimes x')+\sum \left((a\otimes x)\triangleleft(a'\otimes x')^{(1)}\right)\otimes (a'\otimes x')^{(2)}\\
			&+\sum (a'\otimes x')^{(1)}\otimes \left((a\otimes x)\triangleleft (a'\otimes x')^{(2)}\right),
		\end{align*}
		with Sweedler's notation $\rho(a\otimes x)=\sum (a\otimes x)^{(1)}\otimes (a\otimes x)^{(2)}$. By Livernet's rigidity theorem~\cite{Livernet2006}, $D_E\otimes \calT(D_E,D_V )$ is freely generated by $\Ker(\rho)$.

		It remains to compute $\Ker(\rho)$. We define a~new product $\blacktriangleleft$ on $D_E\otimes \calT(D_E,D_V )$ by
		\[
			a\otimes T\otimes W_T\blacktriangleleft a'\otimes T'\otimes W_{T'}=a'\otimes T\curvearrowright_{\rt(T')} T'\otimes a\otimes W_T\otimes W_{T'}.
		\]
		Then, for any $a\otimes T\otimes W_T\in D_E\otimes \calT(D_E,D_V )$,
		\[
			\blacktriangleleft \circ \rho(T\otimes W_T )=\alpha_T T\otimes W_T,
		\]
		where $\alpha_T$ is the number of edges $e\in E(T)$ such that $s(e)=\rt(T)$. Therefore, if an element $x=\sum_i a_i\otimes T_i\otimes W_{T_i}\in \Ker(\rho)$, then
		\[
			\blacktriangleleft\circ \rho(x)=0=\sum_i \alpha_{T_i} a_i \otimes T_i.
		\]
		Therefore, $x$ is a~sum of terms $a\otimes \xymatrix{\boite{b}}$. The converse is immediate.
	\end{proof}

	If $\phi$ is bijective, as $\Theta_\phi$ is an isomorphism from $(\calT(D,E,D_V ),\triangleleft)$ to $(\calT(D,E,D_V ),\triangleleft^\phi)$:

	\begin{corollary}
		If $\phi$ is an invertible tree-compatible map, then $(D_E\otimes \calT(D_E,D_V ),\triangleleft^\phi)$ is, as a~pre-Lie algebra, freely generated by the elements $a\otimes \xymatrix{\boite{b}}$, with $a\in D_E$ and $b\in D_V$.
	\end{corollary}

	\begin{remark}
		If $\phi$ is not injective, then this pre-Lie algebra is not freely generated by these elements: indeed, if $\sum_i a_i\otimes b_i$ is a~nonzero element of $\Ker(\phi)$, then if $a'\in D_E$ and $b'\in D_V$, both nonzero,
		\[
			\sum_i a_i\otimes \xymatrix{\boite{b}}\triangleleft a\otimes \xymatrix{\boite{b_i}}=
			a\otimes\left(\sum_j \sum_i
			\begin{array}{c}
				\xymatrix{\boite{b}\ar@{-}[d]|{\phi_E^j (a_i )}\\ \boite{\phi_V^j (b_i )}}
			\end{array}
			\right)=0.
		\]
	\end{remark}

	%%% REFERENCES %%%
    
\end{document}